\documentclass[bj,noinfoline]{imsart}
\usepackage{amsthm,amsmath,amssymb,amsfonts,delarray,array}
\usepackage{url}
\usepackage{graphicx}
\usepackage{color}
\usepackage{comment}
\usepackage{xr}
\newtheorem{theorem}{Theorem}[section]

\newtheorem{lemma}{Lemma}[section]
\newtheorem{proposition}{Proposition}[section]

\newtheorem{condition}{Condition}[section]
\theoremstyle{definition}

\newtheorem{remark}{Remark}[section]

\newcommand{\R}{\mathbb{R}}

\newcommand{\Ep}{\mathbb{E}}
\renewcommand{\Pr}{\mathbb{P}}
\renewcommand{\tilde}{\widetilde}
\newcommand{\indep}{\mathop{\perp\!\!\!\!\perp}}

\renewcommand{\leq}{\leqslant}
\renewcommand{\geq}{\geqslant}

\begin{document}

\begin{frontmatter}

\title{Minimax Predictive Density\\ for Sparse Count Data}
\runtitle{Predictive density for sparse count data}

\begin{aug}
  \author{\fnms{Keisuke}  \snm{Yano}$^{1}$\corref{}\ead[label=e1]{yano@ism.ac.jp}},
  \author{\fnms{Ryoya} \snm{Kaneko}}$^{2}$\ead[label=e2]{ryykaneko@gmail.com}
  \and
  \author{\fnms{Fumiyasu} \snm{Komaki}}$^{3,4}$\ead[label=e3]{komaki@g.ecc.u-tokyo.ac.jp}

  \runauthor{K.~Yano, R.~Kaneko and F.~Komaki}

\affiliation{The Institute of Statistical Mathematics, The University of Tokyo, and RIKEN Center for Brain Science}
\address{
$^{1}$The Institute of Statistical Mathematics,\\
10-3 Midori cho, Tachikawa City, Tokyo, 190-8562, Japan\\
\printead{e1}}
\address{
$^{2}$Tokyo Marine Holdings, Inc.,\\
1-2-1 Marunouchi, Chiyoda-ku, Tokyo, 100-8050, Japan\\
\printead{e2}}
\address{
$^{3}$Department of Mathematical Informatics,\\
Graduate School of Information Science and Technology,\\
The University of Tokyo,\\
7-3-1 Hongo, Bunkyo-ku, Tokyo, 113-0033,Japan\\
\printead{e3}}

\address{
$^{4}$RIKEN Center for Brain Science,\\
2-1 Hirosawa, Wako City, Saitama,
351-0198, Japan}

\end{aug}

\begin{abstract}
This paper discusses predictive densities under the Kullback--Leibler loss for high-dimensional Poisson sequence models under sparsity constraints.
Sparsity in count data implies zero-inflation.
We present a class of Bayes predictive densities that attain asymptotic minimaxity in sparse Poisson sequence models.
We also show that our class with an estimator of unknown sparsity level plugged-in is adaptive in the asymptotically minimax sense.
For application,
we extend our results to settings with quasi-sparsity and with missing-completely-at-random observations.
The simulation studies as well as application to real data
illustrate the efficiency of the proposed Bayes predictive densities.
\end{abstract}

\begin{keyword}[class=MSC]
\kwd[Primary ]{62M20}
\kwd[; secondary ]{62C20, 62G20, 62P99}
\end{keyword}
\begin{keyword}
\kwd{Adaptation}
\kwd{High dimension}
\kwd{Kullback--Leibler divergence}
\kwd{Missing At Random}
\kwd{Poisson model}
\kwd{Zero inflation}
\end{keyword}

\end{frontmatter}

\maketitle

\section{Introduction}
\label{section: introduction}

Predictive density is a probability density of future observations on the basis of current observations.
It is used not only to estimate future observations but also to quantify their uncertainty.
It has a wide range of application in statistics, information theory, and machine learning.
The simplest class of predictive densities is the class of \textit{plug-in} predictive densities.
A plug-in predictive density is constructed by substituting an estimator into an unknown parameter of a statistical model.
Another class of predictive densities is the class of \textit{Bayes} predictive densities.
A Bayes predictive density is the posterior mixture of densities of future observations.
There is a vast literature on predictive density for statistical models in finite dimensions; 
see Subsection \ref{subsection: literature} for the literature review.
Conversely, little is known about predictive density for statistical models in high dimensions.
In prediction using sparse high-dimensional Gaussian models,
\cite{MukherjeeandJohnstone(2015),MukherjeeandJohnstone(2017)} construct
several predictive densities
(including a Bayes predictive density) superior to all plug-in predictive densities.

The aim of this paper is to construct an efficient predictive density for high-dimensional sparse count data.
The efficiency of a predictive density is measured by the supremum of the Kullback--Leibler risk under sparsity constraints.
Sparsity in count data means that there exhibits an excess of zeros.
See Subsection \ref{subsection: Problem setting} for the formulation.

The motivation for analyzing sparse count data is well-known.
In analyzing high-dimensional count data,
there often exhibits inflation of zeros.
Data with an overabundance of zeros include
samples from
agriculture \cite{Hall(2002)},
environmental sciences \cite{Agarwaletal(2002)},
manufacturing \cite{Lambert(1992)},
DNA sequencing \cite{DattaandDunson(2016)},
and terrorist attacks \cite{DattaandDunson(2016)}.
Another example (Japanese crime statistics) is presented in Section \ref{section: numerical experiments}.

\subsection{Problem setting and contributions}
\label{subsection: Problem setting}

We summarize main results with the problem formulation ahead.
Let $X_{i}$ ($i=1,2,\ldots,n$) be a current observation independently distributed according to $\mathrm{Po}(r\theta_{i})$,
and 
let $Y_{i}$ ($i=1,2,\ldots,n$) be a future observation independently distributed according to $\mathrm{Po}(\theta_{i})$,
where $\theta=(\theta_{1},\ldots,\theta_{n})$ is an unknown parameter and $r$ is a known constant.
Constant $r$ represents the ratio of the mean of the $i$-th ($i=1,\ldots,n$) current observation 
to that of the $i$-th future observation.
By sufficiency, this constant represents the ratio of sample sizes of current observations to those of future observations.
Suppose that $X=(X_{1},\ldots,X_{n})$ and $Y=(Y_{1},\ldots,Y_{n})$ are independent.
The densities of $X$ and $Y$ with parameter $\theta$ are denoted by $p(x\mid \theta)$ and $q(y\mid\theta)$, respectively:
\begin{align*}
p(x \mid \theta) = \prod_{i=1}^{n} \left\{ \frac{1}{x_{i}!}  \mathrm{e}^{-r\theta_{i}}(r\theta_{i})^{x_{i}}\right\} \text{ and }
q(y \mid \theta) = \prod_{i=1}^{n} \left\{\frac{1}{y_{i}!} \mathrm{e}^{- \theta_{i}} \theta_{i}^{y_{i}}\right\}.
\end{align*}
Our target parameter space is the exact sparse parameter space which is defined as follows.
Given $s \in (0,n) $,
$\Theta[s]:=\{\theta\in \R_{+}^{n}: \|\theta\|_{0} \leq s \}$, 
where $\|\cdot\|_{0}$ is the $\ell_{0}$-norm given by $\|\theta\|_{0}:=\#\{i: \theta_{i}>0\}$.

The performance of a predictive density $\hat{q}$ is evaluated by the Kullback--Leibler loss
\[
L(\theta, \hat{q}(\cdot ; x ) ) =
\sum_{y\in \mathbb{N}^{n}} q(y \mid \theta) \log \frac{q(y\mid \theta)}{\hat{q} (y ; x)}.
\]
The corresponding risk (expected loss) is denoted by 
\[
R(\theta,\hat{q}) = \sum_{x\in\mathbb{N}^{n}} \sum_{y\in \mathbb{N}^{n}} p(x\mid \theta) q(y \mid \theta) \log \frac{q(y\mid \theta)}{\hat{q} (y ; x)}.
\]
The minimax Kullback--Leibler risk over $\Theta[s]$ is defined as
\[
\mathcal{R}(\Theta[s]):= \mathcal{R}_{n}(\Theta[s]) =  \mathop{\inf}_{\hat{q}}\sup_{\theta\in\Theta[s]} R(\theta, \hat{q}).
\]

To express high-dimensional settings under sparsity constraints,
we employ the high dimensional asymptotics in which $n\to\infty$ and $\eta_{n}:=s/n=s_{n}/n\to 0$.
The value of $s$ possibly depends on $n$ and thus in what follows the dependence on $n$ is often expressed,
say, $s=s_{n}$.

Main theoretical contributions are summarized as follows:
\begin{itemize}
\item[(i)] In Theorem \ref{theorem: exact minimaxity within sparse Poisson models},
we identify the asymptotic minimax risk $\mathcal{R}(\Theta[s_{n}])$ 
and present a class of Bayes predictive densities attaining the asymptotic minimaxity;
\item[(ii)] In Theorem \ref{theorem: adaptive},
we present an asymptotically minimax predictive density that is adaptive to an unknown sparsity.
\end{itemize}
In Theorem \ref{theorem: exact minimaxity within sparse Poisson models}, we find that the sharp constant in the asymptotic minimax risk is controlled by the constant $r$.
This constant highlights the interesting parallel between Gaussian and Poisson decision theories as discussed in Subsection \ref{subsection: literature}.
In Theorem \ref{theorem: adaptive}, we show that a simple plug-in approach to choose the tuning parameter in the proposed class yields adaptive Bayes predictive densities.
In addition, we obtain the corresponding results for quasi sparse Poisson models and for settings where current observations are missing completely at random
in Section \ref{section: extensions}.
These extensions are important in applications.

The practical effectiveness of the proposed Bayes predictive densities is examined by both simulation studies and applications to real data in Section \ref{section: numerical experiments}. These studies show that the proposed Bayes predictive densities are effective in the sense of both predictive uncertainty quantification and point prediction.

The proposed class of predictive densities builds upon spike-and-slab prior distributions with improper slab priors.
Interestingly, spike-and-slab prior distributions with slab priors having exponential tails do not yield
asymptotically minimax predictive densities
as Proposition \ref{proposition: suboptimality of spike and slab} indicates.
The proposed predictive densities are not only asymptotically minimax 
but also easily implemented by exact sampling.

\subsection{Literature review}
\label{subsection: literature}

There is a rich literature on constructing predictive densities in fixed finite dimensions.
Bayes predictive densities have been shown to dominate plug-in predictive densities in several instances.
Studies of Bayes predictive densities date back to \cite{Aitchison(1975),Murray(1977),Akaike(1978),Ng(1980)}.
The first quantitative comparison of Bayes and plug-in predictive densities in a wide class of parametric models is \cite{Komaki(1996)}.
\cite{Komaki(1996)} showed that there exists a Bayes predictive density that 
dominates a plug-in predictive density under the Kullback--Leibler loss,
employing asymptotic expansions of Bayes predictive densities;
see also \cite{Hartigan(1998)} for asymptotic expansions of Bayes predictive densities.
Minimax Bayes predictive densities for unconstrained parameter spaces are studied in \cite{LiangandBarron(2004),Aslan(2006)}.
Minimax predictive densities under parametric constraints are studied in \cite{Fourdrinieretal(2011),Kubokawaetal(2013),Mouddenetal(2017)}.
Shrinkage priors for Bayes predictive densities under Gaussian models are investigated in
\cite{Komaki(2001),GeorgeLiangandXu(2006),KobayashiandKomaki(2008),MatsudaandKomaki(2015)};
see also \cite{Kato(2009),BoisbunonandMaruyama(2014),Fourdrinieretal(2019)} for the cases where the variances are unknown.
Shrinkage priors for Bayes predictive densities under Poisson models are developed in \cite{Komaki(2004),Komaki(2015)}.
The cases under $\alpha$-divergence losses are covered by \cite{CorcueraandGiummole(1999),SuzukiandKomaki(2010),MaruyamaandStrawderman(2012),Zhangetal(2018),LmouddenandMarchand(2019)}.

Relatively little is known about constructing predictive densities in high dimensions.
\cite{MukherjeeandJohnstone(2015),MukherjeeandJohnstone(2017)} construct an asymptotically minimax predictive density for sparse Gaussian models.
\cite{XuandLiang(2010)} obtained an asymptotically minimax predictive density for nonparametric Gaussian regression models under Sobolev constraints;
thereafter, \cite{YanoandKomaki(2017)} obtained an adaptive minimax predictive density for these models.
See also \cite{XuandZhou(2011)}.
All above results employ Gaussian likelihood and the corresponding results for count data have been not known.

Poisson models deserve study in their own right as prototypical count data modeling (\cite{Robbins(1956),ClevensonandZidek(1975),JohnstoneandMacGibbon(1992),Komaki(2004),BrownGreenshteinRitov(2013)}).
Poisson models exhibit several correspondences to Gaussian models.
\cite{ClevensonandZidek(1975),Johnstone(1984),JohnstoneandMacGibbon(1992),MacGibbon(2010)} find
the correspondence in estimation of means
using the re-scaled squared loss
defined as $\sum_{i=1}^{n}\theta_{i}^{-1}(\theta_{i}-\hat{\theta}_{i}(X))^{2}$.
\cite{GhoshandYang(1988),Komaki(2004),Komaki(2006JMVA),Komaki(2015)} find the correspondence in prediction using the Kullback--Leibler loss.
In particular, \cite{JohnstoneandMacGibbon(1992),MacGibbon(2010)} find the correspondence in
the asymptotic minimaxity under ellipsoidal and rectangle constraints in high-dimensional Poisson models using the re-scaled squared loss (the local Kullback--Leibler loss).
In spite of the interesting correspondence in \cite{JohnstoneandMacGibbon(1992),MacGibbon(2010)},
the re-scaled squared loss is not compatible with sparsity:
the loss diverges if $\theta_{i}=0$ and $\hat{\theta}_{i}(X)\neq 0$ for at least one index $i$.

Employing the Kullback--Leibler divergence,
this paper presents the results of asymptotic minimaxity in both estimation and prediction for sparse Poisson models, which are clearly parallel to the result for sparse Gaussian models by \cite{MukherjeeandJohnstone(2015)}; see Subsection \ref{subsection: discussion} for detailed discussions.
This paper also covers several new topics in predictive density under sparsity constraints: the adaptation to sparsity, quasi-sparsity, and missing completely at random.

Our strategy leverages spike-and-slab priors.
In the literature, it is known that the choice of slab priors impacts on the statistical optimality
\cite{JohnstoneandSilverman(2004),RockovaandGeorge(2018),CastilloandMismer(2018),CastilloandSzabo(2019)}.
But, the behavior has been studied only for (sub-)Gaussian models and the corresponding results for Poisson models have remained unavailable.
In Proposition \ref{proposition: suboptimality of spike and slab},
we show that slab priors with tails as heavy as the exponential distribution suffer from the minimax sub-optimality.
In Proposition \ref{proposition: optimality of Cauchy or Pareto},
we also show that polynomially decaying slabs can attain the minimax optimality.

Relatively scarce are theoretical studies of zero-inflated or quasi zero-inflated Poisson models in high dimensions in spite of their importance.
\cite{DattaandDunson(2016)} constructs global-local shrinkage priors for high-dimensional quasi zero-inflated Poisson models. The constructed priors have good theoretical properties of the shrinkage factors and of the multiple testing statistics. 
We confirm in Section \ref{section: numerical experiments} that our priors broadly outperform their priors in predictive density,
which indicates our priors are more suitable for prediction.
In contrast, we consider that their priors would be more suitable than our priors in multiple testing or in interpreting shrinkage factors.
We shall also mention that in this direction, interesting and powerful extensions of \cite{DattaandDunson(2016)} are now available in \cite{Hamuraetal}.
Appendix \ref{Appendix: supplemental experiments} in the supplementary material provides the comparison of our predictive density to the Bayes predictive density based on the prior in \cite{Hamuraetal}.

\subsection{Organization and notation}

The rest of the paper is organized as follows.
In Section \ref{section: predictive density estimation in sparse Poisson models},
we present an asymptotically minimax predictive density and an adaptive minimax predictive density for sparse Poisson models,
which is the main result in this paper.
In Section \ref{section: extensions},
we present several extensions of the main result.
In Section \ref{section: numerical experiments},
we conduct simulation studies and present application to real data.
In Section \ref{section: proofs for the main theorems},
we give proofs of main theorems (Theorems 
\ref{theorem: exact minimaxity within sparse Poisson models} and \ref{theorem: adaptive}).
In Section \ref{section: proofs for auxiliary lemmas},
we provide proofs of auxiliary lemmas used in Section \ref{section: proofs for the main theorems}.
All proofs of propositions in Section \ref{section: predictive density estimation in sparse Poisson models} are
given in Appendix \ref{section: proofs for propositions in section 2} of the supplementary material.
All proofs of propositions in Section \ref{section: extensions} are 
given in Appendix \ref{section: proofs for propositions in section 3} of the supplementary material.

Throughout the paper, we will use the following notations.
The notation $a_{n}\sim b_{n}$ signifies that
$a_{n}/b_{n}$ converges to 1 as $n$ goes to infinity.
The notation $O(a_{n})$ indicates a term of which the absolute value divided by $a_{n}$ is bounded for a large $n$.
The notation $o(a_{n})$ indicates a term of which the absolute value divided by $a_{n}$ goes to zero in $n$.
For a function $f:\mathbb{N}^{n}\times \mathbb{N}^{n}\to\R$,
the expectation $\Ep_{\theta}[f(X,Y)]$ indicates the expectation of $f(X,Y)$ with respect to $p(x\mid\theta)q(y\mid\theta)$.
Likewise,
for a function $g:\mathbb{N}\to\R$,
the expectation $\Ep_{\lambda}[g(X_{1})]$ indicates the expectation of $g(X_{1})$ with respect to $\mathrm{Po}(\lambda)$.
Constants  $c_{1},c_{2},\ldots$
and $C_{1},C_{2},\ldots$ do not depend on $n$. Their values may be different at each appearance.

\section{Predictive density for sparse Poisson models}
\label{section: predictive density estimation in sparse Poisson models}

\subsection{Main results}

This section presents main results for prediction using sparse Poisson models:
the precise description of the asymptotic minimax risk;
the construction of the class of asymptotically minimax predictive densities;
and
that of adaptive minimax predictive densities.
Detailed discussions are provided in the subsequent subsection.
Proofs of the theorems are presented in Section \ref{section: proofs for the main theorems}.

The first theorem describes the asymptotic minimax risk as well as the Bayes predictive density attaining the asymptotic minimaxity.
For $r\in (0,\infty)$,
let \[\mathcal{C}:=\mathcal{C}_{r}=\left(\frac{r}{r+1}\right)^{r}\left(\frac{1}{r+1}\right).\]
For $h>0$ and $\kappa>0$, let $\Pi[h,\kappa]$ be an improper prior of the form
\begin{align*}
    \Pi[h,\kappa](d\theta) = \prod_{i=1}^{n}\left\{\delta_{0}(d\theta_{i}) + h\theta_i^{\kappa-1}1_{(0,\infty)}(d\theta_{i})\right\},
\end{align*}
where $\delta_{0}$ is the Dirac measure centered at $0$.
\begin{theorem}
\label{theorem: exact minimaxity within sparse Poisson models}
Fix $r\in (0,\infty)$ and fix a sequence $s_{n}\in(0,n)$ such that $\eta_{n}=s_{n}/n=o(1)$.
Then, the following holds:
\begin{align*}
\mathcal{R}(\Theta[s_{n}])\sim \mathcal{C}s_{n}\log(\eta^{-1}_{n}) \text{ as $n\to\infty$.}
\end{align*}
Further,  the predictive density $q_{\Pi[\eta_{n},\kappa]}$ based on $\Pi[\eta_{n},\kappa]$ with $\kappa>0$ is
asymptotically minimax: i.e.,
\begin{align*}
\sup_{\theta\in\Theta[s_{n}]}R( \theta , q_{\Pi[ \eta_{n},\kappa]}) &\sim \mathcal{R}(\Theta[s_{n}])
 \text{ as $n\to\infty$.}
\end{align*}
\end{theorem}

The derivation of this theorem consists of (i) establishing a lower bound of $\mathcal{R}(\Theta[s_{n}])$ 
based on the Bayes risk maximization, and (ii) establishing an upper bound of it based on the Bayes predictive density $q_{\Pi[\eta_n , \kappa]}$.

The first theorem provides asymptotically minimax strategies, 
but the optimal strategies therein require the true value of $\eta_{n}$.
The second theorem presents the optimal strategies without requiring the true value of $s_{n}$, that is,
adaptive minimax predictive densities for sparse Poisson models.
Let $\hat{s}_{n}:=\max\{1,\#\{ i:X_{i}\geq 1 ,i=1,\ldots,n \}\}$, and let $\hat{\eta}_{n}:=\hat{s}_{n}/n$.
\begin{theorem}
\label{theorem: adaptive}
Fix $r\in(0,\infty)$ and $\kappa>0$.
Then, the predictive density $q_{\Pi[\hat{\eta}_{n},\kappa]}$ is adaptive in the asymptotically minimax sense on the class of exact sparse parameter spaces: i.e.,
for any sequence $s_{n}\in[1,n)$ such that $\sup_{n} s_{n} / n < 1$ and $\eta_{n} = s_{n}/n = o(1)$,
\begin{align*}
    \sup_{\theta\in\Theta[s_{n}]}R(\theta,q_{\Pi[\hat{\eta}_{n},\kappa]})
    &\sim \mathcal{R}(\Theta[s_{n}]) \text{  as }n\to\infty.
\end{align*}
\end{theorem}

The derivation builds upon evaluating the difference between two Kullback--Leibler risks 
$R ( \theta , q_{\Pi[\eta_{n},\kappa]} ) $ and $ R ( \theta , q_{\Pi[\hat{\eta}_{n},\kappa]} ) $.
We will show this difference is negligible uniformly in $\theta$ compared to the minimax risk.
To check this, we use three properties of $\hat{s}_{n}$:
\begin{itemize}
\item The estimate $\hat{s}_{n}$ is bounded below by an absolute constant;
\item The first and the second moments of $|\hat{s}_{n}/s_{n}-1|$ are bounded above by an absolute constant;
\item The estimate $\hat{s}_{n}$ can capture nearly the correct growth rate of $s_{n}$ in the relatively dense regime, whenever the true value of $\theta$ is outside a vicinity of 0. Specifically, we will see that 
\[\sup_{\theta : \max_{i} \theta_i > 1/\sqrt{\log s_{n}} }\Ep_{\theta}[\log s_{n}/\hat{s}_{n}] = O(\sqrt{\log s_{n}})\text{ as $s_{n}\to\infty$}.\]
\end{itemize}
The first and the second properties make the difference negligible in the relatively sparse regime.
The third property makes the difference negligible in the relatively dense regime.
See Subsection \ref{subsection: proof of theorem adaptive} for the detail.

\subsection{Discussions}
\label{subsection: discussion}

Several discussions are provided in order.

\subsubsection{Prediction and estimation, Poisson and Gaussian}
\label{subsubsection: prediction and estimation}

\textit{Prediction and estimation}:
For comparison, let us consider estimating $\theta$ under the Kullback--Leibler risk
$R_{\mathrm{e}}(\theta,\hat{\theta}) := R(\theta, q(\cdot \mid \hat{\theta}))$ as in \cite{Deledalle(2017)}.
The minimax risk $\mathcal{E}(\Theta[s_{n}])$ for estimation is defined in such a way that
$\mathcal{E}(\Theta[s]):=  \mathop{\inf}_{ \hat{\theta} }\sup_{\theta\in\Theta[s]} R_{\mathrm{e}}(\theta, \hat{\theta})$.
Since the minimax risk $\mathcal{E}(\Theta[s_{n}])$ for estimation can be viewed as the minimax risk for prediction when predictive densities are restricted to plug-in predictive densities,
we always have $\mathcal{E}(\Theta[s_{n}]) \ge \mathcal{R}(\Theta[s_{n}])$.

The first proposition describes the asymptotic minimax risk for estimation.
This proposition highlights a gap between $\mathcal{E}(\Theta[s_{n}])$ and $\mathcal{R}(\Theta[s_{n}])$.
The second proposition indicates that the same data-dependent prior as in Theorem \ref{theorem: adaptive} yields an adaptive minimax estimator.
\begin{proposition}
\label{proposition: exact minimaxity for estimation}
Fix $r\in (0,\infty)$ and fix a sequence $s_{n}\in (0,n)$ such that $\eta_{n}= s_{n} / n = o(1)$. 
Then, the following holds:
\begin{align*}
\mathcal{E}(\Theta[s_{n}]) \sim \mathrm{e}^{-1} r^{-1} s_{n}\log(\eta^{-1}_{n})
\text{ as }n\to\infty.
\end{align*}
\end{proposition}

\begin{proposition}
\label{theorem: adaptive minimaxity for estimation}
Fix $r\in(0,\infty)$ and $\kappa>0$.
Then, the Bayes estimator $\hat{\theta}_{\Pi[\hat{\eta}_{n},\kappa]}$ is adaptive in the asymptotically minimax sense on the class of exact sparse parameter spaces:
for any sequence $s_{n}\in[1,n)$ such that $\sup_{n} s_{n} / n < 1$ and $\eta_{n}=s_{n}/n = o(1)$, we have
\begin{align*}
    \sup_{\theta\in\Theta[s_{n}]}R_{\mathrm{e}}(\theta,\hat{\theta}_{\Pi[\hat{\eta}_{n},\kappa]})
    &\sim \mathcal{E}(\Theta[s_{n}]) \text{  as }n\to\infty.
\end{align*}
\end{proposition}

\begin{figure}[h]
\begin{minipage}{0.49\hsize}
\hspace{-0.4cm}
\includegraphics[width=6cm]{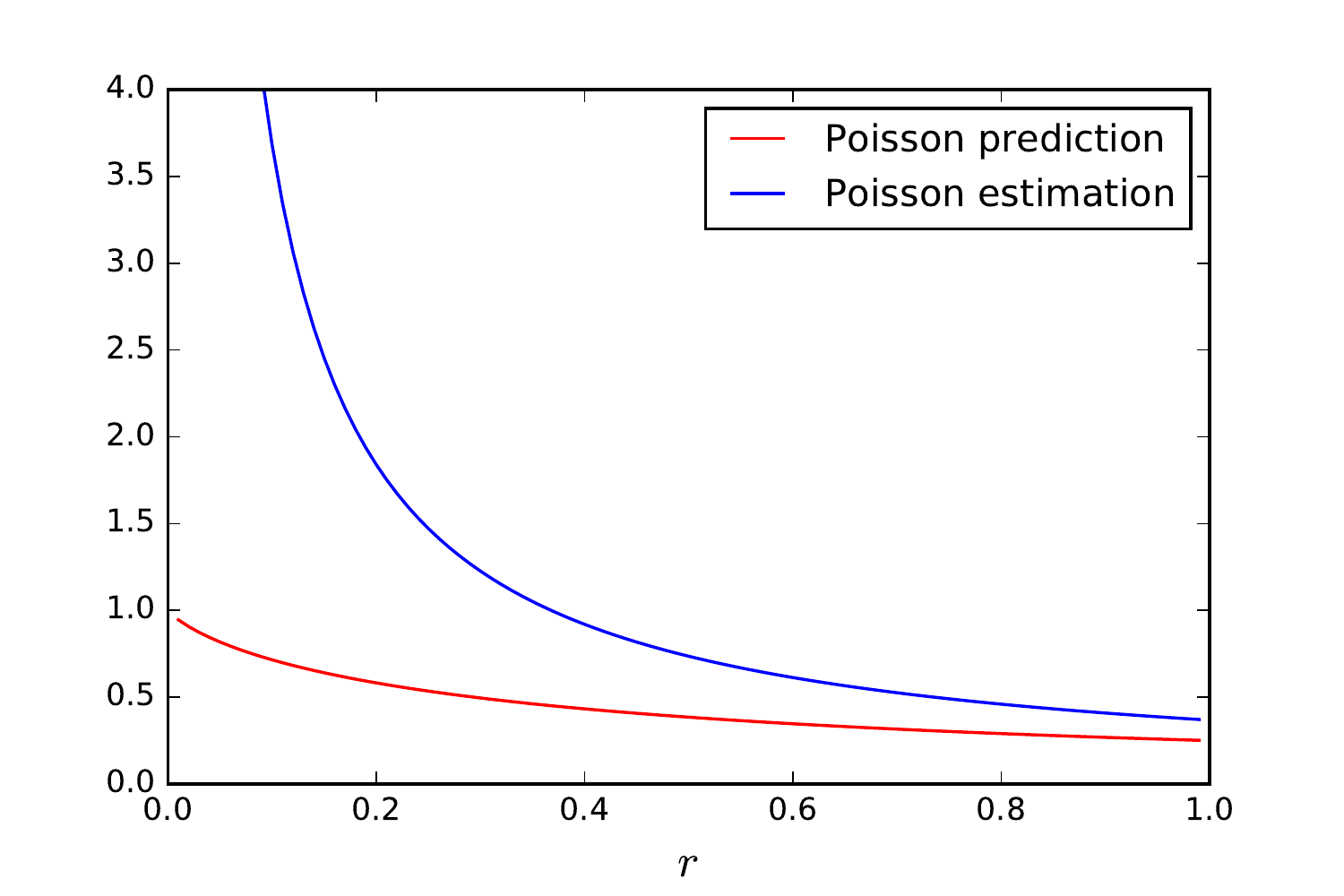}
\caption{Predictive and estimative \,\,\,\protect\linebreak minimax risks for sparse Poisson \,\,\,\protect\linebreak models: the holizontal axis represents $r$.}
\label{figure: exact constant comparison Poisson}
\end{minipage}
\begin{minipage}{0.49\hsize}
\hspace{-0.4cm} \includegraphics[width=6cm]{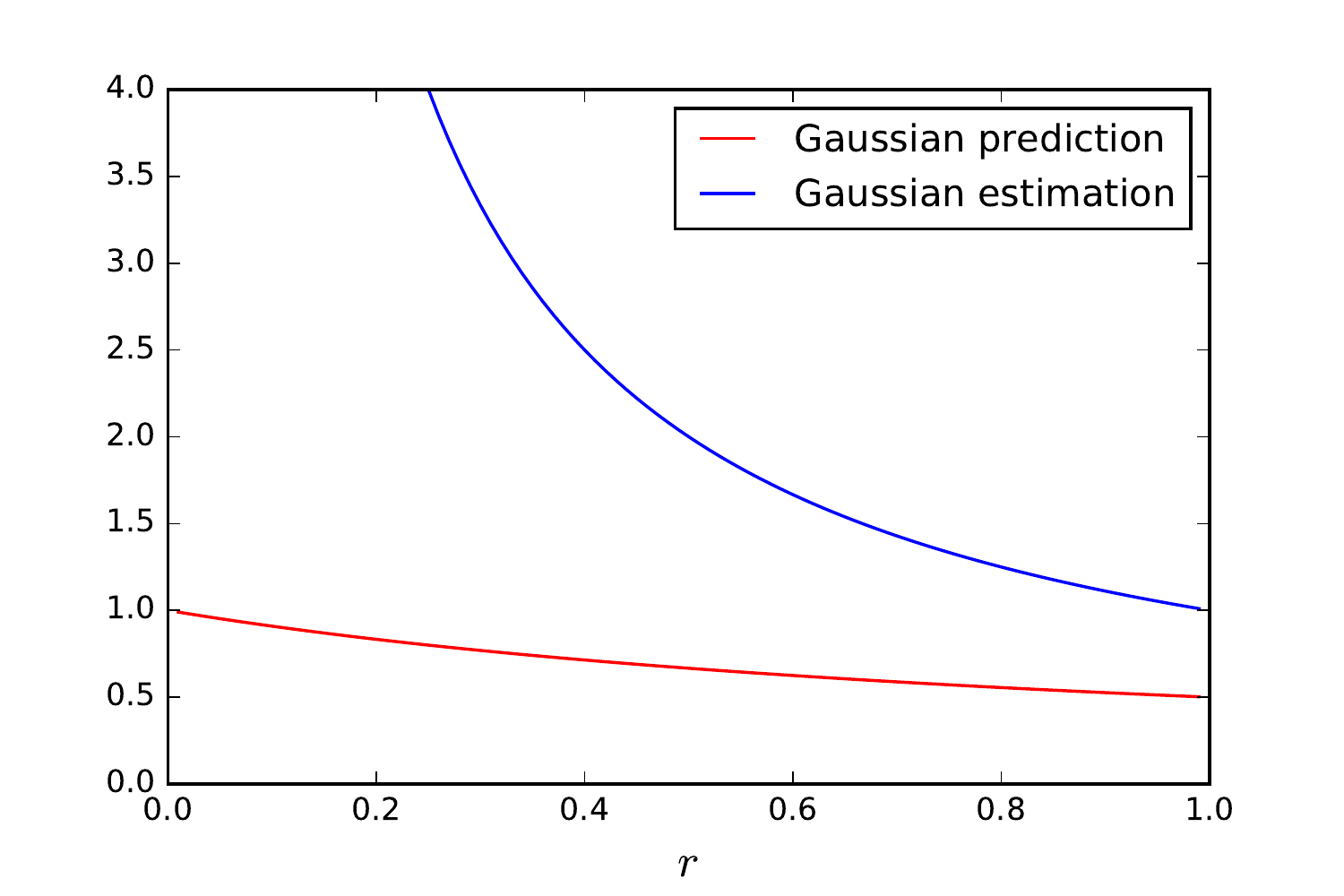}
\caption{Predictive and estimative \,\,\,\protect\linebreak minimax risks for sparse Gaussian \,\,\,\protect\linebreak models: the holizontal axis represents $r$.}
\label{figure: exact constant comparison Gaussian}
\end{minipage}
\end{figure}

According to Theorem \ref{theorem: exact minimaxity within sparse Poisson models} and Proposition \ref{proposition: exact minimaxity for estimation},
the rates (with respect to $n$) of minimax risks for estimation and for prediction are identical.
But, the sharp constants of these minimax risks are different with $r$.
The sharp constant of $\mathcal{R}(\Theta[s_{n}])$ (, i.e., $\mathcal{C}$,) increases as $r$ decreases but remains bounded above by $1$,
while that of $\mathcal{E}(\Theta[s_{n}])$ (, i.e., $\mathrm{e}^{-1}r^{-1}$) grows to infinity as $r$ decreases.
Further, $\mathcal{C}\sim \mathrm{e}^{-1}r^{-1}$ as $r$ increases.

\textit{Poisson and Gaussian}:
\cite{MukherjeeandJohnstone(2015)} shows the asymptotic minimax risk for prediction using sparse Gaussian models
is equal to $\{1/(1+r)\}s_{n}\log \eta_{n}^{-1}$ with $r$ the ratio of sample sizes of current observations to those of future observations.
Comparing our results with \cite{MukherjeeandJohnstone(2015)},
we find interesting similarities between sparse Gaussian and sparse Poisson models.
First, the rates with respect to $n$ of these two problems are identical to $s_{n}\log \eta_{n}^{-1}$.
Second, 
Figures \ref{figure: exact constant comparison Poisson} and \ref{figure: exact constant comparison Gaussian} show the comparisons of the exact constants of minimax risks for sparse Poisson and Gaussian models.
The vertical line indicates values of the risks and the horizontal line indicates values of $r$.
They show the similarity of the behavior with respect to $r$ of minimax risks in Poisson and Gaussian cases.
An interesting observation in comparison of Poisson and Gaussian cases is
that the exact constants of predictive minimax risks in both cases get closer to 1 as $r$ approaches to 0.

\subsubsection{Spike-and-slab priors}

\textit{Computation}:
Let us mention a computational advantage of using improper slab priors ahead.
Bayes predictive densities often suffer from computational intractability because they may involve several numerical integrations.
Using improper slab priors, we can avoid such a computational issue in our set-up.
In fact,
the Bayes predictive density based on $\Pi[h,\kappa]$ has the explicit form
\begin{align*}
&q_{\Pi[h,\kappa]}(y\mid x) \\
&= \prod_{i=1}^{n}
\left\{ \omega_{i} \delta_{0}(y_{i}) 
+ (1-\omega_{i}) \begin{pmatrix} x_{i}+y_{i} +\kappa -1 \\ y_{i} \end{pmatrix} \left(\frac{r}{r+1}\right)^{x_{i}+\kappa}\left(1-\frac{r}{r+1}\right)^{y_{i}} \right\},
\end{align*}
where
\begin{align*}
\omega_{i}:=
\begin{cases}
1\big{/}\big{\{}1+h \Gamma(\kappa)/r^{\kappa}\big{\}} & \text{ if } x_{i}=0, \\
0 & \text{ if }x_{i}\geq 1.
\end{cases}
\end{align*}
The coordinate-wise marginal distribution of $q_{\Pi[h,\kappa]}$ is a zero-inflated negative binomial distribution
and thus sampling from $q_{\Pi[h,\kappa]}$ is easy.


\textit{Condition on slab priors}:
Spike and slab priors has been recently well-investigated in sparse Gaussian models;
see \cite{JohnstoneandSilverman(2004),RockovaandGeorge(2018),CastilloandMismer(2018),CastilloandSzabo(2019)}.
The existing results for sparse Gaussian models are summarized as follows:
\begin{itemize}
\item For point estimation, a slab prior with its tail at least as heavy as Laplace distribution yields
a rate-optimal point estimator of the sparse mean (see \cite{JohnstoneandSilverman(2004)});
\item For uncertainty quantification, Laplace slab prior yields a rate sub-optimal posterior $\ell^{2}$-moment but Cauchy slab prior yields a rate-optimal one (see \cite{CastilloandMismer(2018)}).
\end{itemize}
Here we derive both necessary and sufficient conditions for the minimax prediction in sparse Poisson models.

Consider the following condition on a prior $\Pi$: For the posterior mean $\hat{\theta}_{\Pi}$, the asymptotic equation
\begin{align}
\frac{|\hat{\theta}_{\Pi,i}(x) - x_{i}/r |}{x_{i}/r} \to 0 \text{ as }x_{i}\to\infty
\label{weak tail robustness}
\end{align}
holds for all $i$.
This condition is a Poisson variant of tail-robustness of the posterior mean \cite{Carvalhoetal(2010)}.
\cite{Hamuraetal} studies this condition and a stronger condition in the context of the global-local shrinkage.
The next proposition implies Condition (\ref{weak tail robustness}) is necessary for the minimax prediction.

\begin{proposition}
\label{proposition: suboptimality of spike and slab}
Fix a sequence $s_{n}\in (0,n)$ such that $\eta_{n}= s_{n} / n = o(1)$.
If a spike-and-slab prior does not satisfy (\ref{weak tail robustness}), then we have
\begin{align*}
\sup_{\theta\in \Theta[s_{n}]} R(\theta, q_{\Pi}) \Big{/}
\mathcal{R}(\Theta[s_{n}]) \to\infty \text{ as $n\to\infty$}.
\end{align*}
\end{proposition}

Though Condition (\ref{weak tail robustness}) is not sufficient,
it becomes a useful criterion for checking the sub-optimality of a given predictive density.
For example, a spike-and-slab prior with an exponentially decaying slab (e.g., Laplace slab) does not satisfy (\ref{weak tail robustness}) and hence yields a sub-optimal predictive density.

Consider a spike-and-slab prior $\Pi_{\gamma}[\eta] := \prod_{i=1}^{n} \{ (1-\eta) \delta_{0}(d\theta_{i}) + \eta \gamma(\theta_{i})d\theta_{i} \}$
with a slab density $\gamma$.
Make the following conditions on $\gamma$:
\begin{align}
\sup_{\lambda>0}\Big| \lambda \frac{d}{d\lambda} \log \gamma(\lambda)  \Big| = \Lambda < \infty.
\label{condition on slab}
\end{align}
and
\begin{align}
\int_{0}^{\infty} \mathrm{e}^{-\lambda}\gamma(\lambda) d\lambda < \infty.
\label{condition integrability}
\end{align}
By the fundamental theorem of calculus,
Condition (\ref{condition on slab}) implies $c_{1}\lambda^{-\Lambda} \le \gamma(\lambda) \le C_{1} \lambda^{\Lambda}$ with some $c_{1},C_{1}>0$.
Condition (\ref{condition integrability}) is a posterior integrability condition and is satisfied by all proper priors as well as ours.
A class of slabs satisfying (\ref{condition on slab}) and (\ref{condition integrability}) includes
half-Cauchy, Pareto, and regularly varying priors with index -1 discussed in \cite{MaruyamaandTakemura(2008),Hamuraetal}.
We can easily check that spike-and-slab priors with slabs satisfying (\ref{condition on slab}) and (\ref{condition integrability})
meet Condition (\ref{weak tail robustness}) and 
can show they yield the optimal predictive densities.

\begin{proposition}
\label{proposition: optimality of Cauchy or Pareto}
Fix $r\in(0,\infty)$ and fix a sequence $s_{n}\in (0,n)$ such that $\eta_{n}= s_{n} / n = o(1)$.
Let $\Pi_{\gamma}[\eta_{n}]$ be a spike-and-slab prior with $\gamma$ satisfying (\ref{condition on slab}) and (\ref{condition integrability})
Then, the predictive density $q_{\Pi_{\gamma}[\eta_{n}]}$ based on $\Pi_{\gamma}[\eta_{n}]$ is asymptotically minimax:
i.e.,
\begin{align*}
    \sup_{\theta\in\Theta[s_{n}]}R(\theta,q_{\Pi_{\gamma}[\eta_{n}]})
    &\sim \mathcal{R}(\Theta[s_{n}]) \text{  as }n\to\infty.
\end{align*}
\end{proposition}

\textit{Optimal scaling}:
Our approach uses an improper prior within their mixture and therefore the scaling of the improper slab prior impacts on the resulting predictive density: $\Pi[L\eta_{n},\kappa]$ for arbitrary $L>0$ produces a different predictive density that is asymptotically minimax.
This arbitrariness of the scale is well-known in the objective Bayesian literature and worrisome in practice;
see \cite{Hartigan(2002),LeungandBarron(2006),DawidandMusio(2015)}.

The next proposition provides a guideline for choosing $L$ and removes this arbitrariness.
Let
\[
L^{*}:=L^{*}_{r,\kappa}= \mathcal{C} / \mathcal{K} \quad \text{ with } \mathcal{K} := \Gamma(\kappa+1)\frac{r^{-\kappa}-(r+1)^{-\kappa}}{\kappa}.
\]

\begin{proposition}
\label{proposition: optimal scaling}
Fix $r\in(0,\infty)$ and $\kappa>0$.
Fix also a sequence $s_{n}\in(0,n)$ such that $\eta_{n}=s_{n}/n =o(1)$.
Then, the predictive density $q_{\Pi[L\eta_{n},\kappa]}$ with $L>0$ and $\kappa>0$ satisfies
\begin{equation}
\sup_{\theta\in\Theta[s_{n}]} R(\theta, q_{\Pi[L\eta_{n},\kappa]}) 
\leq \mathcal{C}s_{n}\log (\eta_{n}^{-1})
-\mathcal{C}s_{n}\log L + \mathcal{K}s_{n}L+\Upsilon
\label{eq: upper bound on maximum risk}
\end{equation}
with $\Upsilon$ terms that are independent of $L$ or that are $O(s_{n}\eta_{n})$,
and $L^{*}$ minimizes the right hand sides in (\ref{eq: upper bound on maximum risk}) with respect to $L$.
\end{proposition}

This result shows that the scale of improper slab priors can be specified by the predictive setting (characterized by $r$). Our idea here is relevant to \cite{Hartigan(2002)}: In \cite{Hartigan(2002)}, the scale of improper priors within their mixture is determined to yield log-posterior probabilities that coincide with log maximum likelihood plus an Akaike factor; see also the appendix of \cite{LeungandBarron(2006)}. In this light, \cite{Hartigan(2002)} and this paper indicate that the specifications of the scale of improper priors within their mixture can be done from a predictive viewpoint.

\section{Extensions}
\label{section: extensions}

We present two extensions of our results: one is quasi-sparsity; the other is missing completely at random.
These extensions are important in practice.

\subsection{Quasi-sparsity}

We introduce the notion of the quasi sparse parameter space.
Given $s \in (0,n)$ and a threshold $\varepsilon>0$,
the quasi sparse parameter space is defined as $\Theta[s,\varepsilon] := \{ \theta\in \R_{+}^{n}: N(\theta,\varepsilon)\leq s \}$,
where $N(\theta,\varepsilon):=\#\{i: \theta_{i} > \varepsilon\}$, $\varepsilon>0$.
A threshold value $\varepsilon$ determines whether the parameter value of each coordinate is near-zero or not.

The next two propositions specifies the minimax risk over the quasi-sparse parameter space and presents an adaptive minimax predictive density.
\begin{proposition}
\label{proposition: quasi-sparse minimax} 
Fix $r\in(0,\infty)$ and fix a sequence $s_{n}\in(0,\infty)$ such that $\eta_{n} = s_{n}/n = o(1)$. Fix also a shrinking sequence $\varepsilon_{n}>0$ such that $\varepsilon_{n} =o( \eta_{n} )$.
Then,
for the quasi sparse parameter space $\Theta[s_{n},\varepsilon_{n}]$,
the following holds:
\begin{align*}
\mathcal{R}(\Theta[s_{n},\varepsilon_{n}]):= \inf_{\hat{q}}\sup_{\theta\in \Theta[s_{n},\varepsilon_{n}]} R(\theta, \hat{q}) \sim \mathcal{C} s_{n} \log (\eta^{-1}_{n}) \text{ as $n\to\infty$.}
\end{align*}
Further, the predictive density $q_{\Pi[\eta_{n},\kappa]}$ based on $\Pi[\eta_{n},\kappa]$ with $\kappa>0$ is asymptotically minimax: i.e.,
\begin{align*}
\sup_{\theta\in\Theta[s_{n},\varepsilon_{n}]}R( \theta , q_{\Pi[ \eta_{n},\kappa]}) &\sim \mathcal{R}(\Theta[s_{n},\varepsilon_{n}]) \text{ as }n\to\infty.
\end{align*}
\end{proposition}

\begin{proposition} 
\label{proposition: adaptive in quasi-sparse}
Fix $r\in(0,\infty)$ and $\kappa>0$.
Then, the predictive density $q_{\Pi[\hat{\eta}_{n},\kappa]}$ is adaptive in the asymptotically minimax sense on the class of quasi sparse parameter spaces:
i.e., for any sequence $s_{n}\in[1,n)$ such that $\sup_{n} s_{n} / n < 1$ and $\eta_{n}=s_{n}/n=o(1)$
and for any sequence $\varepsilon_{n}>0$ such that $\varepsilon_{n}=o(\eta_{n})$, we have
\begin{align*}
    \sup_{\theta\in\Theta[s_{n},\varepsilon_{n}]}R(\theta,q_{\Pi[\hat{\eta}_{n},\kappa]})
    &\sim \mathcal{R}(\Theta[s_{n},\varepsilon_{n}]) \text{  as }n\to\infty.
\end{align*}
\end{proposition}

\subsection{Missing completely at random}
\label{subsection: MCAR}
We describe prediction using sparse Poisson models when the current observation is missing completely at random (MCAR).
Let $r_i$'s ($i = 1,2,\ldots$) be positive random variables.
Given $r_{i}$ ($i=1,\ldots,n$),
let $X_i$ $(i = 1, 2, \ldots, n)$ be a current observation independently distributed according to Po$(r_i\theta_i)$, 
and 
let $Y_i$ $(i = 1, 2, \ldots, n)$ be a future observation independently distributed according to Po$(\theta_i)$,
where $\theta_{i}$ $(i=1,\ldots,n)$ is an unknown parameter.
Suppose that $X = (X_1,\ldots,X_n)$ and $Y = (Y_1,\ldots, Y_n)$ are independent.
We denote by $R(\theta,\hat{q}\mid \{r_{i}\})$ the Kullback--Leibler risk conditioned on $r_{i}$s.
We also denote by $\overline{\mathcal{R}}(\Theta[s_{n}] \mid \{r_{i}\})$ the minimax  Kullback--Leibler risk over $\Theta[s_{n}]$ conditioned on $r_{i}$s.

To present mathematically unblemished results,
we assume that $r_{i}$s are independent and identically distributed according to a sampling distribution $G$,
and
make the following condition on $G$.
Let $\mathbb{E}_G$ be the expectation with respect to $G$.
\begin{condition}
\label{Condition: tail and edge}
A sampling distribution $G$ satisfies the following:
(i) $\mathbb{E}_{G}[r_{1}^{2}]<\infty$;
(ii) $\mathbb{E}_{G}[r_{1}^{-2}]<\infty$.
\end{condition}
Condition \ref{Condition: tail and edge} (i) is usual.
Condition \ref{Condition: tail and edge} (ii) excludes any distribution $G$ highly concentrated around $0$ and is not stringent.
Consider a longitudinal situation in which $X_{i}$ ($i=1,\ldots,n$) is obtained as the sum of $\{X_{i,j}:j=1,\ldots,r_{i}\}$,
where $r_{i}$ ($i=1,\ldots,n$) represents the sample size in the $i$-th coordinate,
and
for each $i$,
$X_{i,j}$ $(j=1,\ldots,r_{i})$ follows Po$(\theta_{i})$.
Condition \ref{Condition: tail and edge} implies that for each coordinate there exists at least one observation:
$r_{i}\geq 1$.
Note that in our real data applications (Subsection \ref{subsection: applications to real data}), Condition \ref{Condition: tail and edge} (ii) is satisfied.


The following propositions describe the asymptotic minimax risk and present an adaptive minimax predictive density.
Fix an infinite sequence $\{r_i\in(0, \infty):i\in\mathbb{N}\}$ such that $0 < \inf_i r_i 
\leq \sup_i r_i <\infty$.
For any $i\in\mathbb{N}$, let  \[\mathcal{C}_{i}:=\mathcal{C}_{r_{i}}=\left(\frac{r_i}{r_i+1}\right)^{r_i}\left(\frac{1}{r_i+1}\right).\]
Let $\overline{\mathcal{C}}:=\overline{\mathcal{C}_{n}}=\sum_{i=1}^{n}\mathcal{C}_{i} /n .$

\begin{proposition}
\label{proposition: minimax risk in random design}
Fix a sequence $s_{n}\in(0,n)$ such that $\eta_{n}=s_{n}/n=o(1)$.
Under Condition \ref{Condition: tail and edge},
we have
\[
 \underset{n\to\infty}{\operatorname{plim}} \text{ }  
	\overline{\mathcal{R}}(\Theta[s_{n}] \mid \{r_{i}\}) 
	\Big{/} \{\mathbb{E}_G [\overline{\mathcal{C}}] s_{n} \log (\eta_{n}^{-1})\} 
= 1.
\]
Further, the predictive density $q_{\Pi[\eta_{n},\kappa]}$ based on $\Pi[\eta_{n},\kappa]$ with $0< \kappa \le 1$ is asymptotically minimax:
\[
 \underset{n\to\infty}{\operatorname{plim}} \text{ }  
	\overline{\mathcal{R}}(\Theta[s_{n}]\mid \{r_{i}\}) 
	\Big{/} R(\theta, q_{\Pi[\eta_{n},\kappa]}\mid \{r_{i}\})
= 1.
\]
\end{proposition}

\begin{proposition} 
\label{proposition: adaptive in random design}
Fix $\kappa \in(0,1]$ and assume that Condition \ref{Condition: tail and edge} holds.
Then, the predictive density $q_{\Pi[\hat{\eta}_{n},\kappa]}$ is adaptive in the asymptotically minimax sense on the class of exact sparse parameter spaces:
for any sequence $s_{n}\in[1,n)$ such that $\sup_{n} s_{n} / n < 1$ and $\eta_{n}=s_{n}/n=o(1)$, 
\begin{align*}
 \underset{n\to\infty}{\operatorname{plim}} \text{ }  
	\overline{\mathcal{R}}(\Theta[s_{n}]\mid \{r_{i}\}) 
	\Big{/} R(\theta, q_{\Pi[\eta_{n},\kappa]} \mid \{r_{i}\})
= 1.
\end{align*}
\end{proposition}

Detailed discussions for the results herein are given in Appendix \ref{subsection: Discussion: MCAR}.

\begin{remark}
\label{remark: optimal choice}
We remark that the optimal scaling (in the sense of Proposition \ref{proposition: optimal scaling}) for this set-up is $\overline{L}$, where
\begin{align*}
\overline{L} = \overline{\mathcal{C}} /  \overline{\mathcal{K}} \quad \text{ with }  \overline{\mathcal{K}}= \Gamma(\kappa+1)\sum_{i=1}^{n}\Big{\{}r_{i}^{-\kappa}-(r_{i}+1)^{-\kappa}\Big{\}}\Big{/}(n\kappa).
\end{align*}
The derivation follows almost the same line as in the proof of Proposition \ref{proposition: optimal scaling} and is omitted.
\end{remark}

\section{Simulation studies and application to real data}
\label{section: numerical experiments}

\subsection{Simulation studies}
\label{subsection: simulation studies}

This subsection presents simulation studies to compare the performance of various predictive densities. 
The codes for implementing the proposed method are available at \url{https://github.com/kyanostat/sparsepoisson}.

Consider a sparse Poisson model described as follows.
Parameter $\theta$ and observations $X$ and $Y$ are drawn from
\begin{align*}
\theta_{i} &\sim \nu_{i} e_{S,i} \ (i=1,\dots,n),\\
 X \mid \theta &\sim \otimes_{i=1}^{n}\mathrm{Po}(r\theta_{i}),\
 Y \mid \theta \sim \otimes_{i=1}^{n}\mathrm{Po}(\theta_{i}),
 \text{ and }
 X \indep Y \mid \theta,
\end{align*}
respectively. Here, 
\begin{itemize}
    \item $\nu_{1},\ldots,\nu_{n}$ are independent samples from the Gamma distribution 
with a shape parameter $10$ and a scale parameter $1$;
\item $S$ is drawn from the uniform distribution on all subsets having exactly $s$ elements;
\item $\nu_{1},\ldots,\nu_{n}$ and $S$ are independent.
\end{itemize}
Here for a subset $J\subset \{1,\ldots,n\}$, $e_{J}$ indicates the vector
whose $i$-th component is 1 if $i\in J$ and 0 otherwise.
We examine two cases for $(n, s ,r)$,
and generate 500 current observations $X$'s
and 500 future observations $Y$'s.
See Appendix \ref{Appendix: supplemental experiments} in the supplementary material
for the results with different choices of $(n,s,r)$.

We compare the following four predictive densities:
\begin{itemize}
\item The Bayes predictive density based on $\Pi[L^{*}\hat{\eta}_{n}, \kappa]$ with $L^{*}$ in  Proposition \ref{proposition: optimal scaling};
\item The Bayes predictive density based on the shrinkage prior in \cite{Komaki(2004)};
\item The Bayes predictive density based on the Gauss hypergeometric prior  in \cite{DattaandDunson(2016)};
\item The plug-in predictive density based on an $\ell_1$-penalized estimator.
\end{itemize}
The second predictive density is shown in \cite{Komaki(2004)} to dominate the Bayes predictive density based on the Jeffreys prior.
This predictive density has a hyper-parameter $\beta$ and in simulation studies it is fixed to be 1.
The third predictive density employs the global-local prior proposed in \cite{DattaandDunson(2016)} and the specification of the hyper-parameters follows the online support pages the authors provide.

The performance of predictive densities is evaluated by the following three measures:
\begin{itemize}
\item the mean of the $\ell_1$ distance ($\sum_{i=1}^{n}|u_{i}-v_{i}|$ for $u,v\in\R^{n}$) between the mean of a predictive density and a future observation,
\item the predictive log likelihood, that is, the log of the value of a predictive density at sampled $Y$ and $X$, and
\item  the (empirical) coverage probability of $Y$ on the basis of the joint 90\%-prediction set constructed by a predictive density.
\end{itemize}

Tables 
\ref{table: simulation under exact sparsity  in homogeneity, (n,s,r)=(200,5,1)} 
and
\ref{table: simulation under exact sparsity  in homogeneity, (n,s,r)=(200,5,20)}
show 
the results of the comparison.
The following abbreviations are used in the tables.
The Bayes predictive density proposed in \cite{DattaandDunson(2016)} is
abbreviated to GH.
The Bayes predictive density proposed in \cite{Komaki(2004)} is abbreviated to K04.
The plug-in density based on an $\ell_{1}$-penalized estimator with regularization parameter $r\lambda$ is abbreviated to $\ell_{1}$ ($\lambda$).
The abbreviation $\ell_{1}$ distance represents a mean $\ell_{1}$ distance.
The abbreviation PLL represents a predictive log likelihood.
The abbreviation $90\%$CP represents the empirical coverage probability based on
a $90\%$-prediction set.

\begin{table}[h]
\caption{Comparison of predictive densities with $(n,s,r)=(200,5,1)$: the $\ell_1$ distance, PLL, and $90\%$CP represent the mean $\ell_1$ distance, the predictive log likelihood, and the empirical coverage probability based on a $90\%$-prediction set, respectively. 
For each result, the averaged value is followed by the corresponding standard deviation. Underlines indicate the best performance.}
\label{table: simulation under exact sparsity  in homogeneity, (n,s,r)=(200,5,1)}

\centering

\begin{tabular}{|c||c|c|c|c|c|} \hline
      &$\Pi[L^{*}\hat{\eta}_{n},0.1]$
      &$\Pi[L^{*}\hat{\eta}_{n},1.0]$
      & GH & K04 & $\ell_{1}$ ($\lambda=0.1$) 
      \\ \hline\hline
      $\ell_{1}$ distance & 
      \underline{18.8} (5.8)& 21.9 (6.8)& 104 (4.9)  & 96.5 (8.1) & 22.1 (7.8)
      \\
      PLL &
      \underline{-15.4} (1.8)& -16.1 (1.6) &-66.3 (3.3) & -86.2 (8.8) & -Inf
      \\
      90\%CP (\%) &  92.6 (0.1)& 95.8 (0.1)& \underline{92.0} (1.5) & 40.5 (24.4) & 49.4 (21.6)
      \\
      \hline
\end{tabular}
\end{table}
\begin{table}[h]
\caption{Comparison of predictive densities with $(n,s,r)=(200,5,20)$: the $\ell_1$ distance, PLL, and $90\%$CP represent the mean $\ell_1$ distance, the predictive log likelihood, and the empirical coverage probability based on a $90\%$-prediction set, respectively. 
For each result, the averaged value is followed by the corresponding standard deviation. Underlines indicate the best performance.
}
\label{table: simulation under exact sparsity  in homogeneity, (n,s,r)=(200,5,20)}
\centering
\begin{tabular}{|c||c|c|c|c|c|c|} \hline
      &
      $\Pi[L^{*}\hat{\eta}_{n},0.1]$&
      $\Pi[L^{*}\hat{\eta}_{n},1.0]$&
      GH & K04 & $\ell_{1}$ ($\lambda=0.1$)
      \\ \hline\hline
      $\ell_{1}$ distance & 
      \underline{14.0} (4.9)&   14.5 (4.5)&   15.7 (1.7)  & 22.5 (5.2) & 14.1 (4.5)
      \\
      PLL &
      \underline{-13.3} (1.6)&  -13.5 (1.5) & -15.6 (1.5) &  -21.6 (2.2) & -Inf
      \\
      90\%CP (\%) &
      \underline{90.0} (0.0)& 89.4 (0.0) & 97.6 (0.7) &  97.5 (1.4) &  86.3 (3.9)
      \\
      \hline
\end{tabular}

\end{table}

The results have been summarized as follows.
In regard to the $\ell_1$ distances, 
samples from the predictive density based on $\Pi[L^{*}\hat{\eta}_{n},0.1]$ 
are closer to future observations than those of three other classes of predictive densities.
In regard to the empirical coverage probabilities,
the  predictive densities based on $\Pi[L^{*}\hat{\eta}_{n},0.1]$ and the Gauss hypergeometric prior
give the empirical coverage probabilities of $Y$ that are relatively close to the nominal level.
The prediction set of the plug-in predictive density based on the $\ell_{1}$-penalized estimator is too narrow to cover future observations.
This is mainly because 
for this plug-in predictive density,
an $\ell_1$-penalized estimator
returns zero for a coordinate at which the current observation is zero
and most of the marginal predictive intervals degenerate into zero.
This degeneracy also induces the divergence of a predictive log likelihood value of the plug-in predictive density based on an $\ell_1$-penalized estimator.

Supplemental material provides additional numerical experiments including quasi-sparse, MCAR, large $s_{n}$ setings, and the comparison with the recently developed prior distribution \cite{Hamuraetal}.
These show that the proposed predictive density with $\kappa=0.1$ has stable predictive performance and 
this value of $\kappa$ is suggested as a good default choice.

\subsection{Application to real data}
\label{subsection: applications to real data}
We apply our methods to Japanese crime data
from an official database called \textit{the number of crimes in Tokyo by type and town} \cite{Metropolitan}.
This database reports the total numbers of crimes in Tokyo Prefecture.
They are classified by town and also by the type of crimes.
A motivation for this analysis comes from the importance of taking measures against future crimes by utilizing past crime data.

We use pickpocket data from 2012 to the first half of 2018 at 978 towns in eight wards (Bunkyo Ward, Chiyoda Ward, Chuo Ward, Edogawa Ward, Koto Ward, Minato Ward, Sumida Ward, and Taito Ward). 
Figure \ref{figure: Japanese crime} shows total counts of pickpockets from 2012 to 2017 for all towns in the wards.
The scale of the pickpocket occurrences in each town is expressed by a gradation of colors:
there have occurred more pickpockets in a deeper-colored town over 6 years.
There have not occurred any pickpocket in white-colored towns.
As seen from Figure \ref{figure: Japanese crime}, 
the data have zero or near-zero counts at a vast majority of locations, while having relatively large counts at certain locations.

\begin{figure}[h]
\centering\includegraphics[width=14cm]{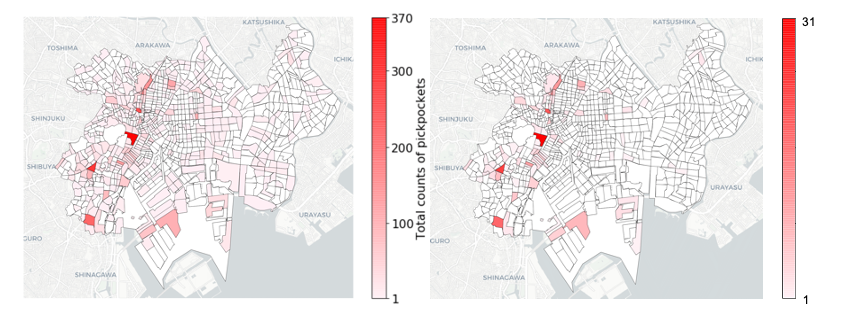}
\caption{Pickpocket data: (Left) Total numbers of pickpockets from 2012 to 2017 in eight wards (Bunkyo Ward, Chiyoda Ward, Chuo Ward, Edogawa Ward, Koto Ward, Minato Ward, Sumida Ward, and Taito Ward). There have occurred more pickpockets in a deeper-colored town over 6 years.
There have occurred no pickpockets in white-colored towns. (Right) Coordinate-wise medians of the proposed predictive density $q_{\Pi[\overline{L}\hat{\eta}_{n},0.1]}$ for the pickpockets in the first half of 2018.}
\label{figure: Japanese crime}
\end{figure}

The experimental settings are as follows.
The data at the 978 towns from 2012 to 2017 are used as current observations.
The data in the first half of 2018 are used as future observations.
Since the counts in the first half of 2018 would be considered as the half of the total counts in 2018, in general,
the ratio $r$ of sample sizes is set as $r=12$.
However, some observations are missing
because several towns, though in rare cases, did not report the counts.

As in Subsection \ref{subsection: simulation studies}, 
we compare the proposed predictive density $q_{\Pi[\overline{L}\hat{\eta}_n,\kappa]}$ (with $\overline{L}$ in Remark \ref{remark: optimal choice})
to the three existing predictive densities, that is, 
the Bayes predictive density GH based on a Gauss hypergeometric prior,
the Bayes predictive density K04 based on the shrinkage prior,
and the plug-in predictive density based on an $\ell_{1}$-regularized estimator.
An estimator $\hat{s}_{n}$ used in $q_{\Pi[\overline{L}\hat{\eta}_n,\kappa]}$ is set as the simple estimator described before Theorem \ref{theorem: adaptive}
with a slight modification: we use the mean of the numbers of values greater than 1 in each year as $\hat{s}_{n}$.
The value of $\kappa$ is fixed to be 0.1 as the numerical simulations suggest.
We evaluate these predictive densities on the basis of the following two measures:
\begin{itemize}
\item  The weighted $\ell_{1}$ distance with the weight proportional to $r$ between the mean vector of a predictive density and the data obtained in the first half of 2018, and
\item the predictive log likelihood at the data obtained in the first half of 2018. 
\end{itemize}
%
\begin{table}[h]
\caption{
Comparison of predictive densities in pickpocket data by the weighted $\ell_{1}$ distance (W-$\ell_{1}$ distance) and
the predictive log likelihood (PLL):
underlines indicate the best performances. 
}
\label{table: pickpocket}
\centering
\begin{tabular}{|c||c|c|c|c|c|c|} \hline
      & $\Pi[\overline{L}\hat{\eta}_{n},0.1]$ & 
      GH & K04 & $\ell_{1}$ ($\lambda=0.1$)  \\ \hline\hline
      W-$\ell_{1}$ distance & \underline{273} & 
      293  & \underline{273} & 297 \\
      PLL & \underline{-399} & 
      \underline{-399} & -429 & -Inf \\ \hline
\end{tabular}

\end{table}
%

Table \ref{table: pickpocket} shows a summary of comparisons.
In all measures, the proposed predictive density $q_{\Pi[\overline{L}\hat{\eta}_{n},0.1]}$
has the best scores.
The predictive density provides not only the mean but also the other statistics.
The figure to the right of Figure \ref{figure: Japanese crime} displays the coordinate-wise medians of the proposed predictive density. It highlights crime spots at which many pickpockets have occurred over the past 6 year, and at the same time shows potential crime spots that the spike-and-slab prior structure suggests.

\section{Proofs of main theorems}
\label{section: proofs for the main theorems}

This section presents proofs for the main theorems in Section \ref{section: predictive density estimation in sparse Poisson models}.

\subsection{Supporting lemmas}
\label{subsection: supporting lemmas for main theorems}

We begin with stating the supporting lemmas.
For an estimator $\hat{\theta}$,
let 
\begin{align*}
R_{\mathrm{e}}(\theta,\hat{\theta}):=R(\theta,q(\cdot\mid\hat{\theta})) 
=  \Ep_{\theta} \sum_{i=1}^{n}\left[\theta_{i}\log\frac{\theta_{i}}{\hat{\theta}_{i}(X)} - \theta_{i} + \hat{\theta}_{i}(X)\right].
\end{align*}
For a prior $\Pi$ of $\theta$, let
\begin{align*}
\hat{\theta}_{\Pi,i}(x;t) = \int \theta_{i} p(x\mid t \theta) d\Pi(\theta) \bigg{/} \int p(x \mid t\theta) d\Pi(\theta), \ i=1,2,\ldots,n,
\end{align*}
and let $\hat{\theta}_{\Pi}(x;t):=( \hat{\theta}_{\Pi}(x;t),\ldots, \hat{\theta}_{\Pi,n}(x;t))$.

The first lemma reduces bounding $R(\theta, q_{\Pi})$ to bounding $R_{\mathrm{e}}(\theta,\hat{\theta}_{\Pi})$.
The proof is given in Section \ref{section: proofs for auxiliary lemmas}.
\begin{lemma}
\label{lemma: Kullback--Leibler formula}
Fix a prior $\Pi$ of $\theta$.
If $\hat{\theta}_{\Pi}(x;t)$ based on $\Pi$ is strictly larger than $0$ for any $x\in\mathbb{N}^{n}$ and any $t\in(r,1+r)$,
then, we have
\begin{align*}
R(\theta,q_{\Pi})=\int_{r}^{r+1} \frac{R_{\mathrm{e}}(t\theta,t\hat{\theta}_{\Pi}(\cdot;t))}{t} dt.
\end{align*}
\end{lemma}

The second and third lemma display useful formulae for Poisson random variables. The proofs are easy and we omit them.
\begin{lemma}
\label{lemma: exp of inverse}
Let $X_{1}$ be a random variable from the Poisson distribution with mean $\lambda$. Then, we have
\[\Ep_{\lambda}\left[\frac{1}{X_{1}+1}\right] = \frac{1- \mathrm{e}^{-\lambda}}{\lambda}.\]
\end{lemma}

\begin{lemma}
\label{lemma: Bennett}
Let $X_{1}$ be a random variable from the Poisson distribution with mean $\lambda$. Then, we have
\[
\Pr ( X_{1} - \lambda  \leq - x ) \leq \exp\left(-\frac{x^{2}}{2\lambda}\right), \ 0 \leq x \leq \lambda.
\]
\end{lemma}

\subsection{Proof of Theorem \ref{theorem: exact minimaxity within sparse Poisson models}}
\label{subsection: Proof of theorem exact minimaxity within sparse Poisson models}

\subsubsection*{Step 1: Lower bound on $\mathcal{R}(\Theta[s_{n}])$}

The Bayes risk minimization with respect to \textit{block-independent} priors
will give a lower bound on $\mathcal{R}(\Theta[s_{n}])$.
Let $\Pi_{\mathrm{B},\nu}(d\theta)$ with $\nu>0$ be a \textit{block-independent} prior  built as follows:
divide $\{1,2,\ldots,n\}$ into contiguous blocks $\{B_{j}: j = 1 , 2 , \ldots , s_{n}\} $ with each length $m_{n}:=\lfloor \eta_{n}^{-1} \rfloor$.
In each block $B_{j}$,
draw $(\theta_{1+m_{n}(j-1)},\allowbreak \ldots,\theta_{m_{n}j})$ independently according to a single spike prior with spike strength $\nu>0$,
where a single spike prior with spike strength $\nu>0$ is the distribution of $\nu e_{I}$
with a uniformly random index $I\in\{1,\ldots,m_{n}\}$ and a unit length vector $e_{i}$ in the $i$-th coordinate direction.
Finally, set $\theta_{i}=0$ for the remaining $n-m_{n}s_{n}$ components.

Start with deriving the explicit form of $\hat{\theta}_{\Pi_{\mathrm{B},\nu}}$.
Let $\mathcal{X}_{j}:= \{ x^{(j)} =(x_{1},\ldots,x_{m_{n}}) : \|x^{(j)} \|_{0}\leq 1 \} $ ($j=1,2,\ldots,s_{n}$).
Observe that the Bayes formula yields, for $j=1,\ldots,s_{n}-1$ and for $i=1+m_{n}(j-1), \ldots, m_{n}j$,
\begin{align}
\hat{\theta}_{\Pi_{\mathrm{B},\nu}, i}(x^{(j)})
&=\frac{\sum_{k=1}^{m_{n}} \int \prod_{l\neq i}\{\theta_l^{x_{l}}\} \theta_i^{x_{i}+1} d\delta_{0}(\theta_1)\cdots d\delta_{\nu}(\theta_k) \cdots d\delta_{0}(\theta_{m_n} ) }{ \sum_{k=1}^{m_{n}} \int \prod_{l=1}^{m_{n}} \theta_l^{x_{l}} d\delta_{0}(\theta_1)\cdots d\delta_{\nu}(\theta_k) \cdots d\delta_{0}(\theta_{m_n} ) }.
\label{Bayes estimator}
\end{align}
This implies that
for each $j=1,\ldots,s_{n}$ and for each $x^{(j)} \in \mathcal{X}_{j}$,
\begin{align}
\hat{\theta}_{\Pi_{\mathrm{B},\nu}, i}(x^{(j)})=
\begin{cases} 
\nu / m_{n} & \text{if $\|x^{(j)}\|_{0}=0$,} \\ 
\nu & \text{if $x^{(j)}_{i}\neq 0$ and $x^{(j)}_{k}=0$ for $k\neq i$,} \\
0   & \text{if otherwise,}
\end{cases}
\label{explicit form}
\end{align}
as well as that for $j=1,\ldots,s_{n}$, 
\begin{align}
\sum_{i=1+m_{n}(j-1)}^{m_{n}j} \hat{\theta}_{\Pi_{\mathrm{B},\nu}, i}(x^{(j)})= \nu
\text{ for } x^{(j)} \text{ such that }\|x^{(j)} \|_{0} \le 1.
\label{sum explicit form}
\end{align}

Fix $\theta$ in the support of $\Pi_{\mathrm{B},\nu}$. 
Then, we have
\begin{align*}
R_{\mathrm{e}}(t\theta,t\hat{\theta}_{\Pi_{\mathrm{B},\nu}})
&=
\Ep_{t\theta}\sum_{i=1}^{n}t\theta_{i}\log \frac{\theta_{i}}{\hat{\theta}_{\Pi_{\mathrm{B},\nu},i}(X)}
-t\sum_{i=1}^{n}(\theta_{i} - \Ep_{t\theta}[\hat{\theta}_{\Pi_{\mathrm{B},\nu},i}(X)])
\\
&=\sum_{j=1}^{s_{n}}\sum_{i=1+m_{n}(j-1)}^{m_{n}j}
\Ep_{t\theta}t\theta_{i}\log \frac{\theta_{i}}{\hat{\theta}_{\Pi_{\mathrm{B},\nu},i}(X)},
\end{align*}
where the second equality holds from (\ref{sum explicit form}).
Letting $i(j)\in \{1+m_{n}(j-1) , \ldots , m_{n}j \}$ denote the index with $\theta_{i(j)}=\nu$,
we further have
\begin{align}
R_{\mathrm{e}}(t\theta,t\hat{\theta}_{\Pi_{\mathrm{B},\nu}})
\ge \sum_{j=1}^{s_{n}}\Pr(X_{i(j)}=0)t\nu\log \frac{\nu}{\nu / m_{n}}
= s_{n} \mathrm{e}^{-t\nu}t\nu\log m_{n}.\nonumber
\end{align}
This, together with Lemma \ref{lemma: Kullback--Leibler formula}, gives
\begin{align*}
R(\theta,q_{\Pi_{\mathrm{B},\nu}})
= \int_{r}^{r+1} \frac{R_{\mathrm{e}}(t\theta,t\hat{\theta}_{\Pi_{\mathrm{B},\nu}})}{t}dt
\geq \{\mathrm{e}^{-r\nu}-\mathrm{e}^{-(r+1)\nu}\}s_{n}\log  m_{n}.
\end{align*}
Taking expectation of $R(\theta,q_{\Pi_{\mathrm{B},\nu}})$ with respect to $\Pi_{\mathrm{B},\nu}$
yields
\begin{align*}
\mathcal{R}(\Theta[s_{n}]) 
&\geq \inf_{\hat{q}}\int R(\theta,\hat{q}) d\Pi_{\mathrm{B},\nu}(\theta)
= \int R(\theta,q_{\Pi_{\mathrm{B},\nu}})d\Pi_{\mathrm{B},\nu}(\theta)\\
&\geq \{\mathrm{e}^{-r\nu}-\mathrm{e}^{-(r+1)\nu}\}s_{n}\log  m_{n}.
\end{align*}
Maximizing the rightmost hand side in the above inequality with respect to $\nu$
presents the desired lower bound
$\mathcal{R}(\Theta[s_{n}]) \geq \mathcal{C} s_{n}\log  m_{n}$, which completes Step 1.

\subsubsection*{Step 2: Upper bound on $\mathcal{R}(\Theta[s_{n}])$}

Let $\Pi$ be an i.i.d.~prior $\Pi$ and consider the coordinate-wise Kullback--Leibler risk of the Bayes predictive density
$q_{\Pi}$:
\begin{align*}
\rho(\lambda) :=
\Ep_{\lambda} \log \left[\frac{\exp(-\lambda)\lambda^{Y_{1}}/Y_{1}!}{q_{\Pi^{*}}(Y_{1}\mid X_{1})}\right],\
\lambda>0,
\end{align*}
where
$q_{\Pi}(y_{i}\mid x_{i})$ is the marginal distribution of $q_{\Pi}$.
Consider the following high-level condition on $\Pi$:
\begin{condition}\label{posterior condition}
There exist constants $K\ge1$, $C_{1},C_{2},C_{3},C_{4}>0$, $C_{5}<K$, $C_{6}>0$ not depending on $n$ for which we have
\begin{enumerate}
\item[(P1)] $C_{1}\eta_{n} \le \hat{\theta}_{\Pi,i}(X_{i};t) \le C_{2}\eta_{n}$ 
for $X_{i}=0$;
\item[(P2)] $C_{3}\le \hat{\theta}_{\Pi,i}(X_{i};t) \le C_{4}$ 
for $1 \le X_{i} \le K$;
\item[(P3)] $0< (X_{i}-C_{5})/t \le \hat{\theta}_{\Pi,i}(X_{i};t) \le (X_{i}+C_{6})/t$ for $K<X_{i}$.
\end{enumerate}
\end{condition}
Under this condition, it will be shown that
\begin{itemize}
\item $\rho(0)= O(\eta_{n})$;
\item $\sup_{\lambda>0}\rho(\lambda) \leq ( \mathcal{C} +o(1))s_{n}\log \eta_{n}^{-1}$,
\end{itemize}
from which we will conclude
\begin{align*}
\mathcal{R}(\Theta[s_{n}]) 
\le
\sup_{\theta\in\Theta[s_{n}]}R(\theta,q_{\Pi})
&= (n-s_{n})\rho(0) + s_{n}\sup_{\lambda>0}\rho(\lambda) \\
&\leq (n-s_{n})O(\eta_{n}) + (\mathcal{C} + o(1)) s_{n} \log \eta_{n}^{-1}.
\end{align*}
Note that for $\Pi[\eta_{n},\kappa]$ with $\kappa>0$,
the Bayes formula gives
\begin{align*}
\hat{\theta}_{\Pi[\eta_{n},\kappa],1}(x_{1};t) 
=
\frac{ 0^{x_{1}+1} + \eta_{n}\Gamma(x_{1}+\kappa+1)/t^{x_{1}+\kappa+1} }
{ 0^{x_{1}} + \eta_{n}\Gamma(x_{1}+\kappa)/t^{x_{1}+\kappa}  }
=
\begin{cases}
\frac{\eta_{n}\Gamma(\kappa+1)/t^{\kappa+1}}{1+\eta_{n}\Gamma(\kappa)/t^{\kappa}} ,
&\text{$x_{1}=0$},\\
 \frac{x_{1}+\kappa}{t},
&\text{$x_{1}\geq 1$},
\end{cases}
\end{align*}
thereby implying $\Pi[\eta_n,\kappa]$ satisfies Condition \ref{posterior condition}.
For $\lambda>0$ and $t\in (r,r+1)$,
let $ \hat{\rho}(\lambda,x_{1};t) :=  t\lambda \log \{\lambda / \hat{\theta}_{\Pi,1}(x_{1};t)\}
- t\lambda + t\hat{\theta}_{\Pi,1}(x_{1};t).$
Lemma \ref{lemma: Kullback--Leibler formula} gives
\begin{align}
\begin{split}
\rho(\lambda)
&\leq \underbrace{\int_{r}^{r+1}
\mathrm{E}_{t\lambda}1_{X_{1}=0}
\Big{\{}
\lambda\log\frac{\lambda}{\hat{\theta}_{\Pi,1}(X_{1};t)}
-\lambda + \hat{\theta}_{\Pi,1}(X_{1};t)
\Big{\}}
dt}_{=:A_{1}}
\\
&\quad
+
\underbrace{\int_{r}^{r+1}
\mathrm{E}_{t\lambda}1_{1\le X_{1}\le K}
\Big{\{}
\lambda\log\frac{\lambda}{\hat{\theta}_{\Pi,1}(X_{1};t)}
-\lambda + \hat{\theta}_{\Pi,1}(X_{1};t)
\Big{\}}
dt}_{=:A_{2}}
\\
&\quad + \underbrace{\int_{r}^{r+1} 
\mathrm{E}_{t\lambda}1_{K< X_{1}}
\Big{\{}
\lambda\log\frac{\lambda}{\hat{\theta}_{\Pi,1}(X_{1};t)}
-\lambda + \hat{\theta}_{\Pi,1}(X_{1};t)
\Big{\}}
dt}_{=:A_{3}}.
\end{split}
\label{bounding rho}
\end{align}
From (P1) in Condition \ref{posterior condition},
we have
\begin{align}
A_{1} \le \{\mathrm{e}^{-r\lambda}-\mathrm{e}^{-(r+1)}\}
\{\log \eta_{n}^{-1}+\log C_{1}^{-1}+\log \lambda\} + C_{2}\eta_{n}.
\label{bounding the first integral}
\end{align}
From (P2) in Condition \ref{posterior condition},
we get
\begin{align}
A_{2} \le \sum_{k=1}^{K}\frac{(r+1)^{k}\lambda^{k}}{k!}\mathrm{e}^{-r\lambda}
\{ \lambda\log \lambda + \lambda \log C_{3}^{-1} + C_{4} \}.
\label{bounding the second integral}
\end{align}
To bound $A_{3}$, for $r\in (r,r+1)$,
we write the integrand in $A_{3}$ as
\begin{align*}
&\mathrm{E}_{t\lambda}1_{K< X_{1}}
\Big{\{}
\lambda\log\frac{\lambda}{\hat{\theta}_{\Pi,1}(X_{1};t)}
-\lambda + \hat{\theta}_{\Pi,1}(X_{1};t)
\Big{\}}\\
&=
\underbrace{ 
\mathrm{E}_{t\lambda}1_{K< X_{1}}
\lambda\log \{\lambda / \hat{\theta}_{\Pi,1}(X_{1};t) \}
}_{=:A_{3,1}}
+
\underbrace{ 
\mathrm{E}_{t\lambda}1_{K< X_{1}}
[-\lambda + \hat{\theta}_{\Pi,1}(X_{1};t)]
}_{=:A_{3,2}}.
\end{align*}
Lemma \ref{lemma: exp of inverse}, together with Jensen's inequality, yields 
\begin{align*}
A_{3,1} &= 
\mathrm{E}_{t\lambda}1_{K< X_{1}} \left[ \lambda \log \frac{t\lambda}{X_{1}+1}\right]
+
\mathrm{E}_{t\lambda}1_{K< X_{1}} \left[ \lambda \log \frac{X_{1}+1}{X_{1}-C_{5}} \right]\\
&\le \lambda (1-\mathrm{e}^{-t\lambda}) 
+ \sum_{k=0}^{K}\frac{(r+1)^{k}\lambda^{k} }{k!}\mathrm{e}^{-r\lambda} \lambda |\log t\lambda|
+ \mathrm{E}_{t\lambda}1_{K< X_{1}} \left[ \lambda \log \frac{X_{1}+1}{X_{1}-C_{5}} \right]\\
&\le  \underbrace{\sum_{k=0}^{K}\frac{(r+1)^{k}\lambda^{k+1} }{k!}\mathrm{e}^{-r\lambda} 
\{\max\{|\log r\lambda|,| \log (r+1)\lambda|\} \}}_{=:F_{1}(\lambda,r)}
 +  \underbrace{\mathrm{E}_{t\lambda}1_{K< X_{1}} \left[ \lambda \log \frac{X_{1}+1}{X_{1}-C_{5}} \right]}_{=:A_{3,1,1}}.
\end{align*}
If $C_{5}\le -1$, then $A_{3,1,1}$ is bounded above by 0.
So, we can assume $C_{5}>-1$.
Observe that there exists some $c_{1}>1$ depending only on $K$ and $C_{5}$ such that we have $(X_{1}+1) / (X_{1}-C_{5}) < c_{1}$ for $X_{1}>K$.
This gives
\begin{align*}
A_{3,1,1} \le \lambda \log c_{1}.
\end{align*}
This bound is crude for large $\lambda$, and we consider another bound for large $\lambda$.
Take $\lambda^{\circ}$ in such a way that 
$(r\lambda^{\circ})-(r\lambda^{\circ})^{3/4}-C_{5}>0$ and $r\lambda^{\circ}>1$.
Then Lemma \ref{lemma: Bennett} yields, for $\lambda>\lambda^{\circ}$,
\begin{align*}
A_{3,1,1} \le \mathrm{e}^{-(r\lambda)^{1/2}/2}\lambda\log c_{1}
+\lambda \log \left\{1+\frac{1+C_{5}}{(r\lambda)-(r\lambda)^{3/4}-C_{5}}\right\}.
\end{align*}
Thus we obtain $A_{3,1,1} \le F_{2}(\lambda,r)$ with
\begin{align*}
F_{2}(\lambda,r) := 
\begin{cases} 
\lambda \log c_{1}, & \lambda \le \lambda^{\circ}, \\ 
\mathrm{e}^{-(r\lambda)^{1/2}/2}\lambda\log c_{1}
+\lambda \log \left\{1+\frac{1+C_{5}}{(r\lambda)-(r\lambda)^{3/4}-C_{5}}\right\}, 
& \lambda > \lambda^{\circ},
\end{cases}
\end{align*}
which, together with simple bounds on $A_{3,1,1}$ and $A_{3,2}$, gives
\begin{align}
A_{3} \le F_{1}(\lambda,r) +  F_{2}(\lambda,r)
+ \sum_{k>K}\frac{(t\lambda)^{k}\mathrm{e}^{-t\lambda}(k+C_{6})}{r( k!)}.
\label{bounding the third integral}
\end{align}

Taking the limit as $\lambda\to 0$ in the right hand sides of (\ref{bounding the first integral}), (\ref{bounding the second integral}),
and (\ref{bounding the third integral}) gives $\rho(0) = O(\eta_{n})$.
Maximizing upper bounds in (\ref{bounding the first integral}), (\ref{bounding the second integral}),
and (\ref{bounding the third integral})
with respect to $\lambda$ yields
$\sup_{\lambda>0} \rho(\lambda) \le (\mathcal{C}+ o(1))\log \eta_{n}^{-1}$.
Thus we obtain the desired upper bound
\[\mathcal{R}(\Theta[s_{n}]) \leq (n-s_{n})O(\eta_{n}) + (\mathcal{C}+o(1)) s_{n} \log \eta_{n}^{-1},\]
which completes the proof.
\qed


\subsection{Proof of Theorem \ref{theorem: adaptive}}
\label{subsection: proof of theorem adaptive}

Let $\bar{p} := \sum_{j=1}^{s_{n}} (1-\mathrm{e}^{-r\theta_{[j]}})/s_{n}$,
where $\theta_{[j]}$ ($j=1,\ldots,s_{n}$) denotes the $j$-th largest component of $\{\theta_{i}:i=1,\ldots,n\}$.

\subsubsection{Supporting lemma: properties of $\hat{s}_{n}$}

We start with summarizing the behaviour of $\hat{s}_{n}$ which is an important ingredient of the proof.
The proof is given in Section \ref{section: proofs for auxiliary lemmas}.
\begin{lemma}
\label{lemma: properties of hats_n}
The following hold for $\theta\in\Theta[s_{n}]$:
\begin{itemize}
\item[(a)] $\hat{s}_{n}\ge 1$;
\item[(b)] $\max\{ \Ep_{\theta}|\hat{s}_{n}/ s_{n} - 1 |, \Ep_{\theta} | \hat{s}_{n} /s_{n} -1 |^{2} \} \le 3$;
\item[(c)] If $s_{n} \ge 4$ and $\theta_{[1]} \ge 1/\sqrt{\log s_{n}}$,
then, for sufficiently large $n$ depending only on $r$,
\[\Ep_{\theta}  \log (s_{n} / \hat{s}_{n}) \le c_{1}\max\{\sqrt{\log s_{n}}, s_{n}\exp(-c_{2}s_{n}/\log s_{n} )\}\]
with positive constants $c_{1}$ and $c_{2}$ depending only on $r$.
\end{itemize}
\end{lemma}

This lemma indicates that $\hat{s}_{n}$ is not so far from $s_{n}$.
In the sparse region (i.e.,~$s_{n}=o(n^{1/2})$),
properties (a) and (b) are sufficient to prove Theorem \ref{theorem: adaptive}.
In the dense region (i.e.,~$s_{n}>c n^{1/2}$ for any $c>0$),
property (c) is additionally required.

\subsubsection{Proving Theorem \ref{theorem: adaptive}}
Decompose the difference between the Kullback--Leibler divergences in such a way that
\begin{align}
R(\theta,q_{\Pi[\hat{\eta}_{n},\kappa]})
-R(\theta,q_{\Pi[\eta_{n},\kappa]})
&=
\sum_{i=1}^{n}
\underbrace{\Ep_{\theta}\log\left\{\frac{
q_{\Pi[\eta_{n},\kappa],i}(Y_{i}\mid X_{i})}{
q_{\Pi[\hat{\eta}_{n},\kappa],i}(Y_{i}\mid X_{i})}\right\}}_{:=D_{i}}
\nonumber\\
&=
\sum_{i\in\mathcal{A}} D_{i}
+
\sum_{i\not\in\mathcal{A}} D_{i},
\label{Decomposition of KL with plug in}
\end{align}
where
for $\theta\in\Theta[s_{n}]$,
let $\mathcal{A}:=\mathcal{A}(\theta)=\{i: \theta_{i}\neq 0\}$.
The following three steps give an upper bound on the rightmost side of (\ref{Decomposition of KL with plug in}).
In all steps, we use the following expression of $q_{\Pi[h,\kappa]}$:
\begin{align}
&q_{\Pi[h,\kappa]}(y\mid x) \nonumber\\
&= \prod_{i=1}^{n}
\left\{ \omega_{i} \delta_{0}(y_{i}) 
+ (1-\omega_{i}) \begin{pmatrix} x_{i}+y_{i} +\kappa -1 \\ y_{i} \end{pmatrix} \left(\frac{r}{r+1}\right)^{x_{i}+\kappa}\left(1-\frac{r}{r+1}\right)^{y_{i}} \right\},
\label{eq: explicit form}
\end{align}
where
\begin{align*}
\omega_{i}:=
\begin{cases}
1\big{/}\big{\{}1+h \Gamma(\kappa)/r^{\kappa}\big{\}} & \text{ if } x_{i}=0, \\
0 & \text{ if }x_{i}\geq 1.
\end{cases}
\end{align*}

\textit{Step 1: Bounding $D_{i}$ for $i\not\in\mathcal{A}$}.
From (\ref{eq: explicit form}), we have
\begin{align}
D_{i} &=\Ep_{\theta}\log \left\{ \frac{1+\hat{\eta}_{n}\Gamma(\kappa)/r^{\kappa} }{1+\eta_{n}\Gamma(\kappa)/r^{\kappa}}\right\}
+\Ep_{\theta}\log \left[ \frac{1+\{\eta_{n}\Gamma(\kappa)/r^{\kappa}\}\{r/(r+1)\}^{\kappa} }
{1+\{\hat{\eta}_{n}\Gamma(\kappa)/r^{\kappa}\}\{r/(r+1)\}^{\kappa}}\right].
\label{Case i not in A decomposition}
\end{align}
Since $\log(1+x) \le x$ for $x>0$, we have
\begin{align}
\Ep_{\theta}\log \left\{ \frac{1+\hat{\eta}_{n}\Gamma(\kappa)/r^{\kappa} }{1+\eta_{n}\Gamma(\kappa)/r^{\kappa}}\right\}
&\le
\Ep_{\theta}\log \left\{ 1 + \frac{(\hat{\eta}_{n}-\eta_{n})\Gamma(\kappa)/r^{\kappa}}{1+\eta_{n}\Gamma(\kappa)/r^{\kappa}} \right\}
\nonumber\\
&\le
\Ep_{\theta}\log \left\{ 1 + \eta_{n} |\hat{s}_{n}/s_{n} -1 | \frac{\Gamma(\kappa)}{r^{\kappa}} \right\} 
\nonumber\\
&\le
\eta_{n}\frac{\Gamma(\kappa)}{r^{\kappa}}\Ep_{\theta}|\hat{s}_{n}/s_{n}-1|
\nonumber\\
&\le c_{1} \eta_{n} \text{ with }c_{1}:=3\frac{\Gamma(\kappa)}{r^{\kappa}}
\label{adaptive bound 1}
\end{align}
where the last inequality follows from Lemma \ref{lemma: properties of hats_n} (b).
Observe that
\[
\frac{\{\eta_{n}-\hat{\eta}_{n}\}\{\Gamma(\kappa)/(r+1)^{\kappa}\}}{1+\eta_{n}\{\Gamma(\kappa)/(r+1)^{\kappa}\}}
\leq \frac{\eta_{n}\Gamma(\kappa)/(r+1)^{\kappa}}
{1+\eta_{n}\Gamma(\kappa)/(r+1)^{\kappa}}
\leq \frac{\Gamma(\kappa)}{1+\Gamma(\kappa)}.
\]
Then, together with the inequality
\[
-\log(1-x) \leq \{1/(1-U)^{2}\}x^{2}+x \text{ for $0<x \leq U$ with some $0<U<1$},
\]
this observation gives
\begin{align}
\Ep_{\theta}&\log \left[ \frac{1+\{\eta_{n}\Gamma(\kappa)/r^{\kappa}\}\{r/(r+1)\}^{\kappa} }
{1+\{\hat{\eta}_{n}\Gamma(\kappa)/r^{\kappa}\}\{r/(r+1)\}^{\kappa}}\right]
\nonumber\\
&\le -\Ep_{\theta}
\log \left[1 - \frac{\{\eta_{n}-\hat{\eta}_{n}\}\{\Gamma(\kappa)/(r+1)^{\kappa}\}}{1+\eta_{n}\{\Gamma(\kappa)/(r+1)^{\kappa}\}}\right]\nonumber\\
&\le \eta_{n}\frac{\Gamma(\kappa)}{(r+1)^{\kappa}}\Ep_{\theta}|\hat{\eta}_{n}/\eta_{n}-1| 
+ \left[\{1+\Gamma(\kappa)\}\eta_{n}\frac{\Gamma(\kappa)}{(r+1)^{\kappa}}\right]^{2}\Ep_{\theta} | \hat{\eta}_{n} / \eta_{n} -1 |^{2}
\nonumber\\
&\le c_{2}\eta_{n} \text{ with } 
c_{2}:=3\frac{\Gamma(\kappa)}{(r+1)^{\kappa}}+3\left[\{1+\Gamma(\kappa)\}\frac{\Gamma(\kappa)}{(r+1)^{\kappa}}\right]^{2},
\label{adaptive bound 2}
\end{align}
where the last inequality follows from Lemma \ref{lemma: properties of hats_n} (b).
Combining (\ref{adaptive bound 1}) and (\ref{adaptive bound 2}) with (\ref{Case i not in A decomposition})
yields
\begin{align}
D_{i} \le c_{3} \eta_{n} \text{ with }c_{3}:=c_{1}+c_{2} \text{ for $i\not\in\mathcal{A}$}.
\label{Case i not in A}
\end{align}

\textit{Step 2: Bounding $D_{i}$ for $i\in \mathcal{A}$}.
Consider the following four cases:
(i) $X_{i}=0$, $Y_{i}=0$; 
(ii) $X_{i}\geq 1$, $Y_{i}=0$;
(iii) $X_{i}=0$, $Y_{i}\geq 1$;
(iv) $X_{i}\geq 1$, $Y_{i}\geq 1$.
In Case (i),
we have
\begin{align}
\log&\left\{
\frac{q_{\Pi[\eta_{n},\kappa],i}(Y_{i}\mid X_{i})}
{q_{\Pi[\hat{\eta}_{n},\kappa],i}(Y_{i}\mid X_{i})}\right\}\nonumber\\
&=\log \left\{\frac{1+\hat{\eta}_{n}\Gamma(\kappa)/r^{\kappa}}{1+\eta_{n}\Gamma(\kappa)/r^{\kappa}}
\right\}
+ 
\log \left[ \frac{1+\{\eta_{n}\Gamma(\kappa)/r^{\kappa}\}\{r/(r+1)\}^{\kappa} }
{1+\{\hat{\eta}_{n}\Gamma(\kappa)/r^{\kappa}\}\{r/(r+1)\}^{\kappa}}\right]\nonumber\\
&\le |\hat{s}_{n}/s_{n}-1| + \log [1+\Gamma(\kappa)/(r+1)^{\kappa}],
\label{Case i}
\end{align}
where we use the inequalities
\begin{align*}
\log \left\{\frac{1+\hat{\eta}_{n}\Gamma(\kappa)/r^{\kappa}}{1+\eta_{n}\Gamma(\kappa)/r^{\kappa}}
\right\}
\le
\log \left\{1+\frac{|\hat{\eta}_{n}-\eta_{n}|\Gamma(\kappa)/r^{\kappa}}{1+\eta_{n}\Gamma(\kappa)/r^{\kappa}}
\right\}
\le |\hat{s}_{n}/s_{n}-1|
\end{align*}
and
\begin{align*}
\log \left[ \frac{1+\{\eta_{n}\Gamma(\kappa)/r^{\kappa}\}\{r/(r+1)\}^{\kappa} }
{1+\{\hat{\eta}_{n}\Gamma(\kappa)/r^{\kappa}\}\{r/(r+1)\}^{\kappa}}\right]
&\le
\log \left[1+ \{\eta_{n}\Gamma(\kappa)/r^{\kappa}\}\{r/(r+1)\}^{\kappa} \right]
\nonumber\\
&\le \log \left[1+ \Gamma(\kappa)/(r+1)^{\kappa} \right].
\end{align*}
Similarly, 
we get the following evaluations:
In Case (ii), 
\begin{align}
\log\left\{\frac{
q_{\Pi[\eta_{n},\kappa],i}(Y_{i}\mid X_{i})}
{q_{\Pi[\hat{\eta}_{n},\kappa],i}(Y_{i}\mid X_{i})}
\right\}
&=0;
\label{Case ii}
\end{align}
In Case (iii), 
\begin{align}
\log\left\{
\frac{q_{\Pi[\eta_{n},\kappa]}(Y_{i}\mid X_{i})} 
{q_{\Pi[\hat{\eta}_{n},\kappa]}(Y_{i}\mid X_{i})}
\right\}
&\leq \log (s_{n}/\hat{s}_{n}) + |\hat{s}_{n}/s_{n}-1|;
\label{Case iii}
\end{align}
In Case (iv), 
\begin{align}
\log\left\{
\frac{
q_{\Pi[\eta_{n},\kappa]}(Y_{i}\mid X_{i})}
{q_{\Pi[\hat{\eta}_{n},\kappa]}(Y_{i}\mid X_{i})}
\right\}
&=0.
\label{Case iv}
\end{align}
From (\ref{Case i})-(\ref{Case iv}),
we have, for $i\in \mathcal{A}$,
\begin{align}
    D_{i} &\le 2\Ep_{\theta}|\hat{s}_{n}/s_{n}-1|
    +\log[1+\Gamma(\kappa)/(r+1)^{\kappa}]
    +\Ep_{\theta} [1_{X_{i}=0,Y_{i}\ge 1}\log(s_{n}/\hat{s}_{n})]\nonumber\\
    &\leq 6 +\log[1+\Gamma(\kappa)/(r+1)^{\kappa}]
    +\Ep_{\theta} [1_{X_{i}=0,Y_{i}\ge 1}\log(s_{n}/\hat{s}_{n})],
    \label{Case i in A}
\end{align}
where the last inequality follows from Lemma \ref{lemma: properties of hats_n} (b).

\textit{Step 3}.
Combining (\ref{Case i not in A}) and (\ref{Case i in A}) with (\ref{Decomposition of KL with plug in}) gives
\begin{align}
R(\theta, q_{\Pi[\hat{\eta}_{n},\kappa]})
&\le 
R(\theta, q_{\Pi[\eta_{n},\kappa]}) + c_{3}(n-s_{n})\eta_{n} + s_{n}\{6 +\log[1+\Gamma(\kappa)/(r+1)^{\kappa}]\} 
\nonumber\\
&\quad + \sum_{i\in \mathcal{A}}\Ep_{\theta} [1_{X_{i}=0,Y_{i}\ge 1}\log(s_{n}/\hat{s}_{n})] \nonumber\\
&\le  R(\theta,q_{\Pi[\eta_{n},\kappa]}) + c_{4}s_{n} + 
\sum_{i\in \mathcal{A}} \underbrace{\Ep_{\theta} [1_{X_{i}=0,Y_{i}\ge 1}\log(s_{n}/\hat{s}_{n})]}_{:=T_{i}},
\end{align}
where $c_{4}:=c_{3}+ 6+ \log[1+\Gamma(\kappa)/(r+1)^{\kappa}]$.
We will show 
\begin{align}
\sum_{i\in\mathcal{A}}T_{i}=o(s_{n}\log (n/ s_{n})).
\label{eq: bounding T}
\end{align}
Since the number of indices in $\mathcal{A}$ is bounded above from $s_{n}$,
it suffices to show $T_{i}=o(\log (n/s_{n}))$ uniformly in $\theta\in\Theta[s_{n}]$ and $i\in\mathcal{A}$.
First consider the case with $s_{n}=o(n^{1/2})$.
Since $\log (s_{n}/\hat{s}_{n}) \le \log s_{n}$ from Lemma 
\ref{lemma: properties of hats_n} (a),
we get
\begin{align}
T_{i} \le \log s_{n} = o( \log (n/s_{n}) ).
\label{eq: eval T1 1}
\end{align}
Next consider the case with $s_{n}>cn^{1/2}$ for any $c>0$.
Since $\log (s_{n}/\hat{s}_{n}) \le \log s_{n}$ from Lemma 
\ref{lemma: properties of hats_n} (a),
we get,
for $\theta\in\Theta[s_{n}]$ such that $\theta_{[1]} \le 1/\sqrt{\log s_{n}}$,
\begin{align}
T_{i}\le \Ep_{\theta} [1_{Y_{i}\ge 1}\log(s_{n}/\hat{s}_{n})] 
\le \left(1-\mathrm{e}^{-\theta_{[1]}}\right)\log s_{n}
\le \theta_{[1]}\log s_{n}
\le \sqrt{\log s_{n}}.
\label{eq: eval T1 2}
\end{align}
Using Lemma \ref{lemma: properties of hats_n} (c),
we have,
for $\theta\in\Theta[s_{n}]$ such that $\theta_{[1]} \ge 1/\sqrt{\log s_{n}}$,
\begin{align}
T_{i}\le \Ep_{\theta} [\log(s_{n}/\hat{s}_{n})] \le
c_{5}\sqrt{\log s_{n}},
\label{eq: eval T1 3}
\end{align}
where $c_{5}$ is the constant depending only on $r$
appearing in Lemma \ref{lemma: properties of hats_n} (c).
From (\ref{eq: eval T1 1})-(\ref{eq: eval T1 3}), we obtain (\ref{eq: bounding T})
and thus complete the proof.
\qed

\section{Proofs of auxiliary lemmas}
\label{section: proofs for auxiliary lemmas}
This section provides proofs of Lemmas \ref{lemma: Kullback--Leibler formula} and \ref{lemma: properties of hats_n}.

\begin{proof}[Proof for Lemma \ref{lemma: Kullback--Leibler formula}]
Let $\Pi$ be a prior of $\theta$
and
suppose that the Bayes estimate $\hat{\theta}_{\Pi}(x;t)$ based on $\Pi$
is strictly larger than $0$ for any $x\in\mathbb{N}^{n}$ and any $t\in(r,r+1)$.

Observe that the Kullback--Leibler risk is decomposed as 
\begin{align*}
R(\theta, q_{\Pi}) = \Ep_{\theta}\left[\log \left\{\frac{s(Y,X\mid \theta)}{ s_{\Pi}(Y,X)}\right\} \right] - \Ep_{\theta}\left[\log \left\{\frac{p(X\mid \theta)}{p_{\Pi}(X)}\right\}\right],
\end{align*}
where $s(y,x\mid \theta)=p(x\mid\theta)q(y\mid\theta)$, $s_{\Pi}(y,x):=\int s(y,x\mid \theta) d\Pi(\theta)$,
and
$p_{\Pi}(x):= \int p(x\mid\theta) d\Pi(\theta)$.
For $z\in\mathbb{N}^{n}$ and $t\in(r,r+1)$,
let $p(z \mid \theta; t):=\prod_{i=1}^{n}\mathrm{e}^{-t\theta_{i}}(t\theta_{i})^{z_{i}-1}/z_{i}!$
and
let $p_{\Pi}(z;t):= \int p(z\mid\theta ;t)\Pi(d\theta).$
From the sufficiency reduction, we have
\begin{align*}
\Ep_{\theta}\left[\log\left\{ \frac{s(Y,X\mid\theta)}{s_{\Pi}(Y,X)}\right\}\right]= \Ep_{\theta}\left[\log \left\{\frac{p( X+Y \mid \theta ; r+1) }{ p_{\Pi}(X+Y;r+1)}\right\}\right].
\end{align*}
Introducing the random variable $Z_{t}$ from $\otimes_{i=1}^{n}\mathrm{Po}(t\theta_{i})$ ($t\in(r,r+1)$),
we get
\begin{align*}
R(\theta,q_{\Pi})=\int_{r}^{r+1} \frac{d}{dt}\Ep\left[\log \left\{ \frac{p(Z_{t}\mid \theta ; t)}{ p_{\Pi}(Z_{t}; t)} \right\}\right] dt.
\end{align*}
Therefore, it suffices to show
\begin{align}
\frac{d}{dt} \Ep\left[\log \left\{ \frac{p(Z_{t}\mid \theta ; t)}{ p_{\Pi}(Z_{t}; t) }\right\}\right] = \frac{R_{\mathrm{e}}(t\theta,t\hat{\theta}_{\Pi}(\cdot;t))}{t},\ t\in (r,r+1).
\label{eq: derivative of KL risk}
\end{align}
Differentiating $\Ep[\log \{ p(Z_{t}\mid \theta ; t)/ p_{\Pi}(Z_{t}; t) \}]$ with respect to $t$ yields
\begin{align}
\Ep[\log \{ p(Z_{t}\mid \theta ; t)/ p_{\Pi}(Z_{t}; t) \}] 
&=
\Ep\left[\left\{\frac{d\log p(Z_{t}\mid \theta ; t)}{dt}\right\} \log\left\{\frac{p(Z_{t}\mid\theta;t)}{p_{\Pi}(Z_{t};t)}\right\}\right]
\nonumber
\\
&\quad+\Ep\left[\frac{d\log p(Z_{t}\mid \theta ; t)}{dt}\right]
-\Ep\left[\frac{d\log p_{\Pi}(Z_{t};t)}{dt}\right].
\label{eq: differentiating}
\end{align}
Let $e_{i}$ be the unit length vector in the $i$-th coordinate direction ($i=1,\ldots,n$).
Together with the simple fact that
\[p_{\Pi}(Z_{t}+e_{i};t)/p_{\Pi}(Z_{t};t)=\hat{\theta}_{\Pi,i}(Z_{t};t),\]
Hudson's lemma 
($\Ep[\sum_{i=1}^{n}(Z_{t,i}-1)f(Z_{t})]=\Ep[\sum_{i=1}^{n}t\theta_{i}f(Z_{t}+e_{i})]$ for any function $f:\mathrm{N}^{n}\to\R$)
yields
\begin{align}
\Ep&[\{d\log p(Z_{t}\mid \theta ; t)/dt\} \log\{p(Z_{t}\mid\theta;t)/p_{\Pi}(Z_{t};t)\}]
\nonumber\\
&= \Ep\left[\left\{\sum_{i=1}^{n}\frac{Z_{t,i}-1-t\theta_{i})}{t}\right\}\log\{p(Z_{t}\mid\theta;t)/p_{\Pi}(Z_{t};t)\} \right]
\nonumber\\
&= \Ep \sum_{i=1}\theta_{i}\big{[}\log\{ p(Z_{t}+e_{i}\mid \theta;t) / p(Z_{t}\mid \theta;t)\} - \log\{p_{\Pi}(Z_{t}+e_{i};t) / p_{\Pi}(Z_{t})\}\big{]}
\nonumber\\
&= \Ep \sum_{i=1}^{n}\theta_{i}\log\{ \theta_{i}/\hat{\theta}_{\Pi,i}(Z_{t};t) \}.
\label{identity 1}
\end{align}
Similarly,
the identity $(d/dt)\log p_{\Pi}(x;t) = -\sum_{i=1}^{n}\{\hat{\theta}_{\Pi,i}(x;t)-x_{i}+1\}$
gives
\begin{equation}
\begin{split}
\Ep\left[\frac{d}{dt}\{\log p(Z_{t}\mid \theta ; t)\}\right] &= \Ep\left[\sum_{i=1}^{n}\frac{Z_{t,i}-1-t\theta_{i}}{t}\right]
\text{ and }\\
\Ep\left[\frac{d}{dt}\{\log p_{\Pi}(Z_{t};t)\}\right] &= \Ep\left[-\sum_{i=1}^{n}\frac{\hat{\theta}_{\Pi,i}(Z_{t};t) -Z_{t,i}+1}{t} \right].
\end{split}
\label{identity 2}
\end{equation}
Combining identities (\ref{identity 1}) and (\ref{identity 2}) with (\ref{eq: differentiating}) gives (\ref{eq: derivative of KL risk}), which completes the proof.
\end{proof}

\vspace{5mm}

\begin{proof}[Proof of Lemma \ref{lemma: properties of hats_n}]
Property (a) is obvious by definition.
Consider the bias of $\hat{s}_{n}-s_{n}$.
Decompose $\#\{i: X_{i}\geq 1\}$ in such a way that
$\#\{i:X_{i}\geq 1\}=\sum_{j=1}^{s_{n}}Z_{j}$
with independent Bernoulli random variable $Z_{j}$ ($j=1,\ldots,s_{n}$) having the success probability $1-\exp(-r\theta_{[j]})$.
This decomposition gives
\begin{align}
-\sum_{j=1}^{s_{n}}\mathrm{e}^{-r\theta_{[j]}}
\le
\Ep_{\theta}(\hat{s}_{n}-s_{n})
\le 0.
\label{proof bias}
\end{align}
Consider the variance of $\hat{s}_{n}-s_{n}$.
Since
\begin{align*}
    -1+\sum_{j=1}^{s_{n}}(Z_{j}-\Ep Z_{j})
    \leq
    (\hat{s}_{n}-\Ep_{\theta}\hat{s}_{n})
    \leq
    1+\sum_{j=1}^{s_{n}}(Z_{j}-\Ep Z_{j}),
\end{align*}
we have
\begin{align}
\Ep_{\theta}\left|\frac{\hat{s}_{n}-\Ep_{\theta}\hat{s}_{n}}{s_{n}}\right|^{2}
\leq \frac{1}{s^{2}_{n}} + \sum_{j=1}^{s_{n}}\frac{\mathrm{e}^{-r\theta_{[j]}}\left(1-\mathrm{e}^{-r\theta_{[j]}}\right)}{s_{n}^{2}}.
\label{proof variance}
\end{align}
Given that $s_{n} \ge 1$,
we get $\max\{\Ep|\hat{s}_{n}/s_{n}-1|,\Ep|\hat{s}_{n}/s_{n}-1|^{2}\} \le 3$
from (\ref{proof bias}) and (\ref{proof variance}),
which shows (b).

By the layer-cake representation and since $\sum_{j=1}^{s_{n}}Z_{j} \le \hat{s}_{n}$, we have
\begin{align*}
\Ep_{\theta} \log \frac{\eta_{n}}{\hat{\eta}_{n}}
&= \int_{0}^{\log s_{n}} \Pr \left( \log \frac{\eta_{n}}{\hat{\eta}_{n}} > x\right)dx\nonumber\\
&=  \int_{1/s_{n}}^{1} \Pr \left( \hat{s}_{n} < \beta s_{n} \right) \frac{d\beta}{\beta}
\le  \int_{1/s_{n}}^{1} \Pr \left( \sum_{j=1}^{s_{n}}Z_{j} < \beta s_{n}  \right)\frac{d\beta}{\beta}.
\end{align*}
Together with the Hoeffding inequality, this yields
\begin{align}
\Ep_{\theta} \log \frac{\eta_{n}}{\hat{\eta}_{n}}
&\le \int_{1/s_{n}}^{1} \Pr \left( \sum_{j=1}^{s_{n}}Z_{j} - \Ep\left[\sum_{j=1}^{s_{n}}Z_{j}\right] < (\beta - \bar{p})s_{n}  \right)
\frac{d\beta}{\beta}
\nonumber\\
&\le \int_{1/s_{n}}^{1} \frac{1}{\beta}\exp\{ -2 s_{n} (\beta - \bar{p})^{2} \} d\beta
= \int_{1/s_{n}}^{1} \exp( f(\beta) ) d\beta,
\label{eq: expression using f}
\end{align}
where $f(\beta):= -2s_{n} (\beta -\bar{p})^{2} - \log \beta$.

We employ the following bound on $f(\beta)$ to obtain an upper bound on the right hand side in (\ref{eq: expression using f}).
\begin{lemma}\label{lemma: f}
If $\bar{p}^{2}> 1/s_{n}$ and $s_{n}\ge 4$,
then we have 
\[
f(\beta) \le \max\left[ f(1/s_{n}), f\left\{ (1/2)\left(\bar{p} + \sqrt{\bar{p}^{2} - 1/s_{n}}\right) \right\}\right]
\text{ for $1/s_{n} \le \beta \le 1$}.
\]
\end{lemma}
\begin{proof}[Proof of Lemma \ref{lemma: f}]
Observe that 
we have
\begin{align*}
f''(\beta) > 0   & \text{ for } 1/s_{n} \le \beta < 1/\sqrt{2s_{n}},\\
f''(\beta) \le 0 & \text{ for } 1/\sqrt{2s_{n}} \le \beta \le 1,
\end{align*}
and
\begin{align*}
f'(\beta) \ge 0 & \text{ for } \max\{1/s_{n},  (1/2)(\bar{p} - \sqrt{\bar{p}^{2}-1/s_{n}}) \} \le \beta < (1/2)(\bar{p} + \sqrt{\bar{p}^{2}-1/s_{n}}),\\
f'(\beta) \le 0 & \text{ for } (1/2)(\bar{p} + \sqrt{\bar{p}^{2}-1/s_{n}}) \le \beta \le 1,
\end{align*}
where the first two inequalities follow from $f''(\beta)  =  -4s_{n} + 1/\beta^{2}$,
and the last two inequalities follow from $f'(\beta)= -4s_{n}(\beta-\bar{p}) - 1/\beta$
and  $(1/2)( \bar{p} - \sqrt{\bar{p}^{2} - 1/s_{n}}) < 1/s_{n}$.
This observation gives the desired inequality.
\end{proof}
Go back to proving (c).
Observe that $\bar{p}^{2} > 1/s_{n}$ for $n\ge N$ with sufficiently large $N$ depending only on $r$,
because the assumption $\theta_{[1]}\ge 1/\sqrt{\log s_{n}}$ implies that
$\bar{p} \ge 1-\exp\{-r\theta_{[1]}\} \ge \tilde{c}_{1}/\sqrt{\log s_{n}}$
with $\tilde{c}_{1}$ depending only on $r$
for sufficiently large $n$ depending only on $r$.
Thus, Lemma \ref{lemma: f}, 
together with (\ref{eq: expression using f}), gives
\begin{align}
\Ep_{\theta} \log \frac{\eta_{n}}{\hat{\eta}_{n}} \le \exp\left[ \max\left\{ f(1/s_{n}) , f\left( \frac{\bar{p} + \sqrt{\bar{p}^{2}-1/s_{n} }}{2}\right) \right\}\right] \text{ for }n\ge N.
\label{eq: lemma final 1}
\end{align}
Simple calculations yield
\begin{align}
f(1/s_{n}) &= -2\tilde{c}_{1}^{2}s_{n}/\log s_{n} + 2\tilde{c}_{1}/\sqrt{\log s_{n}} -2/s_{n} + \log s_{n}\nonumber\\
&\le - \tilde{c}_{2} s_{n}/\log s_{n} +\log s_{n} + \tilde{c}_{3}
\label{eq: lemma final 2}
\end{align}
with constants $\tilde{c}_{2}$ and $\tilde{c}_{3}$ depending only on $\tilde{c}_{1}$,
as well as
\begin{align}
f\left\{ \left(\bar{p} + \sqrt{\bar{p}^{2}-1/s_{n}} \right)/2)\right\}
&=-\frac{s_{n}}{2}(\bar{p}-\sqrt{\bar{p}^{2}-1/s_{n}})^{2} -\log\frac{\bar{p}+\sqrt{\bar{p}^{2}-1/s_{n}}}{2}\nonumber\\
&=-\frac{1}{2s_{n}(\bar{p}+\sqrt{\bar{p}^{2}-1/s_{n}})^{2}} -\log\frac{\bar{p}+\sqrt{\bar{p}^{2}-1/s_{n}}}{2}\nonumber\\
&\le -1/(2s_{n}\bar{p}^{2}) -\log(\bar{p}/2)\nonumber\\
&\le -\tilde{c}_{4}(\log s_{n})/s_{n} + \log (\sqrt{\log s_{n}}) + \tilde{c}_{5}
\label{eq: lemma final 3}
\end{align}
with constants $\tilde{c}_{4}$ and $\tilde{c}_{5}$ depending only on $\tilde{c}_{1}$.
Combining (\ref{eq: lemma final 2}) and (\ref{eq: lemma final 3}) with (\ref{eq: lemma final 1}) 
completes the proof.
\end{proof}

\section{Discussion and conclusions}

We have studied asymptotic minimaxity in sparse Poisson sequence models.
We have presented Bayes predictive densities that are adaptive in the asymptotically minimax sense.

Proposition \ref{proposition: optimality of Cauchy or Pareto} shows that
spike-and-slab priors based on polynomially decaying slabs give asymptotically minimax predictive densities. 
This implies that the spike-and-slab approach is useful for sparse count data analysis.
At the same time, asymptotic minimaxity does not tell which slab prior is the best,
and selecting a single slab prior seems to require the other metrics, say, 
the computational complexity or the robustness.
From the viewpoint of computational complexity,
predictive densities based on our slab priors are easily implemented by exact sampling and are a good starting choice.
\cite{Hamuraetal} investigate Bayesian tail robustness and 
derive yet another interesting heavy-tailed prior in the global-local shrinkage literature.
It would also be helpful to understand more about similarities and differences between our work and \cite{Hamuraetal} from the predictive perspective.

We can consider more general sparsity in count data.
Quasi sparsity with a non-shrinking spike component is one such example.
If the value of the spike component is known,
a slight modification of our method would work.
If the value is unknown, further consideration would be necessary.

\section{Acknowledgement}

We are thankful to the Editor, the Associate Editor, and anonymous referees for their helpful comments.
We used the dataset of \textit{the number of crimes in Tokyo prefecture by town and type} 
\cite{Metropolitan}
by Tokyo metropolitan government and Tokyo Metropolitan Police Department.
We used OpenStreetMap \cite{OpenStreetMap} and CartoDB \cite{CartoDB} to creat Figure \ref{figure: Japanese crime}.
This research was supported by JST CREST (JPMJCR1763), MEXT KAKENHI (16H06533),
and JST KAKENHI (17H06570, 19K20222).

\bibliographystyle{plain}
\bibliography{SparsePoisson}

\newpage

\section*{Supplement to ``Minimax Predictive Density for Sparse Count Data"}

This supplemental material is organized as follows.
In this supplement, the numbering for theorems and propositions follows that of the main manuscript.
Appendix \ref{section: proofs for propositions in section 2} contains proofs of 
all propositions in Section \ref{section: predictive density estimation in sparse Poisson models}. 
Appendix \ref{section: proofs for propositions in section 3} contains proofs of 
all propositions in Section \ref{section: extensions}.
Appendix \ref{Appendix: supplemental experiments} provides additional simulation studies.
Appendix \ref{subsection: Discussion: MCAR} provides discussions on the results for MCAR settings.
Appendix \ref{rare mutation rates} applies our methods to exome sequencing data.
\appendix

\section{Proofs of propositions in Section \ref{section: predictive density estimation in sparse Poisson models}}
\label{section: proofs for propositions in section 2}

This section provides proofs of propositions in Section \ref{section: predictive density estimation in sparse Poisson models}.

\subsection{Proof of Proposition \ref{proposition: exact minimaxity for estimation}}
\label{subsection: proof of proposition exact minimaxity for estimation}


\textit{Lower bound on $\mathcal{E}(\Theta[s_{n}])$}:
The Bayes risk based on a block-independent prior $\Pi_{\mathrm{B},\nu}$ with $\nu>0$
is
\begin{align*}
\int R(\theta,q(\cdot\mid \hat{\theta}_{\Pi_{\mathrm{B},\nu}}))d\Pi_{\mathrm{B},\nu}(\theta)
&= s_{n}[ \mathrm{e}^{-r\nu} \nu \log \{\nu/(\nu \lfloor \eta_{n}\rfloor )\} +(1-\mathrm{e}^{-r\nu})\nu\log(\nu/\nu)] \\
&= s_{n} \nu\mathrm{e}^{-r\nu}\log\lfloor \eta_{n}^{-1} \rfloor
\end{align*}
and thus maximizing $\int R(\theta,q(\cdot\mid\hat{\theta}_{\Pi_{\mathrm{B},\nu}}))d\Pi_{\mathrm{B},\nu}(\theta)$ with respect to $\nu$ yields
\[\mathcal{E}(\Theta[s_{n}])\geq \mathrm{e}^{-1}r^{-1}s_{n}(1+o(1))\log\eta_{n}^{-1}.\]

\textit{Upper bound on $\mathcal{E}(\Theta[s_{n}])$}:
Consider the Bayes estimator based on $\Pi^{*}:=\Pi[\eta_{n},\kappa]$ with $\kappa>0$ given by
\begin{align*}
    \hat{\theta}_{\Pi^{*},i}(x_{i}) = \Bigg{\{}\begin{array}{ll} \eta_{n} \frac{\Gamma(\kappa+1)/r^{\kappa+1}}{1+\eta_{n}\Gamma(\kappa)/r^{\kappa}}, & x_{i}=0, \\ (x_{i}+\kappa)/r, & x_{i} \geq 1, \end{array}
\end{align*}
and introduce the notation
\[
\tilde{\rho}(\lambda) :=
\Ep_{r\lambda}[\lambda\log(\lambda/\hat{\theta}_{\Pi^{*},1}(X_{1})) - \lambda + \hat{\theta}_{\Pi^{*},1}(X_{1})].
\]
Then we get
\[\mathcal{E}(\Theta[s_{n}]) \leq (n-s_{n})\tilde{\rho}(0)+s_{n}\sup_{\lambda>0}\tilde{\rho}(\lambda).\]
So, it suffices to show two (asymptotic) inequalities
\begin{align*}
\tilde{\rho}(0)=O(\eta_{n}) \text{ and } \sup_{\lambda>0}\tilde{\rho}(\lambda)=(1+o(1))\mathrm{e}^{-1}r^{-1}s_{n}\log\eta_{n}^{-1} .
\end{align*}
Consider $\sup_{\lambda>0}\tilde{\rho}(\lambda)$.
For $\lambda>0$, $\tilde{\rho}(\lambda)$ is expressed as
\begin{align*}
\tilde{\rho}(\lambda) &= 
\Ep_{r\lambda} 1_{X_{1}\geq 1}(X_{1}) \left[\lambda\log\left(\frac{r\lambda}{X_{1}+\kappa}\right)-\lambda+\frac{X_{1}+\kappa}{r}\right]
\\
&\quad+
\Ep_{r\lambda } 1_{X_{1}=0}(X_{1})\left[\lambda\log\left(\frac{\lambda}{\eta_{n}}\frac{1+\eta_{n}\Gamma(\kappa)/r^{\kappa}}{\Gamma(\kappa+1)/r^{\kappa+1}}\right) - \lambda+ \eta_{n} \frac{\Gamma(\kappa+1)/r^{\kappa+1}}{1+\eta_{n}\Gamma(\kappa)/r^{\kappa}}\right].
\end{align*}
By the same argument as in the proof of Theorem \ref{theorem: exact minimaxity within sparse Poisson models},
we obtain 
\[\sup_{\lambda>0}\tilde{\rho}(\lambda)= (1+o(1)) \mathrm{e}^{-1}r^{-1}s_{n}\log\eta_{n}^{-1}
\text{ and }
\tilde{\rho}(0)= O(\eta_{n}),\]
which yields the desired upper bound on $\mathcal{E}(\Theta[s_{n}])$ and completes the proof.
\qed

\subsection{Proof of Proposition \ref{theorem: adaptive minimaxity for estimation}}
\label{subsection: proof of adaptive minimax estimation}

It suffices to show that $R_{\mathrm{e}}(\theta,\hat{\theta}_{\Pi[\hat{\eta}_{n},\kappa]}) - R_{\mathrm{e}}(\theta,\hat{\theta}_{\Pi[\eta_{n},\kappa]})$ is uniformly negligible compared to $\mathcal{E}(\Theta[s_{n}])$
since $\sup_{\theta\in\Theta[s_{n}]} R_{\mathrm{e}}(\theta, \hat{\theta}_{\Pi[\eta_{n},\kappa]}) \sim \mathcal{E}(\Theta[s_{n}])$.
We begin with decomposing $R_{\mathrm{e}}(\theta,\hat{\theta}_{\Pi[\hat{\eta}_{n},\kappa]}) - R_{\mathrm{e}}(\theta,\hat{\theta}_{\Pi[\eta_{n},\kappa]})$ in such a way that
\begin{align}
R_{\mathrm{e}}(\theta,\hat{\theta}_{\Pi[\hat{\eta}_{n},\kappa]}) 
- R_{\mathrm{e}}(\theta,\hat{\theta}_{\Pi[\eta_{n},\kappa]})
&=\sum_{i=1}^{n} 
\underbrace{\Ep_{\theta}\left\{\theta_{i}\log \frac{  \hat{\theta}_{\Pi[\eta_{n},\kappa],i}  }{ \hat{\theta}_{\Pi[\hat{\eta}_{n},\kappa],i}} + \hat{\theta}_{\Pi[\hat{\eta}_{n},\kappa],i}  - \hat{\theta}_{\Pi[\eta_{n},\kappa],i} \right\}}_{:=G_{i}}
\nonumber\\
&=\sum_{i\in\mathcal{A}}G_{i} + \sum_{i \not\in \mathcal{A}}G_{i},
\label{decomposition of estiamtion difference}
\end{align}
where $\mathcal{A}:=\mathcal{A}(\theta):=\{i:\theta_{i}\neq 0\}$.
In three steps, we will show 
\[R_{\mathrm{e}}(\theta,\hat{\theta}_{\Pi[\hat{\eta}_{n},\kappa]}) - R_{\mathrm{e}}(\theta,\hat{\theta}_{\Pi[\eta_{n},\kappa]})=o(s_{n}\log (n/s_{n})) \text{ uniformly in } \theta\in\Theta[s_{n}].\]

\textit{Step 1: Bounding $G_{i}$ for $i\not\in\mathcal{A}$}.
By the same argument as in Step 1 of the proof of Theorem \ref{theorem: adaptive},
we have
\begin{align}
G_{i}
&\le c_{1}\eta_{n} \text{ with }c_{1}:=3\frac{\Gamma(\kappa+1)}{r^{\kappa+1}} \eta_{n}.
\label{bound Gi for i not in A}
\end{align}

\textit{Step 2: Bounding $G_{i}$ for $i\in\mathcal{A}$}.
Since 
$\hat{\theta}_{\Pi[\eta_{n},\kappa],i}(X_{i})=\hat{\theta}_{\Pi[\hat{\eta}_{n},\kappa],i}(X_{i})$ for $X_{i}\ge 1$
from the explicit expression of $\hat{\theta}_{\Pi[\eta_{n},\kappa]}$,
we have
\begin{align*}
G_{i}
&=
\Ep_{\theta} 1_{X_{i}=0}\left\{\theta_{i}\log \frac{  \hat{\theta}_{\Pi[\eta_{n},\kappa],i}  }{ \hat{\theta}_{\Pi[\hat{\eta}_{n},\kappa],i}} + \hat{\theta}_{\Pi[\hat{\eta}_{n},\kappa],i}  - \hat{\theta}_{\Pi[\eta_{n},\kappa],i} \right\}
\\
&=
\Ep_{\theta} 1_{X_{i}=0}\left\{
\theta_{i}\log \frac{\eta_{n}}{\hat{\eta}_{n}}\frac{1+\hat{\eta}_{n}\Gamma(\kappa)/r^{\kappa}}{1+\eta_{n}\Gamma(\kappa)/r^{\kappa}}
+
\frac{ (\hat{\eta}_{n}-\eta_{n})\Gamma(\kappa+1)/r^{\kappa+1} }{ \{1+\hat{\eta}_{n}\Gamma(\kappa)/r^{\kappa} \} \{ 1+ \eta_{n}\Gamma(\kappa)/r^{\kappa} \}}
\right\}\\
&\le 
\Ep_{\theta} 1_{X_{i}=0}\left\{
\theta_{i}\log \frac{\eta_{n}}{\hat{\eta}_{n}}\frac{1+\hat{\eta}_{n}\Gamma(\kappa)/r^{\kappa}}{1+\eta_{n}\Gamma(\kappa)/r^{\kappa}}\right\}
+\eta_{n}\Ep_{\theta}|\hat{s}_{n}/s_{n}-1|\frac{\Gamma(\kappa+1)}{r^{\kappa+1}}\\
&\le
\Ep_{\theta} 1_{X_{i}=0}\left\{
\theta_{i}\log \frac{\eta_{n}}{\hat{\eta}_{n}}\frac{1+\hat{\eta}_{n}\Gamma(\kappa)/r^{\kappa}}{1+\eta_{n}\Gamma(\kappa)/r^{\kappa}}\right\}
+3\frac{\Gamma(\kappa+1)}{r^{\kappa+1}}\eta_{n},
\end{align*}
where the last inequality follows from Lemma \ref{lemma: properties of hats_n} (b).
Consider the following decomposition of the first term in the rightmost side of the above inequality:
\begin{align*}
&\Ep_{\theta} 1_{X_{i}=0}\left\{
\theta_{i}\log \frac{\eta_{n}}{\hat{\eta}_{n}}\frac{1+\hat{\eta}_{n}\Gamma(\kappa)/r^{\kappa}}{1+\eta_{n}\Gamma(\kappa)/r^{\kappa}}\right\}
\\
&=
\underbrace{\theta_{i}
\Ep_{\theta} 1_{X_{i}=0} \log \frac{\eta_{n}}{\hat{\eta}_{n}}  }_{=:T^{(1)}_{i}}
+
\underbrace{\theta_{i}
\Ep_{\theta} 1_{X_{i}=0} \log \frac{1+\hat{\eta}_{n} \Gamma(\kappa)/r^{\kappa}}{1+ \eta_{n}\Gamma(\kappa)/r^{\kappa}} }_{=:T^{(2)}_{i}}.
\end{align*}
By the same argument as in Step 3 of the proof of Theorem \ref{theorem: adaptive},
we have $T^{(1)}_{i}=o(\log (n/s_{n}))$ for any $s_{n}$ such that $1\le \inf_{n} s_{n}$ and $\sup_{n} s_{n}/n < 1$.
Let us bound $T^{(2)}_{i}$.
Observe that 
\begin{align*}
\log \frac{1 + \hat{\eta}_{n}\Gamma(\kappa)/r^{\kappa}}{1 + \eta_{n} \Gamma(\kappa)/r^{\kappa}}
\le \log \left\{ 1 + \frac{|\hat{\eta}_{n}-\eta_{n}| \Gamma(\kappa)/r^{\kappa}}{1+\eta_{n}\Gamma(\kappa)/r^{\kappa}}\right\}
\le \eta_{n} |\hat{s}_{n}/s_{n}-1| \frac{\Gamma(\kappa)}{r^{\kappa}}.
\end{align*}
Together with Lemma \ref{lemma: properties of hats_n} (b), this gives
\begin{align*}
T^{(2)}_{i} 
&\le \eta_{n} \frac{\Gamma(\kappa)}{r^{\kappa}} \theta_{i} \Ep_{\theta}[1_{X_{i}=0} |\hat{s}_{n}/s_{n}-1|]\\
&\le \frac{\Gamma(\kappa)}{r^{\kappa}} \sup_{\theta_{i}>0}\{\theta_{i}\mathrm{e}^{-r\theta_{i}/2}\} \sqrt{\Ep_{\theta} |\hat{s}_{n}/s_{n}-1|^{2}}\\
&\le c_{2} \text{ with some $c_{2}$ depending only on $r$},
\end{align*}
which concludes that 
\begin{align}
G_{i}=T^{(1)}_{i}+T^{(2)}_{i}=o(\log (n/s_{n})).
\label{bound Gi for i in A}
\end{align}

\textit{Step 3}. Combining (\ref{bound Gi for i not in A}) and (\ref{bound Gi for i in A}) into (\ref{decomposition of estiamtion difference})
yields
\begin{align*}
R_{\mathrm{e}}(\theta,\hat{\theta}_{\Pi[\hat{\eta}_{n},\kappa]}) 
- R_{\mathrm{e}}(\theta,\hat{\theta}_{\Pi[\eta_{n},\kappa]})
&=o(s_{n}\log (n/s_{n})),
\end{align*}
which completes the proof. \qed

\subsection{Proof of Proposition \ref{proposition: suboptimality of spike and slab}}

Violating Condition (\ref{weak tail robustness}) in the main manuscript implies for some $i$,
\begin{align*}
\mathop{\mathrm{liminf}}_{\theta_{i}\to\infty}\hat{\theta}_{\Pi,i}(X_{i};t) / \theta_{1} >1
\text{ almost surely}.
\end{align*}
Together with Lemma \ref{lemma: Kullback--Leibler formula},
this gives
\begin{align*}
\mathop{\mathrm{liminf}}_{\theta_{i}\to\infty}\frac{R(\theta,q_{\Pi_{\Pi}})}{\theta_{i}}
= 
\mathop{\mathrm{liminf}}_{\theta_{i}\to\infty}
\int_{r}^{r+1}\Ep_{t\theta} 
\left[ \log \frac{\theta_{i}}{\hat{\theta}_i (X_{i};t)}
+\frac{\hat{\theta}_{i} (X_{i};t)}{\theta_{i}} -1
\right]dt>0,
\end{align*}
which completes the proof.
\qed

\subsection{Proof of Proposition 
\ref{proposition: optimality of Cauchy or Pareto}}

Without loss of generality, we can assume that $\eta_{n} \le 1/2$.
Further, we can assume $\Lambda>2$ (if not, replace $\Lambda$ by $\Lambda+2$).
It suffices to show $\Pi=\Pi_{\gamma}[\eta_{n}]$ satisfies Condition \ref{posterior condition} in the main manuscript.

Let $I_{\gamma}(s;t):= \int_{0}^{\infty} \lambda^{s-1} \mathrm{e}^{-t\lambda} \gamma(\lambda) d\lambda $.
Note $I_{\gamma}(s;t)$ is finite for $s\ge 1$ and $t>0$ because
\begin{align*}
I_{\gamma}(s;t) 
\le \underbrace{ \int_{0}^{1}  \mathrm{e}^{-t\lambda}\gamma(\lambda)d\lambda }_{<\infty \text{ by Condition 
(\ref{condition integrability}) in the main manuscript}} 
+ 
\underbrace{ \int_{1}^{\infty}\lambda^{s-1} \mathrm{e}^{-t\lambda} \gamma(\lambda)d\lambda }_{<\infty \text{ by Condition (\ref{condition on slab}) in the main manuscript}}.
\end{align*}
Observe that for $i=1,\ldots,n$,
\begin{align*}
\hat{\theta}_{\Pi,i}(X_{i};t) &= \frac{ \eta_{n} I_{\gamma}(X_{i}+2;t) }{ (1-\eta_{n}) 0^{X_{i}} + \eta_{n} I_{\gamma}(X_{i}+1;t)}.
\end{align*}
This gives $C_{1} \eta_{n} \le \hat{\theta}_{\Pi,i}(0;t) \le C_{2} \eta_{n}$
and
$C_{3} \le \hat{\theta}_{\Pi,i}(X_{i};t) \le C_{4}, \,X_{i}=1,\ldots, \lceil \Lambda-1\rceil$,
where
\begin{align*}
&C_{1}:= \frac{I_{\gamma}(2;r+1)}{1+I_{\gamma}(1;r)}, \quad\quad\quad\quad\quad\quad\quad\quad\quad C_{2}:= 2I_{\gamma}(2;r),\\
&C_{3}:= \min_{s=1,\ldots,\lceil \Lambda-1\rceil}\frac{I_{\gamma}(s+2;r+1)}{I_{\gamma}(s+1;r)}, 
 \text{ and } C_{4}:=\max_{s=1,\ldots,\lceil \Lambda-1\rceil}\frac{I_{\gamma}(s+2;r)}{I_{\gamma}(s+1;r+1)}.
\end{align*}
These conclude that $\Pi$ meets (P1) and (P2) of Condition \ref{posterior condition}.

Observe that the integration-by-parts formula yields
\begin{align*}
I_{\gamma}(s+2;t) = \frac{s+1}{t} I_{\gamma}(s+1;t) + \frac{1}{t}\int_{0}^{\infty}\lambda^{s}\mathrm{e}^{-t\lambda}\gamma(\lambda)
\left\{\lambda\frac{d}{d\lambda} \log \gamma(\lambda)\right\} d\lambda.
\end{align*}
This, together with Condition (\ref{condition on slab}) in \cite{YKK}, implies
\begin{align*}
\frac{X_{i}+1-\Lambda}{t} \le \hat{\theta}_{\Pi,i}(X_{i};t) = \frac{I_{\gamma}(X_{i}+2;t)}{I_{\gamma}(X_{i}+1;t)} \le \frac{X_{i}+1+\Lambda}{t}, \, X_{i}>  \lceil \Lambda-1\rceil,
\end{align*}
which concludes that $\Pi$ satisfies (P3) of Condition \ref{posterior condition} and completes the proof.
\qed

\subsubsection{Proof of Proposition 
\ref{proposition: optimal scaling}}
\label{subsection: proof for optimal scaling}

From the explicit form of $\hat{\theta}_{\Pi[\eta_{n},\kappa]}$,
we can replace $C_{2}$ in (\ref{bounding the first integral}) in the main manuscript by $\Gamma(\kappa+1) \{r^{-\kappa}-(r+1)^{\kappa}\}/\kappa$.
After this replacement, we multiply $\eta_{n}$ by $L$ in $\Pi^{*}=\Pi[\eta_{n},\kappa]$ of 
the proof of Theorem \ref{theorem: exact minimaxity within sparse Poisson models}.
Then, inequalities (\ref{bounding the first integral})-(\ref{bounding the third integral}) 
in the main manuscript give
\begin{align*}
    \sup_{\theta\in\Theta[s_{n}]} R(\theta, q_{\Pi[L\eta_{n},\kappa]})&= s_{n}\sup_{\lambda>0}R_{1}(\lambda,q_{\Pi[L\eta_{n},\kappa],1}) + (n-s_{n})R_{1}(0,q_{\Pi[L\eta_{n},\kappa],1})\\
    &\leq \mathcal{C}s_{n}\log (\eta_{n}^{-1}L^{-1}) + \Upsilon_{1} + \mathcal{K} L n\eta_{n}\\
    &= \mathcal{C}s_{n}\log \eta_{n}^{-1} + s_{n}\{\mathcal{K} L  - \mathcal{C}\log L\} + \Upsilon_{1},
\end{align*}
where $\Upsilon_{1}$ consists of terms independent of $L$ and terms that are $o(s_{n})$. This completes the proof.\qed

\section{Proofs for Section \ref{section: extensions}}
\label{section: proofs for propositions in section 3}

This section presents proofs of propositions in Section \ref{section: extensions}.

\subsection{Proof of Proposition \ref{proposition: quasi-sparse minimax}}
\label{subsection: proof of proposition quasi-sparse minimax}

The proof follows almost the same line as in the proof of Theorem \ref{theorem: exact minimaxity within sparse Poisson models}.

\subsubsection*{Step 1: Lower bound on $\mathcal{R}(\Theta[s_{n},\varepsilon_{n}])$}

We start with extracting an exact sparse subspace in $\Theta[s_{n},\varepsilon_{n}]$.
Fix $\nu^{\circ}$ in such a way that $\mathrm{e}^{-r\nu^{\circ}}-\mathrm{e}^{-(r+1)\nu^{\circ}}=\mathcal{C}$.
Taking a sufficiently large $n$ (depending only on $\nu^{\circ}$), we have
\[
\left\{ \theta \in \R_{+}^{n} : \theta_{[1]} = \cdots = \theta_{[ s_{n} ] } = \nu , \ \theta_{ [ s_{n} + 1 ] } = \cdots = \theta_{[n]} = 0 \right\} 
 \subset 
\Theta [ s_{n} , \varepsilon_{n} ],
\]
where $\theta_{[i]}$ is the $i$-th largest component of $\{\theta_{i}:i=1,\ldots,n\}$.
This gives
\[
\mathcal{R}(\Theta[s_{n},\varepsilon_{n}]) 
\geq \int R(\theta, q_{\Pi_{ \mathrm{B}, \nu^{\circ}} }) d \Pi_{\mathrm{B}, \nu }(\theta) 
\geq \{ \mathcal{C} s_{n} \log (n / s_{n}) \} (1+o(1)),
\]
which completes Step 1.

\subsubsection*{Step 2: Upper bound on $\mathcal{R}(\Theta[s_{n},\varepsilon_{n}])$}

Recall that $N(\theta,\varepsilon):=\#\{i:\theta_{i} > \varepsilon\}$.
By definition of $N(\theta,\varepsilon)$, we have, for an independent prior $\Pi$,
\begin{align}
R(\theta, q_{\Pi}) 
& \leq N(\theta, \varepsilon) \sup_{\varepsilon_{n}< \lambda} R_{i} (\lambda , q_{\Pi, i})
+ (n - N(\theta, \varepsilon) ) \sup_{\lambda \leq \varepsilon_{n}} R_{i} ( \lambda, q_{\Pi, i}),
\label{eq: upper bound in case (b)}
\end{align}
where
$R_{i}$ ($i=1,\ldots,n$) is the Kullback--Leibler risk for the $i$-th coordinate.
Since $N(\theta,\varepsilon_{n})\leq s_{n}$ for $\theta\in\Theta[s_{n},\varepsilon_{n}]$, we have
\begin{align}
R(\theta, q_{\Pi})
& \leq s_{n} \sup_{0<\lambda}R_{i}(\lambda,q_{\Pi,i}) + n \sup_{\lambda \leq \varepsilon_{n}} R_{i} (\lambda, q_{\Pi,i}).
\label{eq: upper bound 2 in case (b)}
\end{align}
We put a prior $\Pi^{*} = \Pi[\eta_{n},\kappa]$ into $\Pi$ and then we have $R_{i} (\lambda , q_{\Pi^{*}, i}) = \rho(\lambda)$.
From (\ref{bounding the first integral})-(\ref{bounding the third integral}) in the main manuscript, we have 
$\sup_{0 < \lambda}R_i(\lambda,q_{\Pi^{*}, i})\leq \mathcal{C}\log(n / s_n)$
as well as $\sup_{\lambda \leq \varepsilon_{n}} R_{i}(\lambda,q_{\Pi_{\Pi^{*}}}) = O(\varepsilon_{n}\log \{\eta_{n}^{-1}\})$,
which completes the proof.\qed

\subsection{Proof of Proposition \ref{proposition: adaptive in quasi-sparse}}

The proof follows almost the same line as in the proof for Theorem \ref{theorem: adaptive}.
For $\theta\in\Theta[s_{n},\varepsilon_{n}]$,
let $\mathcal{A}:=\mathcal{A}(\theta)=\{i:\theta_{i}>\varepsilon_{n}\}$.

For $i\in\mathcal{A}$, we use the bound (\ref{Case i in A}) in the main manuscript.
For $i\not\in\mathcal{A}$ and for $X_{i}\geq 1$,
from (\ref{Case ii}) in the main manuscript and (\ref{Case iv}) in the main manuscript,
we have
\begin{align}
    \Ep_{Y_{i}\mid\theta_{i}}\log \left\{\frac{q_{\Pi[\eta_{n},\kappa],i}(Y_{i}\mid X_{i})}{q_{\Pi[\hat{\eta}_{n},\kappa],i}(Y_{i}\mid X_{i})}\right\}
    =0.
    \label{quasi Case x_i > 0}
\end{align}
For $i\not\in\mathcal{A}$ and for $X_{i}=0$,
from (\ref{Case i}) in the main manuscript and (\ref{Case iii}) in the main manuscript,
we have
\begin{align}
&\Ep_{Y_{i}\mid\theta_{i}}\log \left\{\frac{q_{\Pi[\eta_{n},\kappa],i}(Y_{i}\mid X_{i}=0)}{q_{\Pi[\hat{\eta}_{n},\kappa],i}(Y_{i}\mid X_{i}=0)}\right\}
\nonumber\\
&\leq
\left|\frac{\hat{s}_{n}}{s_{n}}-1\right|+\log \left\{1+\frac{\Gamma(\kappa)}{(r+1)^{\kappa}}\right\}
+(1-\mathrm{e}^{-\varepsilon_{n}})
\left\{\left|\frac{\hat{s}_{n}}{s_{n}}-1\right|+\log \frac{s_{n}}{\hat{s}_{n}}\right\}.
\label{quasi Case x_i = 0}
\end{align}

Thus these, together with (\ref{Case i in A}),
(\ref{eq: eval T1 1}), (\ref{eq: eval T1 2}), and (\ref{Decomposition of KL with plug in}) in the main manuscript, completes the proof.
\qed

\subsection{Proof of Proposition \ref{proposition: minimax risk in random design}}
\subsubsection{Proof of Proposition \ref{proposition: minimax risk in random design}}
\label{subsection: imgredient of theorem minimax risk in random design}

We begin with presenting the proof by assuming that the next lemma holds.
\begin{condition}
\label{condition: condition on ris}
The asymptotic equality
$\overline{\mathcal{C}}:=\sum_{i = 1}^n\mathcal{C}_i / n \sim \sum_{i \in J} 
\mathcal{C}_i / s_n$ 
holds
for any subset $J\subset\{1,\ldots,n\}$ satisfying $|J| = s_n$.
\end{condition}

\begin{lemma}
\label{lemma: exact minimaxity within sparse Poisson models in inhomogeneous case}
Fix an infinite sequence $\{r_i\in(0, \infty):i\in\mathbb{N}\}$ such that $0 < \inf_i r_i 
\leq \sup_i r_i <\infty$.
Fix also $0<\kappa \le 1$.
Suppose Condition \ref{condition: condition on ris} holds. Then, the asymptotic equality
\begin{align*}
\overline{\mathcal{R}}(\Theta[s_{n}] \mid \{r_{i}\}) 
\sim \sup_{\theta\in\Theta[s_{n}]} R(\theta, q_{\Pi[\eta_{n},\kappa]}\mid \{r_{i}\})
\sim \overline{\mathcal{C}} s_{n} \log (\eta_{n}^{-1})
\end{align*}
holds as $n\to\infty$ and $\eta_{n}=s_{n}/n\to 0$.
\end{lemma}

Using Lemma \ref{lemma: exact minimaxity within sparse Poisson models in inhomogeneous case}, 
we will show that
under Condition \ref{Condition: tail and edge},
for $\delta\in(0,1)$,
the asymptotic inequality 
\begin{align}
\left|\overline{\mathcal{R}}(\Theta[s_{n}] \mid \{r_{i}\}) 
	/ (\mathbb{E}_G \overline{\mathcal{C}} s_{n} \log \eta_n^{-1}) - 1\right|
    \leq b_n + \sqrt{ \{1/(2n)\}\log (2/\delta)}
\label{eq: probab}
\end{align}
holds with probability greater than $1-\delta$, 
where $b_n$ is an $o_{P}(1)$ term independent of $\delta$.
From Lemma \ref{lemma: exact minimaxity within sparse Poisson models in inhomogeneous case},
it suffices to show that under Condition \ref{Condition: tail and edge},
for $\delta\in(0,1)$,
the asymptotic inequality 
\[\left|\overline{\mathcal{C}} s_{n} \log \eta_n^{-1})
	/ (\mathbb{E}_G \overline{\mathcal{C}} s_{n} \log \eta_n^{-1}) - 1\right|
    \leq b_n + \sqrt{ \{1/(2n)\}\log (2/\delta)}\]
holds with probability greater than $1-\delta$.
Take the expectation of $\overline{\mathcal{C}}$ with respect to $G$. 
Since $\overline{\mathcal{C}}$ is in (0,1),
the Hoeffding inequality gives
\[\mathrm{Pr}
	\left(
    	\left|
        	\overline{\mathcal{C}} - \mathbb{E}_G \overline{\mathcal{C}} 
		\right|
		\geq t
	\right)
    \leq 2\exp\left(-2nt^2\right), \ t>0.
\]
This implies that for any $\delta\in(0,1)$, we have
\begin{equation}
\left|\overline{\mathcal{C}} - \mathbb{E}_G \overline{\mathcal{C}} \right| 
\leq \sqrt{\{1/(2n)\} \log(2 / \delta)}
\label{concentration}
\end{equation}
with probability greater than $1-\delta$.
Observe that Condition \ref{Condition: tail and edge} together with (\ref{bounding the first integral})-(\ref{bounding the third integral}) 
in the main manuscript assures that
$\overline{\mathcal{R}}(\Theta[s_{n}]\mid \{r_{i}\}) - \overline{\mathcal{C}}s_{n}\log(\eta_{n}^{-1}) =O_{P}(1)$.
Then, this together with (\ref{concentration}) yields
\begin{align*}
\overline{\mathcal{R}}(\Theta[s_{n}]\mid \{r_{i}\}) &\ge
\left(
	\mathbb{E}_G \overline{\mathcal{C}} - \sqrt{ \{1/(2n)\}\log(1 / \delta)} + b_n
\right)
s_n \log \eta_n^{-1} \text{ and }\\
\overline{\mathcal{R}}(\Theta[s_n] \mid \{r_{i}\})
&\le \left(
	\mathbb{E}_G \overline{\mathcal{C}} + \sqrt{ \{1/(2n)\}\log(2 / \delta)} + b_n
\right)
s_n \log \eta_n^{-1}
\end{align*}
with probability greater than $1-\delta$,
where $b_n$ is an $o_P(1)$ term that is independent of $\delta$. 
This shows (\ref{eq: probab}).

By substituting $\delta = \delta_n = \mathrm{e}^{-n / \log n}$ in (\ref{eq: probab})
and by using Lemma \ref{lemma: exact minimaxity within sparse Poisson models in inhomogeneous case},
we complete the proof of Proposition \ref{proposition: minimax risk in random design}. \qed

In the rest of this subsection,
we will present the supporting lemma for the proof of Lemma \ref{lemma: exact minimaxity within sparse Poisson models in inhomogeneous case} ahead,
and
then will provide the full proof for Lemma \ref{lemma: exact minimaxity within sparse Poisson models in inhomogeneous case}.

\subsubsection{Supporting lemma}

We provide the useful formula for the Kullback-Leibler risk for MCAR settings.
For $i = 1,2,\ldots,n$, 
let $t_i = t_i(\tau)$ $(i = 1,\cdots, n)$ be a smooth and monotonically increasing function of $\tau\in[0,1]$ 
such that $t_i(0) = r_i$ and $t_i(1) = 1 + r_i$. 
Using $t_i(\tau)$, 
let $Z_i(\tau)$ ($i=1,\ldots,n$) be a random variable independently distributed to Po$(t_i(\tau)\theta_i)$. 
The density of $Z(\tau) = (Z_1(\tau),\ldots,Z_n(\tau))$ is denoted by 
\[p(z \mid \theta;\tau) = \prod_{i=1}^{n} \left[ \frac{\exp\{- t_i(\tau)\theta_{i}\}\{t_i(\tau)\theta_{i}\}^{z_{i}}}{z_{i}!}\right].\]
By definition, $X$ and $Y$ follow the same distributions as those of $Z(0)$ and $Z(1) - Z(0)$, respectively.
For a prior $\Pi$ of $\theta$, 
let 
\[\hat{\theta}_{\Pi,i}(z;\tau) := \int \theta_{i} p(z \mid \theta;\tau) d\Pi(\theta) \bigg{/} \int p(z \mid \theta ; \tau) d\Pi(\theta), \ i=1,\ldots,n.
\]
For the Bayes estimator $\hat{\theta}_\Pi$ based on $\Pi$, 
let
\[R_\mathrm{e}(\theta, \hat{\theta}_\Pi;\tau) := \mathbb{E}_{\theta;\tau}
\left[ \sum_{i=1}^n\dot{t}_i(\tau) 
\left\{ \theta_i\log\frac{\theta_i}{\hat{\theta}_{\Pi,i}(Z;\tau)} - \theta_i + \hat{\theta}_{\Pi,i}(Z;\tau)  
\right\}
\right], \, \tau\in[0,1],
\]
where $\Ep_{\theta;\tau}$ is the expectation with respect to $p(\cdot \mid \theta ; \tau)$.
\begin{lemma}{\cite{Komaki(2015)}}
\label{lemma: Kullback--Leibler formula in inhomogeneous case}
For a prior $\Pi$ of $\theta$,
if $\hat{\theta}_{\Pi}(z;\tau)$ based on $\Pi$ is strictly larger than $0$ for any $z\in\mathbb{N}^{n}$ and any $\tau\in[0,1]$,
then, the equality
\begin{align*}
R(\theta, q_\Pi) = \int_0^1 R_\mathrm{e}(\theta, \hat{\theta}_\Pi;\tau)d\tau
\end{align*}
holds.
\end{lemma}

\subsubsection{Proof of Lemma \ref{lemma: exact minimaxity within sparse Poisson models in inhomogeneous case}}
\label{subsubsection: proof of lemma exact minimaxity within sparse Poisson models in inhomogeneous case}

Hereafter, we suppress the dependence on $\{r_{i}\}$:
we denote by $R(\theta, \hat{q})$ the Kullback--Leibler risk conditioned on $r_{i}$s
and
denote by $\overline{R}(\Theta[s_{n}])$ the minimax Kullback--Leibler risk conditioned on $r_{i}$s.

\subsubsection*{Step 1: Lower bound on $\overline{\mathcal{R}}(\Theta[s_n])$
}

A lower bound on $\overline{\mathcal{R}}(\Theta[s_n])$ builds upon the Bayes risk maximization based on a \textit{varied-spike block-independent} prior.
Let \[\Pi_{\mathrm{VB}, \nu}\text{ with }\nu=(\nu^{(1)},\ldots,\nu^{(s_{n})})\in \mathbb{R}^{m_{n}} \times \mathbb{R}^{m_{n}} \times \cdots \times \mathbb{R}^{m_{n}} \times \mathbb{R}^{n-m_{n}s_{n}}\] be a \textit{varied-spike block-independent} prior built as follows: 
divide $\{1,2,\ldots,n\}$ into contiguous blocks $B_{j}$ ($j=1,2,\ldots,s_{n}$) with each length $m_{n}:=\lfloor \eta_{n}^{-1} \rfloor$.
In each block $B_{j}$,
draw $(\theta_{1+m_{n}(j-1)},\ldots,\theta_{m_{n}j})$ independently according to a single spike prior with spike strength parameter $\nu^{(j)}\in\mathbb{R}^{s_n}_+$,
where a single spike prior with spike strength parameter $\nu^{(j)}\in\mathbb{R}^{s_n}_+$ is the distribution of $\nu^{(j)}_I e_{I}$
with a uniformly random index $I\in\{1,\ldots,m_{n}\}$ and a unit length vector $e_{i}$ in the $i$-th coordinate direction.
Finally, set $\theta_{i}=0$ for the remaining $n-m_{n}s_{n}$ components. 
It is worth noting that a varied-spike block-independent prior may be different from a block-independent prior
since the spike strength may be varied in each coordinate. 
Hereafter, we use the block notation of a vector in $\mathbb{R}^{n}$:
For $v=(v_{1},v_{2},\ldots,v_{n})\in\mathbb{R}^{n}$,
for $j=1,\ldots,s_{n}-1$,
and for $k=1,\ldots,m_{n}$,
$v^{(j)}_{k}=v_{k+m_{n}(j-1)}$ 
and $v^{(j)}=(v^{(j)}_{1},\ldots,v^{(j)}_{m_{n}})$.

We first get the explicit expression of $\hat{\theta}_{\Pi_{\mathrm{VB},\nu}}=(\hat{\theta}_{\Pi_{\mathrm{VB}}}^{(1)},\ldots,\hat{\theta}_{\Pi_{\mathrm{VB}}}^{(s_{n})})$.
From the Bayes formula, we have,
for $j=1,\ldots,s_{n}-1$ and for $k=1,\ldots,m_{n}$,
\begin{align}
\hat{\theta}_{\Pi_{\mathrm{VB},\nu}, k}^{(j)}(x^{(j)})=
\begin{cases} 
w_k^{(j)}\nu_k^{(j)}& \text{if $\|x^{(j)}\|_{0}=0$,} \\ 
\nu_k^{(j)} & \text{if $x^{(j)}_{k}\neq 0$ and $x^{(j)}_{l}=0$ for $l\neq k$,} \\
0   & \text{otherwise,}
\end{cases}
\label{eq: expression of block estimator}
\end{align}
where $w_k^{(j)} := \exp(-r_k^{(j)}\nu_k^{(j)}) / \sum_{l=1}^{m_n}\exp(-r_l^{(j)}\nu_l^{(j)})$.

By the coordinate-wise additive property of the Kullback-Leibler divergence,
we decompose $R_\mathrm{e}(\theta, \hat{\theta}_{\Pi_{\mathrm{VB},\nu}};\tau)$ as
\[
R_\mathrm{e}(\theta, \hat{\theta}_{\Pi_{\mathrm{VB},\nu}};\tau)
=
\sum_{j=1}^{s_n -1 } 
R^{(j)}(\theta^{(j)};\tau) + R^{(s_{n})}( 0 ;\tau) ,\]
where, for $j=1,\ldots,s_{n}$,
\begin{align*}
&R^{(j)}(\theta^{(j)};\tau)\\
&:=
\Ep_{\theta;\tau}
 \sum_{k=1}^{m_n} 
\dot{t}^{(j)}_k(\tau) 
\bigg{[}
\theta^{(j)}_{k}\log\{\theta^{(j)}_k /\hat{\theta}^{(j)}_{\Pi_{\mathrm{VB},\nu},k}(Z^{(j)};\tau)\} 
- \theta^{(j)}_k 
+ \hat{\theta}^{(j)}_{\Pi_{\mathrm{VB},\nu},k}(Z^{(j)};\tau)  
\bigg{]}
\end{align*}
and, for $j=s_{n}+1$,
\begin{align*}
&R^{(s_{n})}(\theta^{(s_{n})};\tau)\\
&:=
\Ep_{\theta;\tau}
 \sum_{k=1}^{n-m_{n}s_{n}} 
\dot{t}^{(j)}_k(\tau) 
\bigg{[}
\theta^{(j)}_{k}\log\bigg{\{}
\frac{\theta^{(j)}_k}{ \hat{\theta}^{(j)}_{\Pi_{\mathrm{VB},\nu},k}(Z^{(j)};\tau)}\bigg{\}}
- \theta^{(j)}_k 
+ \hat{\theta}^{(j)}_{\Pi_{\mathrm{VB},\nu},k}(Z^{(j)};\tau)  
\bigg{]}.
\end{align*}
Here $\dot{f}(\tau)$ denotes the derivative of $f$ with respect to $\tau$.

Fix $\theta^{(j)}$ to be the $j$-th block of $\theta$ in the support of $\Pi_{\mathrm{VB},\nu}$
and consider $R^{(j)}(\theta^{(j)};\tau)$.
For notational brevity, we omit $\tau$ in $t_i(\tau)$'s.
From (\ref{eq: expression of block estimator}),
we have
\begin{align*}
&R^{(j)}(\theta^{(j)};\tau)\\
&= \Ep_{\theta;\tau}
	\sum_{k=1}^{m_n}  \dot{t}_k^{(j)}
	\left[ \theta_k^{(j)}\log\{\theta_k^{(j)} / 			\hat{\theta}_{\Pi_{\mathrm{VB},\nu},k}^{(j)}(Z^{(j)};\tau)\} 
		- \theta_k^{(j)} 
		+ \hat{\theta}_{\Pi_{\mathrm{VB},\nu},k}^{(j)}(Z^{(j)};\tau)  
	\right] \\
&= \mathrm{e}^{-t^{(j)}_{\gamma} \nu^{(j)}_{\gamma}} 
	\left\{
    	\dot{t}_{\gamma}^{(j)} \nu_{\gamma}^{(j)}\log \frac{1}{w^{(j)}_{\gamma}} 
		-\dot{t}^{(j)}_{\gamma}\nu^{(j)}_{\gamma} 
        + \sum_{k = 1}^{m_n}\dot{t}_{k}^{(j)} w_{k}^{(j)}\nu_{k}^{(j)}
	\right\},
\end{align*}
where we denote by $\gamma=\gamma(j)$ the location in which the element is a spike.
Taking the expectation with respect to $\Pi_{\mathrm{VB},\nu}$ yields
\begin{align*}
&\int R^{(j)}(\theta^{(j)};\tau) d\Pi_{\mathrm{VB},\nu}(\theta)\\
&= \frac{1}{m_n}\sum_{k=1}^{m_n}
	\mathrm{e}^{-t_{k}^{(j)}\nu_{k}^{(j)}}
	\left(
    	\dot{t}_{k}^{(j)} \nu_{k}^{(j)}\log \frac{1}{w_{k}^{(j)}} 
		-\dot{t}_{k}^{(j)}\nu_{k}^{(j)} 
        + \sum_{l = 1}^{m_n}\dot{t}_{l}^{(j)} w_{l}^{(j)}\nu_{l}^{(j)}
	\right) \\
&\geq \frac{1}{m_n}\sum_{k=1}^{m_n}
	\mathrm{e}^{-t_{k}^{(j)} \nu_{k}^{(j)}} 
	\left\{
    	\dot{t}_{k}^{(j)} \nu_{k}^{(j)}\log (1 / w_{k}^{(j)}) 
		-\dot{t}_{k}^{(j)}\nu_{k}^{(j)} 
	\right\} .
\end{align*}
Integrating the both hand sides of the above equality with respect to $\tau$ over $[0,1]$, 
we have
\begin{align*}
\int_0^1 
	&\left[\int R^{(j)}(\theta^{(j)};\tau) 
	d\Pi_{\mathrm{VB},\nu} (\theta)\right]d\tau \\
&\geq \frac{1}{m_n} \sum_{k=1}^{m_n} 
	\bigg{\{}
    	f_{k}^{(j)}(\nu_{k}^{(j)})\log \bigg{(}\frac{1}{w_{k}^{(j)}}\bigg{)}
		-  f_{k}^{(j)}(\nu_{k}^{(j)})
    \bigg{\}},
\end{align*}
where $f_{k}^{(j)}(\lambda) := \exp\{-r_{k}^{(j)} \lambda\} - \exp\{-(1 + r_{k}^{(j)})\lambda\}$, $\lambda>0$.
By summing up the block-wise risk evaluation, we have the following lower bound on the overall Bayes risk of $\Pi_{\mathrm{VB},\nu}$:
\begin{align}
\overline{\mathcal{R}}(\Theta[s_n])
& \geq \sum_{j=1}^{s_n} \int_0^1 
	\left[
    	\int R^{(j)} (\theta^{(j)} ;\tau) 
        d\Pi_{\mathrm{VB},\nu} (\theta)
    \right]
    d\tau \nonumber \\
&\geq \frac{1}{m_n} \sum_{j=1}^{s_n}\sum_{k=1}^{m_n} 
	f^{(j)}_k(\nu^{(j)}_k)\log \frac{1}{w^{(j)}_k} 
    - \frac{1}{m_n} \sum_{j = 1}^{s_n} \sum_{k = 1}^{m_n}
    f^{(j)}_k(\nu^{(j)}_{k}) \nonumber\\
&= \frac{1}{m_n} \sum_{j=1}^{s_n}\sum_{k=1}^{m_n} 
	f^{(j)}_k(\nu^{(j)}_k)\log \frac{1}{w^{(j)}_k} 
	- \frac{1}{m_n} \sum_{i = 1}^{n} f_i(\nu_{i}),
\label{lower}\end{align}
where $f_{i}(\lambda) := \exp\{-r_{i}\lambda\} - \exp\{ -(1+r_{i})\lambda \}$, $\lambda>0$.

We next show that the asymptotic inequality
\begin{equation}
\overline{\mathcal{R}}(\Theta[s_n]) 
\geq  \overline{f}(\nu)
	\{s_n \log (\eta_{n}^{-1})\} (1 + o(1))
\label{lowerbound}
\end{equation}
holds with $\overline{f}(\nu) := \sum_{i=1}^{n} f_{i}(\nu_{i}) / n$ ,  provided that the following condition holds for $\nu$:
\begin{condition}
\label{condition: homogeneity in some sense}
There exists a positive constant $C$ 
such that 
$\max_{l} r^{(j)}_l\nu^{(j)}_l \leq C$ 
for any $j = 1,\ldots,s_n.$
\end{condition}
For $\nu$ satisfying Condition \ref{condition: homogeneity in some sense},
the first term of (\ref{lower}) is rewritten as
\begin{align*}
\frac{1}{m_n} \sum_{j=1}^{s_n}\sum_{k=1}^{m_n} 
	f^{(j)}_i(\nu^{(j)}_i)\log\frac{1}{w^{(j)}_k} 
&= \frac{1}{m_n} \sum_{j=1}^{s_n}\sum_{k=1}^{m_n} 
	f^{(j)}_i(\nu^{(j)}_i)\left( \log m_n + \log\frac{1}{m_n w^{(j)}_k}\right) \\
&= \overline{f}(\nu) s_{n}\{\log (\eta_{n}^{-1})\} (1 + o(1)),
\end{align*}
because by definition of $w^{(j)}_k$
we have,
for any $k=1,...,m_n$ and $j=1,...,s_n$,
\begin{align*}
\exp\left\{- \max_{l=1,...,m_n} r^{(j)}_l\nu^{(j)}_l 
	+ r^{(j)}_k\nu^{(j)}_k\right\}
&\leq \frac{1}{m_nw^{(j)}_k}\\
&\leq \exp\left\{- \min_{l=1,...,m_n} r^{(j)}_l\nu^{(j)}_l
 	+ r^{(j)}_k\nu^{(j)}_k\right\}.
\end{align*}
For $\nu$ satisfying Condition \ref{condition: homogeneity in some sense},
the second term of (\ref{lower}) is negligible compared to the first term.

Since Condition \ref{condition: homogeneity in some sense}
holds for $\nu^{\circ}$ that maximizes the right hand side of (\ref{lowerbound}), 
that is, $\nu^{\circ}_i = \log (1 + 1 / r_i)$ $(i = 1,\ldots,n)$,
we obtain the desired lower bound by substituting $\nu=\nu^{\circ}$,
which completes Step 1.

\subsubsection*{Step 2: Upper bound on $\overline{\mathcal{R}}(\Theta[s_n])$ 
}
We derive an upper bound on $\overline{\mathcal{R}}(\Theta[s_n])$.
For simplicity, 
we only show that
\[
\sup_{\theta\in\Theta[s_{n}]} R(\theta, q_{\Pi[\eta_{n},1]}) \sim \overline{C}s_{n} \log \eta_{n}^{-1},
\]
and omit the proof for general $\kappa \in (0,1]$.
Let $\Pi^{*}=\Pi[\eta_{n},1]$.
Fix $i=1,2,\ldots,n$. For $\lambda_i > 0$, 
let
\begin{align*}
\rho_i(\lambda_i) :=
\Ep_{\lambda_{i}} \log [ \{\exp(-\lambda_i)\lambda_i^{Y_{1}}/Y_{1}!\}/\{q_{\Pi^{*},i}(Y_{1}\mid X_{1})\} ] .
\end{align*} 
For $\lambda_i > 0$ and $\tau\in[0,1]$, 
let 
\begin{align*}
\hat{\rho}_i(\lambda_i,z_{1};\tau) 
:=  \dot{t}_i(\tau)
	\left[
    	\lambda \log \{\lambda / \hat{\theta}_{\Pi^{*},1}(z_{1};\tau)\} 
		- \lambda_i 
        + \hat{\theta}_{\Pi^{*},1}(z_{1};\tau)
	\right].
\end{align*}

Observe that 
we have, for $\tau\in[0,1]$,
\begin{align*}
    \hat{\theta}_{\Pi^{*},i}(z_{i};\tau) &= (\eta_{n} /t_{i}(\tau)^{2}) \big{/} (1+\eta_{n}/t_{i}(\tau)^{1}), &z_{i}=0,\\
    \hat{\theta}_{\Pi^{*},i}(z_{i};\tau) &= (z_{i}+1)/ t_{i}(\tau), &z_{i}\geq 1.
\end{align*}
This gives
\begin{align*}
    \hat{\rho}_{i}(\lambda_{i},z_{i};\tau) &\leq 
    \dot{t}_i(\tau) \{\lambda_i \log \eta_n^{-1} 
	+  \lambda_i \log \lambda_i 
    -  \lambda_i 
    +  \lambda_i \log C_{1} 
    +  \eta_{n}/t_{i}(\tau)^{2} \}, 
    & z_{1}=0, \\
    \hat{\rho}_i(\lambda_i,z_{1};\tau) 
&\leq \dot{t}_i(\tau) \{\lambda_i \log \{ \dot{t}_{i}(\tau)\lambda_i /(z_{i}+1)\}
	- \lambda_i
    + (z_{i}+1) / t_i(\tau) \},
    & z_{i}\geq 1,
\end{align*}
where $C_{1}=\sup_{i}(r_{i}+2)^{2}$.

By the same way as in Subsection \ref{subsection: Proof of theorem exact minimaxity within sparse Poisson models}, we have,
for sufficiently large $n\in\mathbb{N}$ and for $\lambda_i > 0$,
\begin{equation}
\begin{split}
\rho_i(\lambda_i) 
\leq \int_0^1 
	\Bigg[
    	\mathrm{e}^{-t_i(\tau) \lambda_i} \dot{t}_i(\tau) \lambda_i \log \eta_n^{-1}
		&+ \mathrm{e}^{-t_i(\tau)\lambda_i} 
        	\dot{t}_i(\tau)\lambda_i
       		\log(\lambda_i C_{1}) \\
&
        + \frac{\dot{t}_i(\tau)}{t_i(\tau)}(1 - \mathrm{e}^{-t_i(\tau)\lambda_i})
 	\Bigg]
 	d\tau  + \eta_nC_2,
\end{split}
 \label{upper}
\end{equation}
where $C_{2}:=\sup_{i}\{1/r_{i}-1/(r_{i}+1)\}$.
Inequality (\ref{upper}) yields
\[\sup_{\lambda > 0}\rho_i(\lambda) 
\leq (\mathcal{C}_{i} + o(1)) \log \eta_n^{-1}\] 
and a similar procedure yields the inequality
$\rho_i(0) = O(\eta_n)$.

Finally, we derive an upper bound on $\overline{\mathcal{R}}(\Theta[s_{n}])$ by combining asymptotic inequalities
$\sup_{\lambda > 0}\rho_i(\lambda) 
\leq (\mathcal{C}_{i} + o(1)) \log \eta_n^{-1}$
and
$\rho_i(0) = O(\eta_n)$ for $i=1,\ldots,n$.
Let subscripts $[1],\ldots,[n]$ be the permutation of $1,\ldots,n$ that satisfies $\mathcal{C}_{[1]}\geq \ldots \geq \mathcal{C}_{[n]}$. 
The minimax risk $\overline{\mathcal{R}}(\Theta[s_n])$ is bounded as 
\begin{align}
\overline{\mathcal{R}}(\Theta[s_n]) 
&\leq \sup_{\theta\in\Theta[s_n]} R(\theta, q_{\Pi^{*}}) 
= \sum_{i = 1}^{s_n} \sup_{\lambda_>0}\rho_{[i]}(\lambda) 
	+ \sum_{i = s_n + 1}^n \rho_{[i]}(0).
\label{minimax risk with IM}
\end{align}
Together with Condition \ref{condition: condition on ris},
the asymptotic inequalities
\[
\sup_{\lambda_i > 0}\rho_i(\lambda_i)\leq (\mathcal{C}_{i}+o(1))\log \eta_{n}^{-1}
\text{ and }
\rho_i(0)=O(\eta_{n}) \ (i=1,\ldots,n)
\]
give
\begin{align*}
\overline{\mathcal{R}}(\Theta[s_n]) 
&\leq \sum_{i = 1}^{s_n} \{\mathcal{C}_{[i]} + o(1)\} \log \eta_n^{-1} 
	+ (n - s_n) O(\eta_n) \\ 
&\sim \frac{s_n}{n} \sum_{i = 1}^{n} \mathcal{C}_{i} \{\log \eta_n^{-1}\} (1 + o(1)) 
\sim \bar{\mathcal{C}} s_n \log \eta_n^{-1},
\end{align*}
from which we obtain the desired upper bound on $\overline{\mathcal{R}}(\Theta[s_n])$. \qed

\subsection{Proof of Proposition \ref{proposition: adaptive in random design}}
\label{subsection: proof of adaptive in random design}

Almost the same lemma as Lemma \ref{lemma: properties of hats_n}
holds for MCAR settings
by
replacing $\mathrm{e}^{-r\theta_{[j]}}$ in (\ref{proof bias}) in \cite{YKK} and (\ref{proof variance}) in \cite{YKK} with $\mathrm{e}^{-r_{[j]}\theta_{[j]}}$.
Thus following the same line as in Subsection \ref{subsection: proof of theorem adaptive} completes the proof.
\qed

\section{Supplemental experiments}
\label{Appendix: supplemental experiments}

This section presents supplemental experiments.

\subsection{Quasi-sparsity}

This subsection provides simulation studies for quasi-sparsity.
Parameter $\theta$ and observations $X$ and $Y$ are drawn from
\begin{align*}
\theta_{i} &\sim \nu_{i} e_{S,i} + {\xi}_{i} e_{S^{\mathrm{c}}, i} \ (i=1,\dots,n),\\
 X \mid \theta &\sim \otimes_{i=1}^{n}\mathrm{Po}(r\theta_{i}),
 Y \mid \theta \sim \otimes_{i=1}^{n}\mathrm{Po}(\theta_{i}),
 \text{ and }
 X \indep Y \mid \theta,
\end{align*}
respectively, where 
\begin{itemize}
    \item $\nu_{1},\ldots,\nu_{n}$ are independently drawn from the gamma distribution 
with a shape parameter $10$ and a scale parameter $1$;
\item $\xi_1,\ldots, \xi_n$ are independently and uniformly drawn from $[0,10^{-2}]$;
\item $S$ is drawn from the uniform distribution on all subsets having exactly $s$ and $S^{\mathrm{c}}$ is its complement;
\item $\nu_{1},\ldots,\nu_{n}$ and $S$ are independent.
\end{itemize}
Here for a subset $J\subset \{1,\ldots,n\}$, $e_{J}$ indicates the vector
of which the $i$-th component is 1 if $i\in J$ and 0 if otherwise.
We examine two cases $(n, s ,r) = (200, 5, 20)$ and $(n,s,r)=(200,20,20)$,
and generate 500 current observations $X$'s
and 500 future observations $Y$'s.

\begin{table}[h]
\caption{Comparison of predictive densities with $(n,s,r)=(200,5,20)$ and with quasi-sparsity:
for each result, the averaged value is followed by the corresponding standard deviation. Underlines indicate the best performance.
The same abbreviations as in Table \ref{table: simulation under exact sparsity  in homogeneity, (n,s,r)=(200,5,1)} are used.
}
\label{table: simulation under quasi sparsity 1 in homogeneity}
\centering
\begin{tabular}{|c||c|c|c|c|c|c|} \hline
      &$\Pi[L^{*}\hat{\eta}_{n},0.1]$&
      $\Pi[L^{*}\hat{\eta}_{n},1.0]$&
      GH & K04 & $\ell_{1}$ ($\lambda=0.1$)
      \\ \hline\hline
      $\ell_{1}$ distance & 
      13.8 (4.0) &
      \underline{13.6} (4.0)&
      18.1 (1.7) 
      & 63.3 (8.3) & 17.6 (4.9) 
      \\
      PLL & 
      -19.0 (5.3) &
      \underline{-18.8} (5.3)&
      -20.8 (4.3) &
      -43.8 (4.9) & -Inf
      \\
90\%CP (\%) &
      90.7 (2.8) &
      \underline{90.1} (2.9) &
      45.6 (18.9)
      & 43.5 (15.0) & 85.5 (4.2) 
      \\
      \hline
\end{tabular}
\end{table}
\begin{table}[h]
\caption{Comparison of predictive densities with $(n,s,r)=(200,20,20)$ and with quasi-sparsity:
for each result, the averaged value is followed by the corresponding standard deviation. Underlines indicate the best performance.
}
\label{table: simulation under quasi sparsity 2 in homogeneity}
\centering
\begin{tabular}{|c||c|c|c|c|c|c|} \hline
      & $\Pi[L^{*}\hat{\eta}_{n},0.1]$ &
      $\Pi[L^{*}\hat{\eta}_{n},1.0]$ &
      GH & K04 & $\ell_{1}$ ($\lambda=0.1$)
      \\ \hline\hline
      $\ell_{1}$ distance &
      49.0 (8.5) &
      \underline{48.8} (8.6)&
      50.3 (3.0) 
      & 206 (15) & 57.7 (9.1)
      \\
      PLL & 
      -52.3 (5.6) &
      \underline{-52.2} (5.7)&
      -56.7 (4.5) &
      -121 (8.8) & -Inf
      \\
      90\%CP (\%)& 
      \underline{89.8} (2.9) &
      89.5 (3.1)&
      50.3 (3.1)
      & 0.0 (0.0) & 82.8 (4.3) 
      \\
      \hline
\end{tabular}
\end{table}

Tables \ref{table: simulation under quasi sparsity 1 in homogeneity} and \ref{table: simulation under quasi sparsity 2 in homogeneity} show that
the performance of the proposed predictive densities does not depend on whether a parameter is exact sparse or quasi-sparse.

\subsection{Effect of $s$}

This subsection provides  simulation studies highlighting the effect of $s$. 
The set-up except for $s$ is the same as that in Subsection \ref{subsection: simulation studies}.

\begin{table}[h!]
\caption{Comparison of predictive densities with $(n,s,r)=(200,50,20)$ with exact sparsity:
for each result, the averaged value is followed by the corresponding standard deviation. Underlines indicate the best performance.
}
\label{table: simulation under exact sparsity  in homogeneity, s=50}

\centering
\begin{tabular}{|c||c|c|c|c|c|c|} \hline
      & $\Pi[L^{*}\hat{\eta}_{n},0.1]$ &
      $\Pi[L^{*}\hat{\eta}_{n},1.0]$ &
      GH & K04 & $\ell_{1}$ ($\lambda=0.1$)
      \\ \hline\hline
      $\ell_{1}$ distance &  
      \underline{106.0} (11.5)&
      106.4 ( 11.6)&
      136.3 (4.76) 
      & 117.1 (12.4) & 118.0 (13.0)
      \\
      PLL & 
      \underline{-110.2} (4.5)&
      110.4 (4.5)&
      -115.2 (4.6) &
      -117.8 (4.3) & -Inf
      \\
      90\%CP (\%)&
      89.0 (2.8)&
      \underline{89.8} (2.7)&
      92.1 (1.4)
      & 95.4 (2.4) & 79.4 (5.4) 
      \\
      \hline
\end{tabular}
\end{table}

\begin{table}[h!]
\caption{Comparison of predictive densities with $(n,s,r)=(200,100,20)$ and with exact sparsity:
for each result, the averaged value is followed by the corresponding standard deviation. Underlines indicate the best performance.
}
\label{table: simulation under exact sparsity  in homogeneity, s=100}

\centering
\begin{tabular}{|c||c|c|c|c|c|c|} \hline
      & $\Pi[L^{*}\hat{\eta}_{n},0.1]$ &
      $\Pi[L^{*}\hat{\eta}_{n},1.0]$ &
      GH & K04 & $\ell_{1}$ ($\lambda=0.1$)
      \\ \hline\hline
      $\ell_{1}$ distance &
      208.0 (18.6)&
      208.4 (18.7)&
      \underline{190.3} (6.0) 
      & 209.6 (12.4) & 221.3 (17.8)
      \\
      PLL & 
      -204.4 (6.7)&
      -204.8 (6.7)&
      \underline{-191.9} (1.9) &
      -207.4 (6.1) & -Inf
      \\
      90\%CP (\%)&
      88.9 (2.8)&
      \underline{89.7} (2.8)&
      100.0 (0.0)
      & 92.9 (2.2) & 72.1 (5.6) 
      \\
      \hline
\end{tabular}
\end{table}

Tables \ref{table: simulation under exact sparsity  in homogeneity, s=50} and \ref{table: simulation under exact sparsity  in homogeneity, s=100} find
that
the Bayes predictive density based on a Gauss hypergeometric prior
works better in the predictive likelihood sense as the sparsity $s_{n}$ is relatively large.

\subsubsection{Sparse Poisson model with Missing-Completely-At-Random settings}
\label{subsubsection: simulation study with sample-size heterogeneity}
This subsection presents simulation studies for MCAR settings.
Consider a sparse Poisson model with MCAR.
Parameter $\theta$ and observations $X$ and $Y$ are 
drawn in the following way:
\begin{align*}
\theta_{i} &\sim \nu_{i} e_{S,i} \ (i=1,\ldots, n), \\
X\mid \theta &\sim \otimes_{i = 1}^{n} \mathrm{Po}(r_i\theta_i), \
Y\mid \theta \sim \otimes_{i = 1}^n \mathrm{Po}(\theta_i), \
\text{and }
X \indep Y \mid \theta,
\end{align*}
where 
$\nu_{1},\ldots,\nu_{n}$ and $S$ follow the same distributions as those in the previous subsection.
In addition, $r_i-1$ ($i = 1,\ldots, n$) are independently drawn from the binomial distribution $\mathrm{Bi}(m, p)$ with the parameters $(m, p)$ either $(1, 0.9)$ or $(10, 0.9)$.
We set $(n,s) = (200, 5)$,
and
generate 500 current observations $X$'s and 500 future observations $Y$'s.

We compare the following four predictive densities:
\begin{itemize}
    \item The proposed predictive density based on $\Pi[\overline{L}\hat{\eta}_{n},\kappa]$ with $\overline{L}$ in Remark \ref{remark: optimal choice};
    \item The Bayes predictive density based on the shrinkage prior proposed in \cite{Komaki(2015)};
    \item The Bayes predictive density based on the Gauss hypergeometric prior proposed in \cite{DattaandDunson(2016)};
    \item The plug-in predictive density based on an $\ell_1$-penalized estimator.
\end{itemize}
An estimator $\hat{s}_{n}$ is determined in the same manner as in the previous subsection.
In \cite{Komaki(2015)},
the second predictive density is shown to dominate the Bayes predictive density based on the Jeffreys prior in the case where the numbers of observations are coordinate-wise different. 
In simulation studies, each hyper parameter $\beta_i$ of the second predictive density is fixed to be 1.

In comparing the performance,
we use the weighted $\ell_{1}$ distance between the mean of the predictive density and a future observation.
Here the the weighted $\ell_{1}$ distance between $u,v\in \mathbb{R}^n$ is given by $ \sum_{i=1}^{n}r_{i}|u_{i}-v_{i}| / (\sum_{i=1}^{n} r_{i}/n)$. 
For the construction of prediction sets, we also use this weighted $\ell_{1}$ distance.

\begin{table}[h]
\caption{
Comparison of predictive densities with MCAR and with $(n,s,m, p)=(200,5,1, 0.9)$:
the W-$\ell_1$ distance, PLL, and $90\%$CP represent the weighted mean $\ell_1$ distance, the predictive log likelihood, and the empirical coverage probability based on a $90\%$-prediction set, respectively.
For each result, the averaged value is followed by the corresponding standard deviation. 
Underlines indicate the best performance.
}
\label{table: simulation under exact sparsity in hetarogeneity, (n,s,m, p)=(200,5,1,0.9)}

\centering
\begin{tabular}{|c||c|c|c|c|c|c|} \hline
      &$\Pi[\overline{L}\hat{\eta}_{n},0.1]$
      &$\Pi[\overline{L}\hat{\eta}_{n},1.0]$&
      GH &K15& $\ell_{1}$ ($\lambda=0.1$)
      \\ \hline\hline
      $\text{W-}\ell_{1}$ distance &
      \underline{17.1} (5.6)& 26.8 (7.3) &
      48.3 (5.6) & 25.6 (8.7) & \underline{17.1} (5.8)
       \\
      PLL & 
      \underline{-14.7} (2.2)&
      -16.8 (2.7)&
      -42.3 (2.0) &-18.2 (3.4)& -Inf
      \\
      90\%CP (\%) &
      \underline{91.3} (0.1)&
      71.8 (0.2) &
      100 (0.0) &61.7 (20.9)& 68.6 (13.6)
      \\
      \hline
\end{tabular}
\end{table}

 \begin{table}[h]
\caption{Comparison of predictive densities with MCAR and with $(n,s,m, p)=(200,5,10, 0.9)$:
the W-$\ell_1$ distance, PLL, and $90\%$CP represent the weighted mean $\ell_1$ distance, the predictive log likelihood, and the empirical coverage probability based on a $90\%$-prediction set, respectively.
For each result, the averaged value is followed by the corresponding standard deviation. 
Underlines indicate the best performance.
}
\label{table: simulation under exact sparsity in heterogeneity, (n,s,m, p)=(200,5,10,0.9)}

\centering
\begin{tabular}{|c||c|c|c|c|c|c|} \hline
      &$\Pi[\overline{L}\hat{\eta}_{n},0.1]$
      &$\Pi[\overline{L}\hat{\eta}_{n},1.0]$&
      GH & K15  &　$\ell_{1}$ ($\lambda=0.1$)
      \\ \hline\hline
      $\text{W-}\ell_{1}$ distance &
      \underline{12.3} (4.2)&
      12.4 (4.2)&
      17.3 (1.8) &15.2 (4.2)& 13.0 (4.4)
       \\
      PLL & 
      \underline{-12.2} (1.7)&
      \underline{-12.2} (1.7)&
      -19.3 (1.7)& -14.0 (1.9)& -Inf
      \\
      90\%CP (\%) &
      88.3 (4.2)&
      \underline{88.5} (3.9)&
      51.7 (8.6) &78.3 (7.3)& 84.8 (4.6)
      \\
      \hline
\end{tabular}
\end{table}

Tables \ref{table: simulation under exact sparsity in hetarogeneity, (n,s,m, p)=(200,5,1,0.9)} and \ref{table: simulation under exact sparsity in heterogeneity, (n,s,m, p)=(200,5,10,0.9)}
show the results.
In addition to the abbreviations used in the previous subsection, the abbreviation W-$\ell_{1}$ distance represents a mean weighted $\ell_{1}$ distance.
We see that the weighted $\ell_1$ distances by the proposed predictive density based on $\Pi[\overline{L}\hat{\eta}_{n},0.1]$
and the plug-in predictive density based on the $\ell_1$-penalized estimator
are in the smallest level of all predictive densities compared here.
As is the case without MCAR, 
the proposed predictive density with $\Pi[\overline{L}\hat{\eta}_{n},0.1]$ broadly has the coverage probability that is relatively close to the nominal level, 
whereas the other predictive densities (with one exception) do not.

\subsection{Comparison to the extremely heavy tailed prior}

This subsection provides an additional comparison to 
the predictive density based on
the extremely heavy tailed prior distribution in \cite{Hamuraetal}.
We denote by EH the prior in \cite{Hamuraetal} and use the same settings as in Tables 1 and 2 of the main manuscript.
In summary, we found EH is a very good competitor and a thorough comparison would offer an interesting insight.

\begin{table}[h]
\caption{Comparison under the same setting as in Table 1 in the main manuscript. For each result, the averaged value is followed by the corresponding standard deviation. 
Underlines indicate the best performance.}
\begin{center}
\begin{tabular}{|c||c|c|c|c|c|c|}
\hline
 & $\Pi[L^{*}\hat{\eta}_{n},0.1]$ & $\Pi[L^{*}\hat{\eta}_{n},1.0]$ & GH & K04 & $\ell_{1}$ &EH \\
 \hline \hline
 $\ell_{1}$ distance & \underline{18.8} (5.8) & 21.9 (6.8) & 104 (4.9) & 96.5 (8.1) & 22.1 (7.8) & 21.3 (6.0) \\
 PLL & \underline{-15.4} (1.8) & -16.1 (1.6) & -66.3 (3.3) & -86.2 (8.8) & -Inf & -16.6 (1.2) \\
 90$\%$CP ($\%$) & 92.6 (0.1) & 95.8 (0.1) & \underline{92.0} (1.5) & 40.5 (24.4) & 49.4 (21.6) & 92.7 (0.1) \\\hline
 \end{tabular}
\end{center}
\end{table}

\begin{table}[h!]
\caption{Comparison under the same setting as in Table 2 in \cite{YKK}. For each result, the averaged value is followed by the corresponding standard deviation. 
Underlines indicate the best performance.}
\begin{center}
\begin{tabular}{|c||c|c|c|c|c|c|}
\hline
 & $\Pi[L^{*}\hat{\eta}_{n},0.1]$ & $\Pi[L^{*}\hat{\eta}_{n},1.0]$ & GH & K04 & $\ell_{1}$ &EH \\
 \hline \hline
 $\ell_{1}$ distance & \underline{14.0} (4.9) & 14.5 (4.5) & 15.7 (1.7) & 22.5 (5.2) & 14.1 (4.5) & 17.6 (1.6) \\
 PLL & -13.3 (1.6) & -13.5 (1.5) & -15.6 (1.5) & -21.6 (2.2) & -Inf & \underline{-12.23} (1.4) \\
 90$\%$CP ($\%$) & \underline{90.0} (0.0) & 89.4 (0.0) & 97.6 (0.7) & 97.5 (1.4) & 86.3 (3.9) & 91.8 (0.0) \\\hline
 \end{tabular}
\end{center}
\end{table}

\vspace{5mm}

\section{Discussions for Subsection \ref{subsection: MCAR}}
\label{subsection: Discussion: MCAR}

In this subsection,
we compare Proposition \ref{proposition: minimax risk in random design}
with Theorem \ref{theorem: exact minimaxity within sparse Poisson models}:
\begin{itemize}
\item[(A)] $G$ is the Gamma distribution with shape parameter $r/l$ and scale parameter $l$ for $0<l\leq 1$ and $r\geq 2$;
\item[(B)] $G$ is the distribution of $1+S$, where $S$ follows the binomial distribution with trial number $N$ and success probability $p$.
\end{itemize}
In Setting (A), the mean remains $r$ for any $l$ and the variance is $rl$,
which means that as $l\to0$, $G$ is weakly convergent to the Dirac measure $\delta_{r}$ centered at $r$ corresponding to a non-MCAR case.
In Setting (B), $G$ is weakly convergent to the Dirac measure $\delta_{1}$ centered at $1$ as $p\to 0$,
and $G$ is weakly convergent to the Dirac measure $\delta_{1+N}$ as $p\to1$.

\begin{figure}[ht]
\begin{minipage}{0.49\hsize}
\hspace{-0.4cm}
\includegraphics[width=6cm]{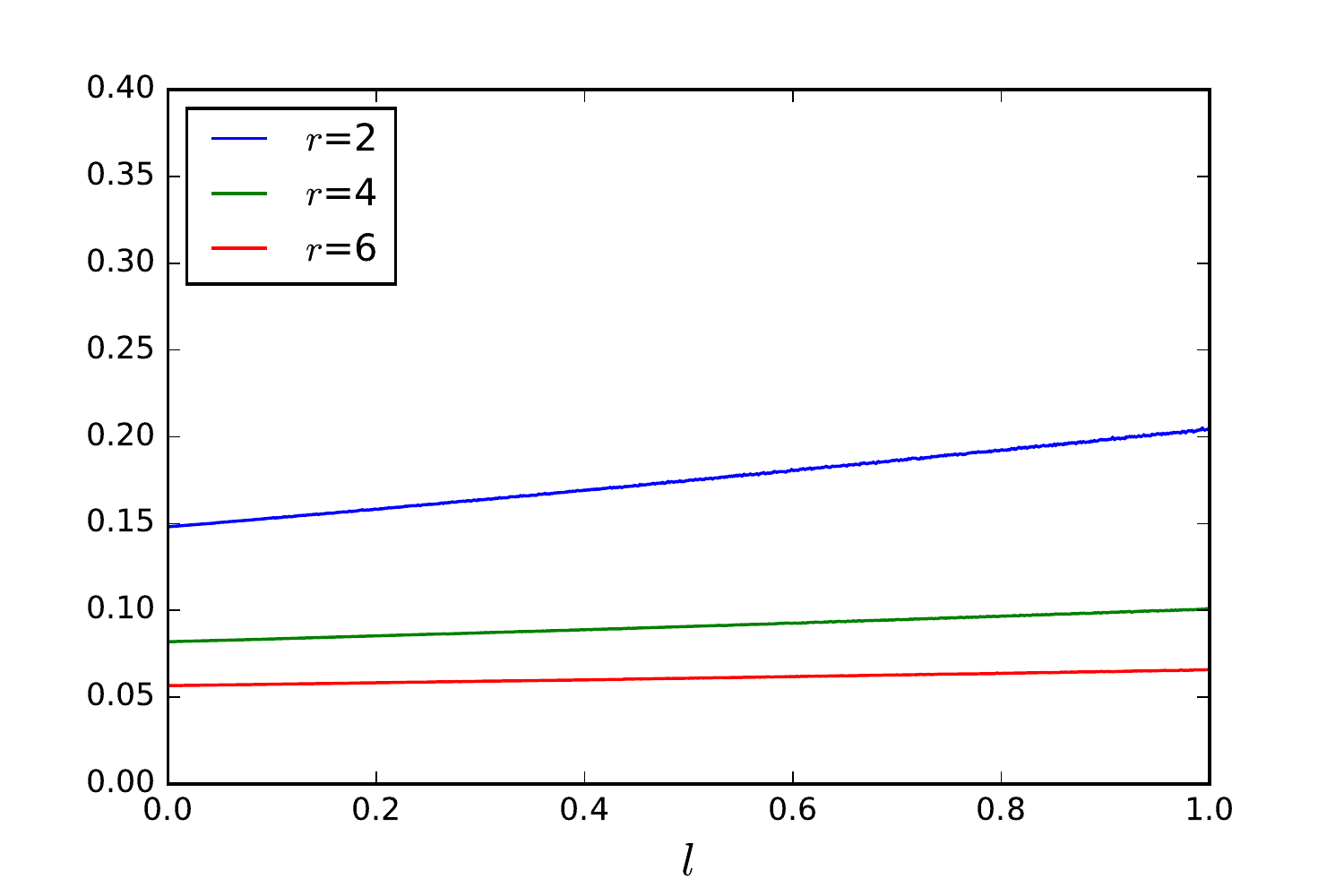}
\caption{Comparison of $\Ep_{G}[\overline{\mathcal{C}}]$ in \,\,\,\,\protect\linebreak Setting (A): the horizontal line indicates \,\,\,\,\protect\linebreak $l$ of $G$.}
\label{figure: comparison of homogeneous and heterogeneous (A)}
\end{minipage}
\begin{minipage}{0.49\hsize}
\hspace{-0.4cm}
\includegraphics[width=6cm]{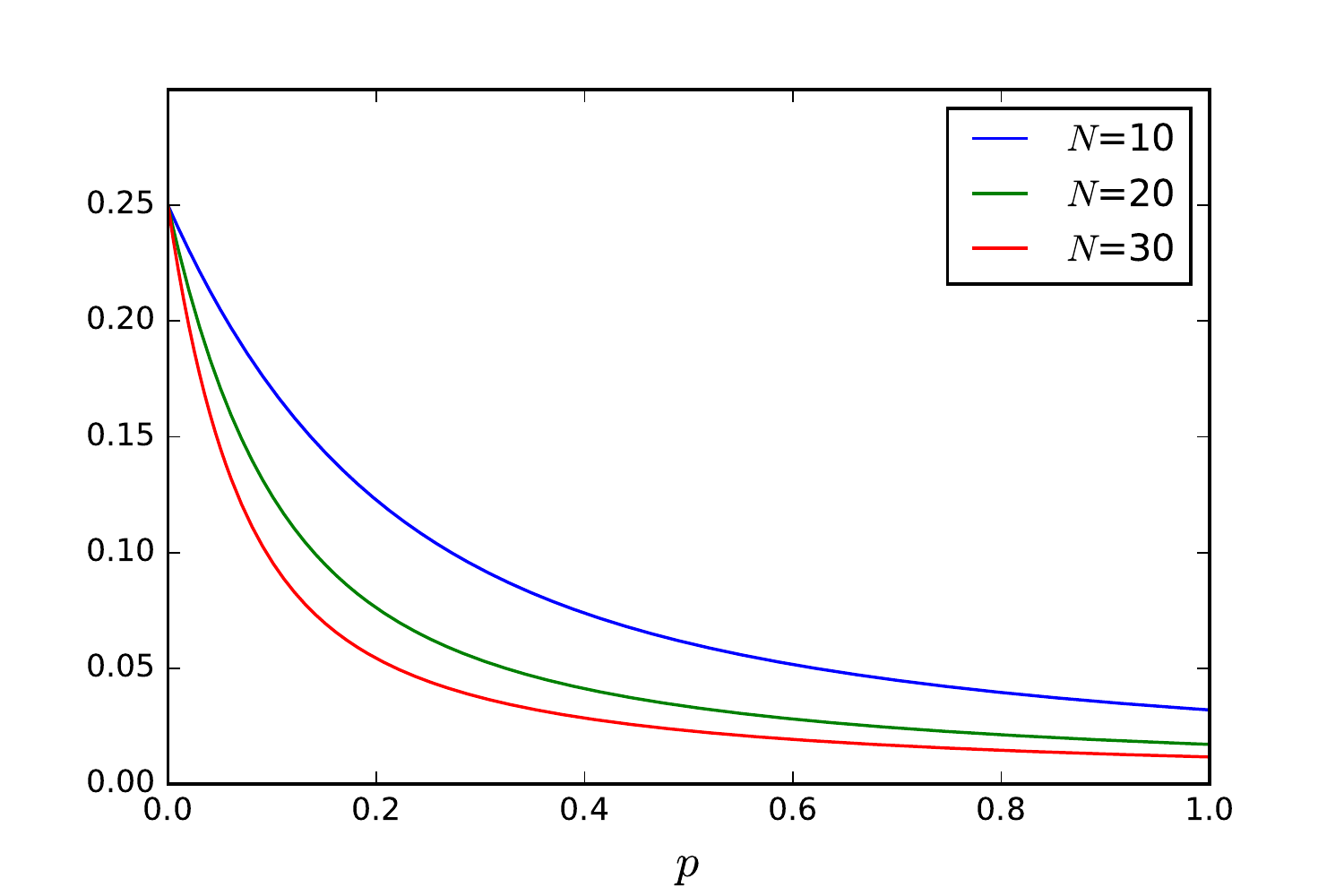}
\caption{Comparison of $\Ep_{G}[\overline{\mathcal{C}}]$ in \,\,\,\,\protect\linebreak Setting  (B): the horizontal line indicates \,\,\,\,\protect\linebreak $p$ of $G$.}
\label{figure: comparison of homogeneous and heterogeneous (B)}
\end{minipage}
\end{figure}

In Figure \ref{figure: comparison of homogeneous and heterogeneous (A)},
the vertical line indicates exact constants and the horizontal line indicates values of $l$. The blue line denotes the case with $r=2$, the green line denotes the case with $r=4$, and the red line denotes the case with $r=6$.
In Figure \ref{figure: comparison of homogeneous and heterogeneous (B)},
the vertical line indicates exact constants and the horizontal line indicates values of $p$.
The blue line denotes the case with $N=10$, the green line denotes the case with $N=20$, and the red line denotes the case with $N=30$.

Figures \ref{figure: comparison of homogeneous and heterogeneous (A)} and \ref{figure: comparison of homogeneous and heterogeneous (B)} show the exact constants of predictive minimax risks in Settings (A) and (B), respectively.
According to Figure \ref{figure: comparison of homogeneous and heterogeneous (A)},
the constant gets larger as the variance of $G$ increases.
The constant in a MCAR case approaches to that in a non-MCAR case in the limit $l\to 0$.
According to Figure \ref{figure: comparison of homogeneous and heterogeneous (B)},
the constant gets smaller as the missing probability $1-p$ gets smaller.
Further, the numerical result in Setting (B) is consistent to 
the results in \cite{Efromovich(2011),Efromovich(2013)} for the literature of nonparametric regression in the presence of missing observations.
Theorems 1 and 2 in \cite{Efromovich(2011)} provide
tight lower and upper bounds on mean integrated squared errors (MISE) in nonparametric regression with predictors missing at random.
Those theorems also provide an exact asymptotically minimax estimator for MCAR cases
and
show that if predictors are MCAR, the minimax MISE gets smaller as the missing probability approaches zero.

\section{Application to Rare mutation rates in an oncogene}
\label{rare mutation rates}

We present an application of the proposed methods to exome sequencing data from a huge database called the Exome Aggregation Consortium (ExAC). 
ExAC reports the total numbers of mutant alleles in each genetic position along the whole exome, gathered from 60706 unrelated individuals.
We focus on rare allele mutations 
in a gene PIK3CA; 
For the importance of analysing rare allele mutations and the choice of the gene, see \cite{DattaandDunson(2016)}.
We also follow the pre-process of the data described in \cite{DattaandDunson(2016)}.

We apply the sparse Poisson model to the numbers of rare mutant alleles as follows. 
We denote by $X_i$ $(i = 1\ldots,551)$ the number of rare mutant alleles in the $i$-th position. 
We assume that $X_i$ is distributed according to $\mathrm{Po}(r_i\theta_i)$.
Here $r_i$ ($i=1,\ldots,551$) is the double number of individuals whose $i$-th location are sequenced, and $\theta_i$ is the frequency rate common to individuals in the $i$-th position.
The doubling is necessary because each individual has two copies of each allele.
Since numerous fragments of DNA sequences are sampled and read at random,
$r_i$'s are different; 
hence the data is regarded as having an MCAR structure.

Our goal is to predict the behavior of rare allele mutations under the assumption that 
all individuals could be sequenced at all positions
and
the sequencing depth is uniform across the gene;
the target mutation counts $Y_i$'s are assumed to be distributed according to $\otimes_{i = 1}^n\mathrm{Po}(\overline{r}\theta_i)$
with $\overline{r}=121412=2\times 60706$.
We compare the proposed prior $\Pi[\overline{L}\hat{\eta}_{n},\kappa]$ with the two existing priors, that is, 
the Gauss hypergeometric prior in \cite{DattaandDunson(2016)} 
and the shrinkage prior in \cite{Komaki(2004)}. 
For the proposed prior,
$\kappa$ is set to 0.1 and
the estimate $\hat{s}_{n}$ to be plugged in is determined by the k-means clustering with $k=2$ (the resulting value of $\hat{s}_{n}$ is 17). 
In this study, we give the qualitative comparison using samples from predictive densities.

\begin{figure}
\centering\includegraphics[width=12cm]{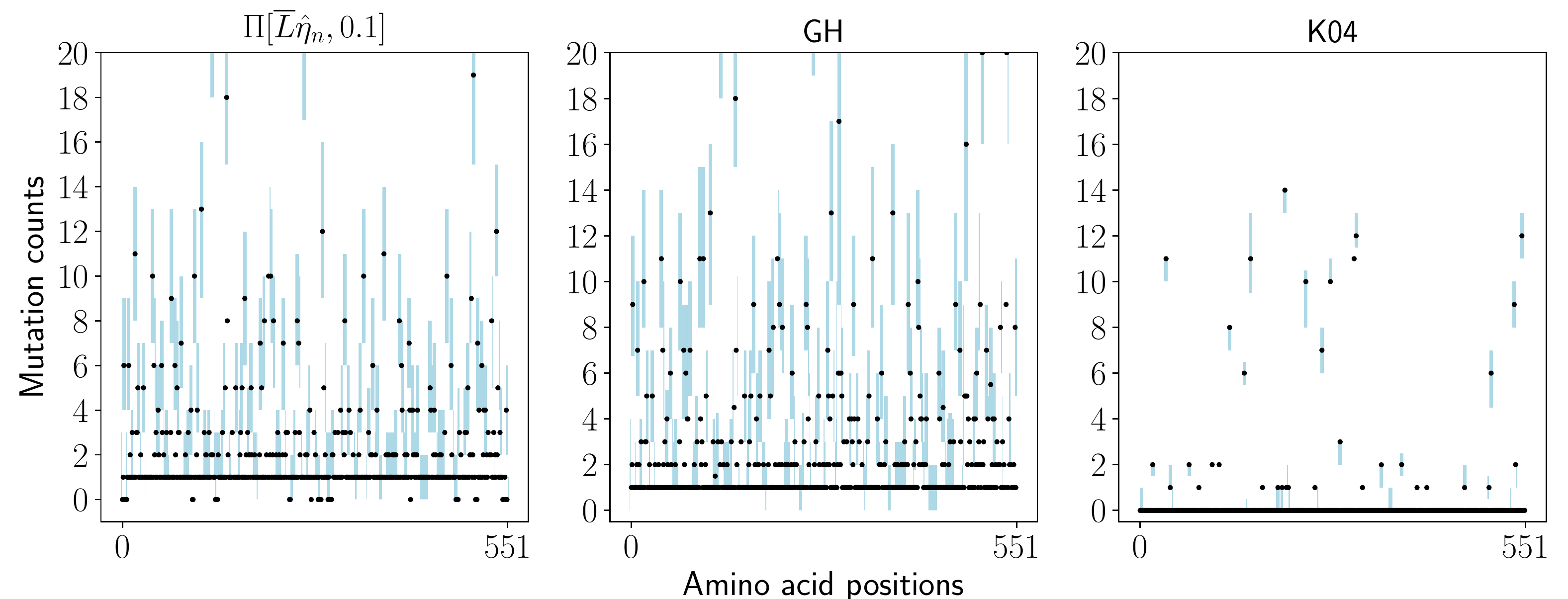}
\caption{Comparison of Bayes predictive densities in terms of marginal 50\%-prediction intervals. For each $i$-th position, the top and the bottom of the blue lines indicate 25\% and 75\% percentiles, respectively. The black points show the medians of marginal densities.}
\label{figure: exac result}
\end{figure}

Figure \ref{figure: exac result} shows marginal prediction intervals at a nominal level of 50\%. 
The prediction intervals of
the proposed predictive density $q_{\Pi[\overline{L}\hat{\eta}_{n},0.1]}$ and of the Bayes predictive density based on GH show apparently similar behaviors:
they shrink at positions whose counts are low,
and they remain the scales at positions whose counts are large.
The prediction intervals by K04 degenerate at most of the locations.
In lower regimes of counts,
there exists a dissimilarity between $q_{\Pi[\overline{L}\hat{\eta}_{n},0.1]}$ and GH.
Most of the medians of the intervals constructed by the predictive density based on GH are away from $0$; 
while those constructed by $q_{\Pi[\overline{L}\hat{\eta}_{n},0.1]}$
sometimes reach zero.
Owing to the coordinate-wise independence,
this dissimilarity is considered to reflect on
that the proposed predictive density has more flexibility than GH.

%
%
%

\end{document}